\documentclass[smallextended]{article}
\usepackage[authoryear]{natbib}
\usepackage{bibunits}
\defaultbibliography{references}
\defaultbibliographystyle{abbrvnat}

\usepackage[T1]{fontenc}
\usepackage{comment}
\usepackage{amssymb}
\usepackage{graphicx}
\usepackage{float}
\usepackage{epsfig}
\usepackage{here}
\usepackage{psfrag}
\usepackage{latexsym}
\usepackage{epic}
\usepackage{eepic}
\usepackage{amsmath}
\usepackage{enumerate}
\usepackage{appendix}
\usepackage{color}
\usepackage{url}

\usepackage{geometry}
\usepackage{mathabx}

\usepackage[affil-it]{authblk}

\usepackage{hyperref}
\usepackage{soul}
{\providecommand{\noopsort}[1]{}}

\newcommand{\bea}{\begin{eqnarray}}
\newcommand{\eea}{\end{eqnarray}}
\newcommand{\be}{\begin{equation}}
\newcommand{\ee}{\end{equation}}
\newcommand{\beann}{\begin{eqnarray*}}
\newcommand{\eeann}{\end{eqnarray*}}
\newcommand{\bal}{\begin{align}}
\newcommand{\eal}{\end{align}}
\newcommand{\balnn}{\begin{align*}}
\newcommand{\ealnn}{\end{align*}}

\newcommand{\nn}{\nonumber}

\def\R{{\mathbb R}}
\def\ba{\begin{array}}
\def\ea{\end{array}}
\def\bd{\begin{displaymath}}
\def\ed{\end{displaymath}}

\def\P{{\mathbb P}}

\def\1{{\mathbf 1}}

\def\A{{\mathcal A}}
\def\U{{\mathcal U}}
\def\E{{\mathbb E}}
\def\F{{\mathcal F}}

\def\M{{\mathcal M}}
\def\N{{\mathbb N}}
\def\G{{\mathbf G}}

\def\Y{{\mathcal Y}}
\def\X{{\mathcal X}}

\def\G{{\mathbb G}}
\def\T{{\mathcal T}}

\def\NN{{\mathcal N}}

\def\Z{{\mathbb Z}}

\def\B{{\mathcal B}}

\def\de{{\mathrm{d}}}

\DeclareMathOperator{\e}{e}
\def\B{{\mathcal B}}

\DeclareMathOperator{\supp}{supp}

\usepackage{amsthm}
\newtheorem{thm}{Theorem}
\newtheorem{cor}{Corollary}

\newtheorem{lemma}{Lemma}

\newtheorem{rem}{Remark}
\newtheorem{definition}{Definition}

\title{Functional marked point processes -- A natural structure to unify spatio-temporal frameworks and to analyse dependent functional data
}
\date{}

\author{}

\begin{document}
\newcommand{\edt}[1]{{\vbox{ \hbox{#1} \vskip-0.3em \hrule}}}

\maketitle

\begin{center}
Cronie, O.\footnote{e-mail: ottmar.cronie@umu.se (corresponding author)}\footnote{Department of Mathematics and Mathematical Statistics, Ume{\aa} University, Sweden}, Ghorbani, M.$^{\dagger}$, %\footnote{e-mail: mohammad.ghorbani@umu.se (corresponding author)} 
Mateu, J.\footnote{
Department of Mathematics, 
Universitat Jaume I, 
Campus Riu Sec, 
12071 Castell\'on, 
Spain
} 
and 
Yu, J.$^{\dagger}$ 
% %\footnote{
% Department of Mathematics and Mathematical Statistics, Ume{\aa} University, Sweden
% },

\end{center}

\begin{abstract}

This paper treats functional marked point processes (FMPPs), which are defined as marked point processes where the marks are random elements in some (Polish) function space. Such marks may represent e.g.\ spatial paths or functions of time. 
To be able to consider e.g.\ multivariate FMPPs, we also attach an additional, Euclidean, mark to each point. 
We indicate how FMPPs quite naturally connect the point process framework with both the functional data analysis framework and the geostatistical framework. 
We further show that various existing models fit well into the FMPP framework. 
In addition, we introduce a new family of summary statistics, weighted marked reduced moment measures, together with their non-parametric estimators, in order to study features of the functional marks. 
We further show how they generalise other summary statistics and we finally apply these tools to analyse population structures, such as demographic evolution and sex ratio over time, in Spanish provinces.

\end{abstract}

\noindent {\bf Key words}: 
%Boolean model, 
%C\`adl\`ag stochastic process, 
%Conditional intensity, 
%Discrete sampling, 
%Constructed functional marks,
%Animal movement,
Correlation functional, 
Functional data analysis, 
%Functional marked point process,
%Functional mark distribution,
%Geostatistical prediction with locations error,
Intensity functional,
Marked point process, 
Non-parametric estimation, 
Palm distribution, 
%Marked inhomogeneous $K$-functional, 
Population growth,
%LISTA function, 
%Marked reduced Palm measure, 
%Markov process, 
%Maximum (pseudo)likelihood, 
%Papangelou conditional intensity,
%Product density,
%Random field, 
%Spatio-temporal functional marked point process, 
Spatio-temporal geostatistical marking,
Weighted marked reduced moment measure.
% %Spatio-temporal intensity dependent marks, 
% Wiener measure.

\begin{bibunit}

\section{Introduction}\label{SectionIntroduction}

Many types of functional data, such as financial time series, animal movements, 
%growth expressions for individual trees in a forest, 
growth functions 
%blue
%of 
for 
trees in a forest stand, the spatial extensions of outbreaks of a disease over time with respect to the outbreak centres, population growth functions of towns/cities in a country, and different functions describing spatial dependence (e.g.\ LISA functions; see Section \ref{SectionMarkStructures} and the references therein), 
are represented as collections $\{f_1(t),\ldots,f_n(t)\}$, $t\in\T\subset[0,\infty)$, $n\geq1$, of functions/paths in some $k$-dimensional Euclidean space $\R^k$, $k\geq1$; note that the argument $t$ need not represent time, it could e.g.\ represent spatial distance. 
The common approach to deal with such data within the field of functional data analysis (FDA) \citep{RamsaySilverman2005} is to assume that the functions $f_i$, $i=1,\ldots,n$, belong to some suitable family of functions (usually an $L_2$-space) and are realisations/sample paths of some collection of independent and identically distributed (iid) random functions/stochastic processes $\{F_1(t),\ldots,F_n(t)\}$, $t\in\T$, with sample paths belonging to the family of functions in question. 

% {\color{red}\bf Here we should provide some examples and reference from FDA.}

For many applications, however, the following two adequate questions may quite naturally arise: 

\begin{enumerate}

\item Does it make sense to assume that the random elements $F_1,\ldots,F_n$, which have generated the functional data set $\{f_1,\ldots,f_n\}$, are in fact iid? 

\item Is the study designed in such a way that the sample size $n$ is known a priori, or is $n$ in fact unknown before the data set is realised?
\end{enumerate}
\begin{figure}[H]
\begin{center}
\begin{tabular}{cc}
 \includegraphics*[width=0.35\textwidth,  height=0.35\textwidth]{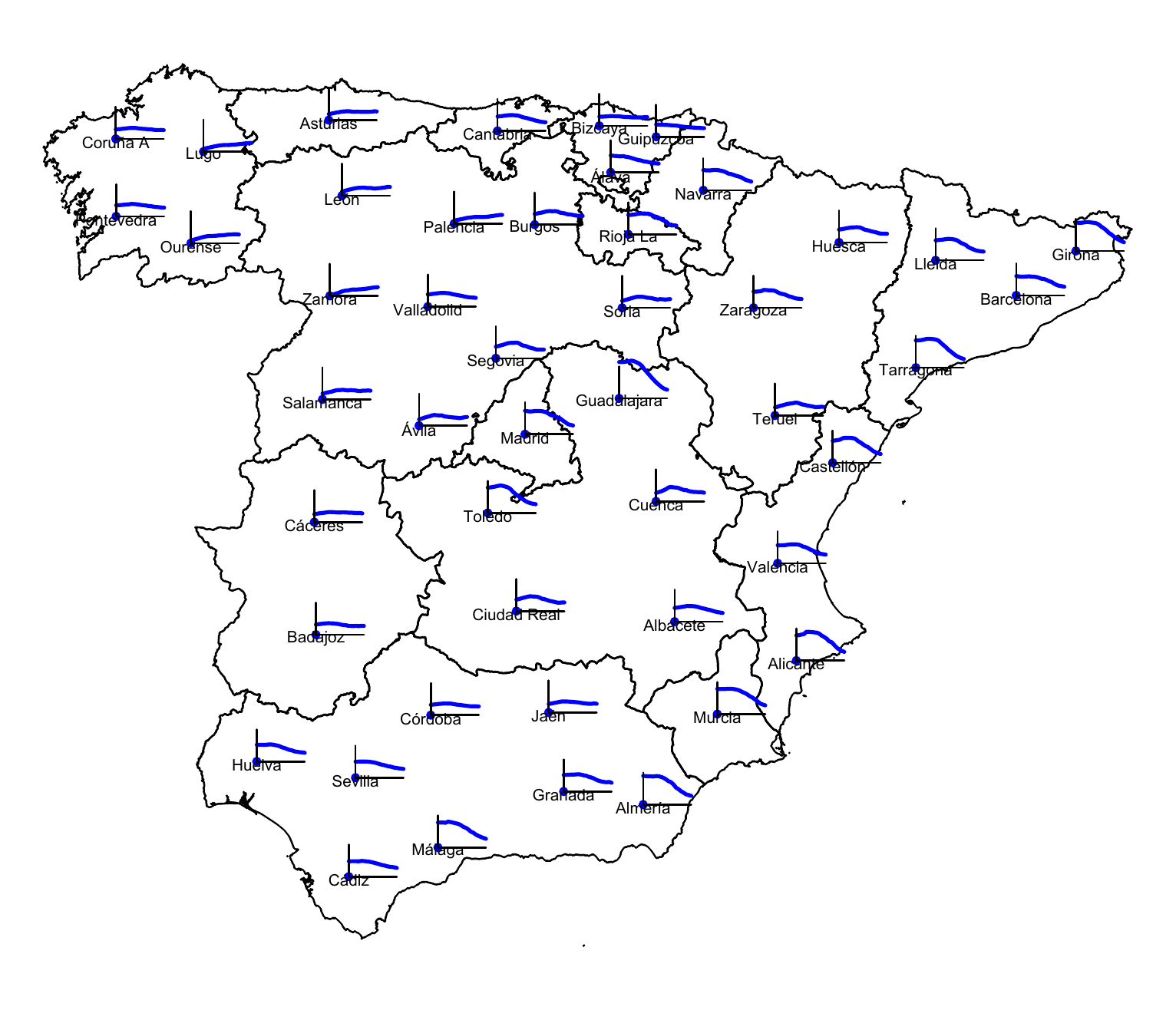}
&
 \includegraphics*[width=0.35\textwidth,  height=0.35\textwidth]{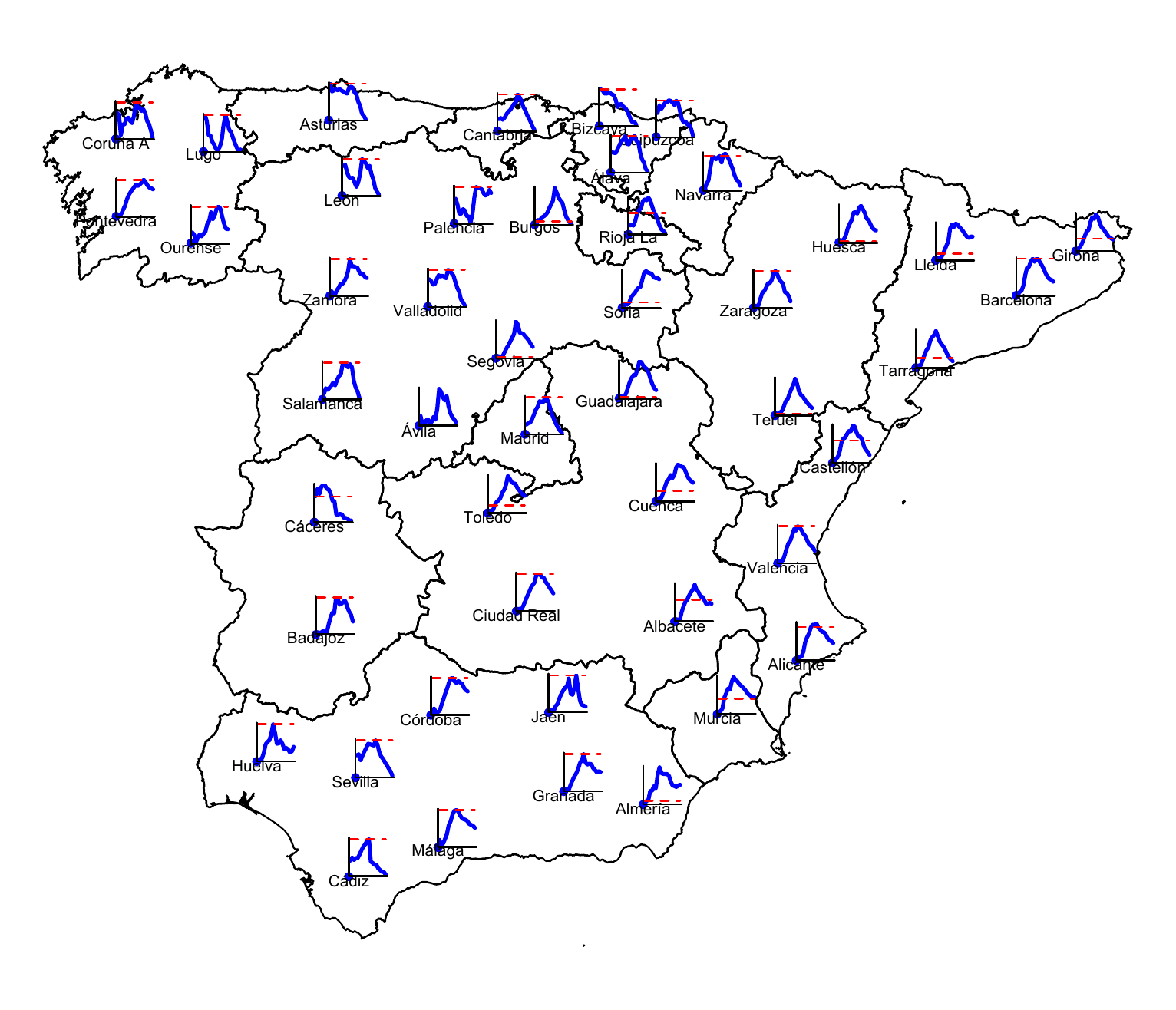}
 \cr
\includegraphics*[width=0.35\textwidth, height=0.35\textwidth]{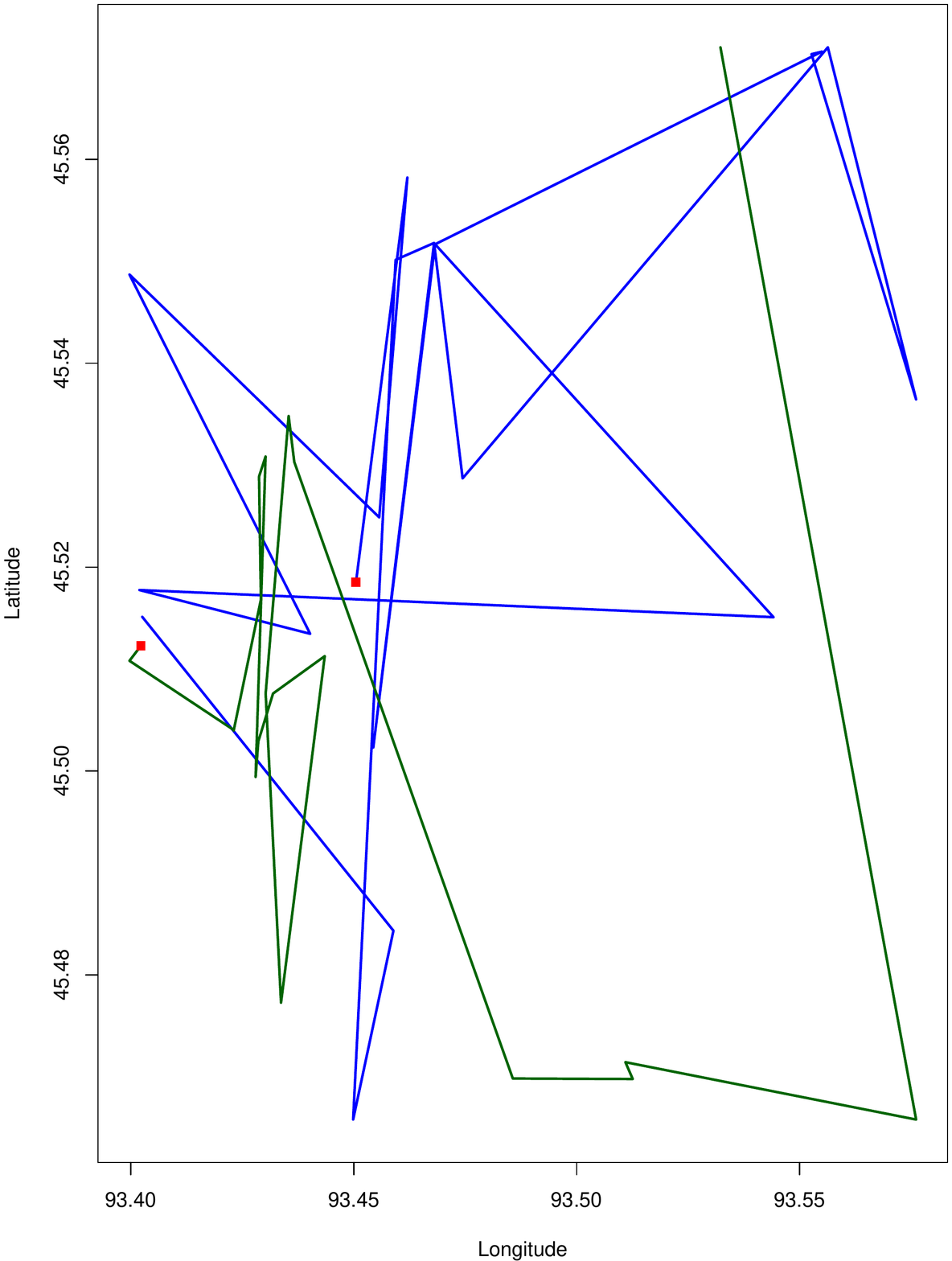}
&
\includegraphics*[width=0.35\textwidth, height=0.33\textwidth]{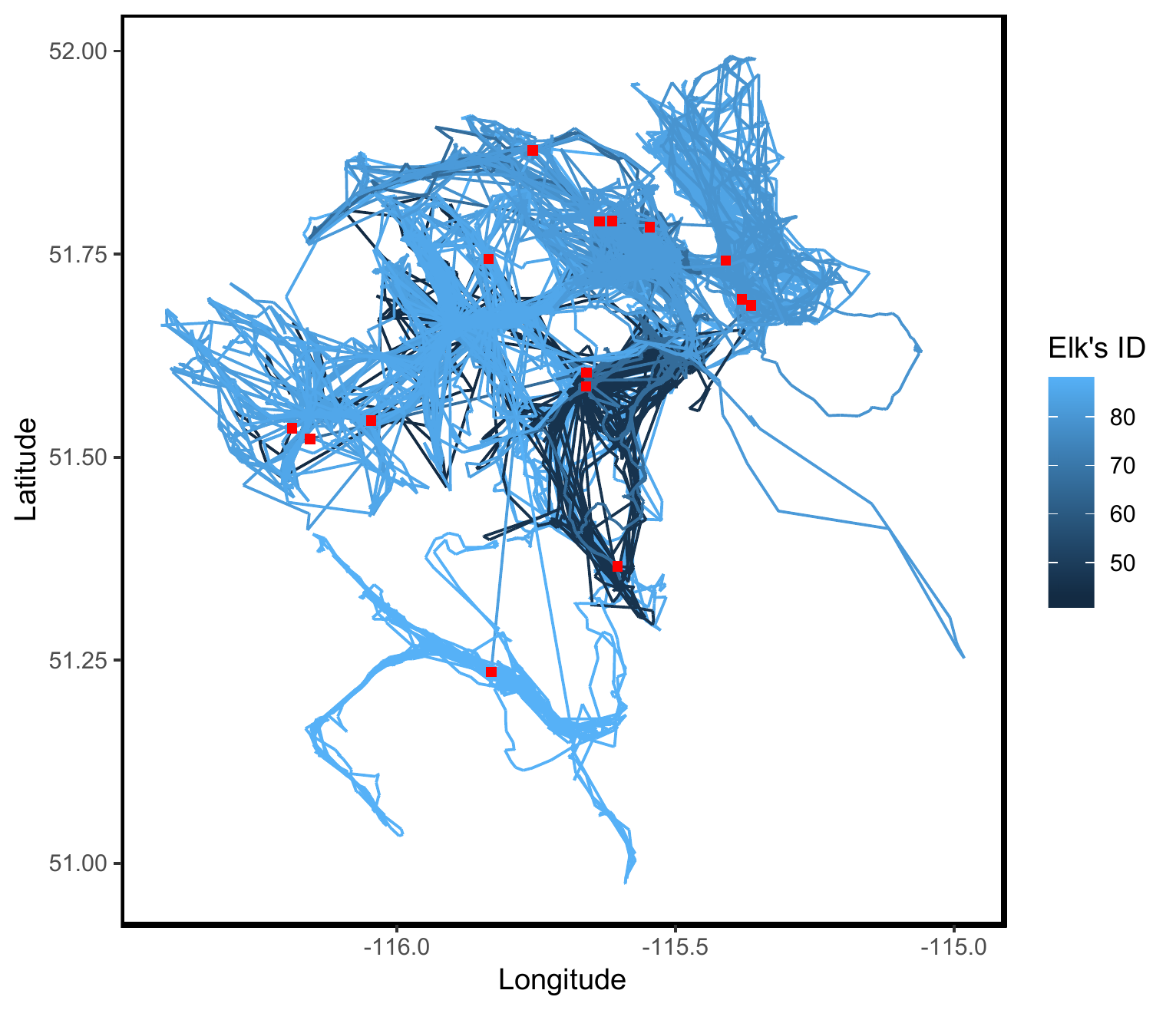}
%  \includegraphics*[width=0.5\textwidth,height=7.6cm]{DemoEvolubyState}
% &
%  \includegraphics*[width=0.5\textwidth,height=7.6cm]{SexRatioEvolSpainProvincelatlong}
%  \cr
% \includegraphics*[width=0.5\textwidth, height=7.6cm]{Wolves_Mongolia20}
% &
% \includegraphics*[width=0.5\textwidth, height=7.3cm]{ElktrackPlotS}
\end{tabular}
 \caption{
 Top panels: Spanish province data. Log-scale demographic evolution (top left) and sex ratio (top right) in 47 provinces of Spain, for the years 1998 to 2017. 
 Bottom panel:  Movement tracks. The first 20 movement tracks of two Mongolian wolves (bottom left). Movement tracks of 15 Ya Ha Tinda elks in Banff national park, Canada (bottom right); 
 %. In bottom panels 
 the red squares are the starting points of the tracks.}
 \label{Illu}
\end{center}
\end{figure}
Functional data sets (believed to be) generated in accordance with the above remarks will be referred to as {\em functional marked point patterns} and 
Figure~\ref{Illu} provides illustrative examples of such data sets. 
%where the functional data sets most likely are generated in accordance with the above remarks, and we refer to such data sets as {\em functional marked point patterns}. 
The top panels show two functional marked point patterns 
%functional data sets, which we refer to as functional marked point patterns,  
based on the 
%a data set of 
centres of the provinces on the Spanish mainland. To each point, which corresponds to a centre, we have associated the demographic evolution of the population on logarithmic scale (left) and the sex ratio (right), over the the years 1998 to 2017. 
%has been assigned is  in each province between 1998 to 2017.
In the top right panel, for each of the $47$ functions/provinces, 
%$i$, $i=1,\ldots, 47$ 
the horizontal red dashed line corresponds to $y=1$, which illustrates the case where we have the same size of genders in the province in question. 
 The 
bottom panels show animal movement tracks. The lower left panel shows the first 20 movement tracks of two Mongolian wolves, starting from random initial monitoring locations (red squares); the data are taken from the Movebank website. The lower right panel shows the movement tracks of 15 Ya Ha Tinda elk \citep{Hebblewhite:Merrill:08}, starting from some random initial monitoring locations.

Another setting where these questions also naturally arise is found in spatio-temporal geostatistics \citep{Montero:etal:15}. Assume that each of the data-generating stochastic processes $F_i(t)=Z(x_i,t)$, $t\in\T$, $i=1,\ldots,n$, is associated with a spatial location $x_i\in W\subset\R^d$ and that $Z(x,t)$, $(x,t)\in W\times\T$, is a (Gaussian) spatio-temporal random field. Here the functions $F_1,\ldots,F_n$ are clearly not independent (ignoring pathological cases) and one may further ask whether it would not in fact make sense to assume that the sampling/monitoring locations $x_1,\ldots,x_n$ are actually randomly generated. In addition, does it make sense to assume that the total number of such locations was fixed a priori, or did these locations e.g.\ appear over times (in relation to each other), thus allowing us to treat them as a randomly evolving entity with a random total number of components $N\geq1$? For instance, all the weather stations monitoring precipitation in a given country/region have (most likely) arrived over time, in relation to each other, rather than being placed at their individual locations at the same time. E.g., we do not know a priori how many stations will have appeared during the period 2010-2040 and where they will be located.

% {\color{red} We should mention spatially dependent functional data somewhere in the text by referring to Giraldo et al. (2009). }

Taking these remarks into account, we argue that for many functional data sets $\{f_1(t),\ldots,f_n(t)\}$, $t\in\T\subset[0,\infty)$, $n\geq1$, 
%in many settings 
it would make sense to assume i) that $n\geq1$ is the realisation of some discrete non-negative random variable $N$ and ii) that (conditional on $N=n$) the random functions $F_1,\ldots,F_n$ are possibly dependent. 
A natural way to tackle the statistical analysis under such non-standard assumptions 
%, we argue, 
is to assume that the functional data set is generated by a {\em point process} in some space $\F$ of functions $f:\T\to\R^k$. This would mean that we would model the functional data set (a functional marked point pattern) as the realisation of a set of random functions $\{F_1(t),\ldots,F_N(t)\}$, $t\in\T$, of random size $N$. Note that by construction, all components $F_i$ have the same marginal distributions. Under such a setup, a so-called binomial point process \citep{Moller,VanLieshout} would yield the classical FDA setup mentioned above. 
Note that the idea of analysing point patterns (collections of points) with attached functions has already been noted in the literature \citep{Comas,DelicadoGiraldoComas}.

It is often the case that 
%As is often the case, 
these functions 
%may 
have some sort of spatial dependence. E.g., two functions $f_i$ and $f_j$, with starting points $f_i(0)$ and $f_j(0)$ which are spatially close to each other in $\R^k$, either gain or loose from each other's vicinity. Accordingly, it seems natural to generate $F_1,\ldots,F_N$ conditionally on some collection of random spatial locations $X_i$ and some further set of random variables $L_i$ associated with the random functions $F_i$; conditionally on the spatial locations, the $L_i$'s would influence the random functions $F_i$ in a non-spatial sense. We argue that the natural setting to do this is through {\em functional marked point processes (FMPPs)}. More precisely, we define an FMPP $\Psi=\{(X_i,(L_i,F_i))\}_{i=1}^N$ as a spatial point process $\Psi_G=\{X_i\}_{i=1}^N$ in $\R^d$ to which we assign marks $\{(L_i,F_i)\}_{i=1}^N$; note that by forcing all $L_i$ to take the same value, we may reduce the FMPP to the collection $\{(X_i,F_i)\}_{i=1}^N$.

We here take a full grip and provide a proper framework for FMPPs, where we in particular take into account that for the standard point process machinery to go through (in particular the use of regular conditional probability distributions), one has to assume that the mark space, and thereby the function space $\F$, is a Polish space \citep{DVJ2}. 
In particular, one may then provide a reference stochastic process $X^{\F}$, 
with sample paths in $\F$, 
%on $\F$, 
whose distribution $\nu_{\F}$ on $\F$ acts as a reference measure which one integrates with respect to (in a Radon-Nikodym sense). We further provide a plethora of examples from the literature which fit into the FMPP framework and discuss these in some detail. Examples include geostatistics \citep{cressie:kornak:03} with random sampling locations, point processes marked with "spatio-temporal random closed sets", e.g.\ spatio-temporal boolean models \citep{Sebastian:etal:06}, constructed functional marks, e.g.\ so-called LISA functions \citep{MateuLorenzoPorcu}, and the Renshaw-S{\"a}rkk{\"a} growth-interaction model \citep{RS1,RS2}. To be able to carry out statistical analyses in the context of FMPPs, various moment characteristics, such as product densities, are required and we here cover such characteristics. A key observation here is that we, in contrast to previous works, completely move away from the (arguably unrealistic) assumption of stationarity. 
We then proceed to discussing various general marking structures, such as the marks having a common marginal distribution and the marks being (conditionally) independent. 
%
% Turning to summary statistics for FMPPs, 
% we here consider the marked summary statistics of \citet{CronieLieshoutMPP}, which allow us analyse spatial interaction with respect to the spatial locations, between different user-specified mark-groupings. More precisely, by categorising the (functional) marks, we may study how the associated spatial locations interact when we assume that the spatial intensity of points varies over space (inhomogeneity). 
To study interactions between functional marks, we further define new types of summary statistics (of arbitrary order), which we refer to as {\em weighted marked reduced moment measures} and {\em mark correlation functionals}. These summary statistics are essentially mark-test function-weighted summary statistics %of the unmarked point process, 
which have been restricted to pre-specified mark-groupings. We study them in different contexts and show how they under different assumptions reduce to different existing summary statistics. 
In addition, we provide non-parametric estimators for all the summary statistics and show their unbiasedness.
%and consistency results. 
We also show how these summary statistic estimators can be employed to 
%perform 
carry out 
functional data analysis when the functional data-generating elements are spatially dependent (according to an FMPP). 
%{\color{red}
%To illustrate the use of our estimators, 
We finally apply our summary statistic estimators to the data illustrated in the top panels of Figure~\ref{Illu}, in order to analyse 
%different sets of functional data, namely, 
population structures such as demographic evolution and sex ratio of human population over time in Spanish provinces.
\section{Functional marked point processes}\label{sec:FMPP}

% {\color{blue} We should mention that under which conditions our functional marked point process reduce to regular marked point process.}

% {\color{red}\bf We need to be consistent with our subset notation; do we use $\subset$ or $\subset$ (I usually use $\subseteq$)?}

Throughout, let $\X$ be a subset of $d$-dimensional Euclidean space $\R^d$, $d\geq1$, which is either compact or given by all of $\R^d$. 
Denote by $\|\cdot\|=\|\cdot\|_d$ the $d$-dimensional Euclidean norm, by $\B(\X)$ the Borel sets of $\X\subset\R^d$ and by $|\cdot|=|\cdot|_d$ the Lebesgue measure on $\X$; $\int\de x$ denotes integration w.r.t.\ $|\cdot|$. It will be clear from the context whether $|\cdot|$ is used for the Lebesgue measure or the absolute value. We 
%will 
denote by $\B(\cdot)^n$ the $n$-fold product of an arbitrary Borel $\sigma$-algebra $\B(\cdot)$ with itself. Moreover, we 
%will 
denote by $\mu_1\otimes\mu_2$ the product measure generated by measures $\mu_1$ and $\mu_2$ and by $\mu_1^n$ the $n$-fold product of $\mu_1$ with itself. Recall further that a topological space is called {\em Polish} if there is a metric/distance which generates the underlying topology and turns the space into a complete and separable metric space. A closed ball of radius $r\geq0$, centred in $x\in S$, where the space $\mathcal{S}$ is equipped with a metric $d_{\mathcal{S}}(\cdot,\cdot)$, will be denoted by $B_{\mathcal{S}}[x,r]=\{y\in S:d_{\mathcal{S}}(x,y)\leq r\}$.

Consider a point process $\Psi_G=\{X_i\}_{i=1}^N$, $N\in\N_0=\{0,1,2,\ldots,\infty\}$, on $\X$ \citep{Illian,SKM}. 
%is a locally finite random countable subset of $\X$.
%some subset $\X$ of $d$-dimensional Euclidean space $\R^d$, $d\geq1$; we let $\X$ be either compact or the whole of $\R^d$ and we denote by $|\cdot|$ the Lebesgue measure on the Borel sets $\B(\X)$. 
Throughout the paper we refer to $\Psi_G$ as a {\em ground/unmarked} point process. 
To each point of $\Psi_G$ we may attach a further random element, a so-called mark, in order to construct a marked point process $\Psi$. 
In this paper, a mark is given by a $k$-dimensional random function/stochastic process $F_i(t)=(F_{i1}(t),\ldots,F_{ik}(t))$, $t\in\T\subset[0,\infty)$, a {\em functional mark}, possibly together with some further random variable $L_i$, which we refer to as an {\em auxiliary/latent mark}. The resulting marked point process $\Psi=\{(X_i,(L_i,F_i))\}_{i=1}^N$, $N\in\N_0$, will be referred to as a {\em functional marked point process (FMPP)}. 
The main purpose of including auxiliary marks is to control the supports of the functional marks, on the one hand, and on the other hand they may serve as indicators/labels for different types of points of the point process, in a classical multi-type point process sense.

\subsection{Construction of functional marked point processes}

%{\color{red}Change $k_{\A}$ etc to $\ell_{\A}$ etc.

%Check that all the references are ok.

%Check e.g. and i.e.
%}

To formally define an FMPP, we first need to specify the underlying mark space $\M$. The general theory for marked point processes \citep{DVJ1,DVJ2,VanLieshout} allows us to consider any Polish space $\M$ as mark space. 
Here we let the mark space be the Polish product space $\M=\A\times\F$ given by the product of 
\begin{itemize}
\item a Borel subset $\A\ni L_i$ of some Euclidean space $\R^{k_{\A}}$, $k_{\A}\geq1$, referred to as the {\em auxiliary/latent mark space}, 
\item a Polish {\em function space} $\F=\U^k\ni F_i$, $k\geq1$; 
each element $f=(f_1,\ldots,f_k)\in\F=\U^k$ 
%where $f(t)\in\R^k$ for any $t\in\T$,
%where 
%$f:\T\subset[0,\infty)\to\R$ for $f\in\U$. {\color{red} where 
has components $f_j:\T\to\R$, 
%$F_i$, i.e.\, $F_{ij}$, 
$j=1,\ldots,k$. %of elements $f=(f_1,\ldots,f_k)\in\F$ 
%which are functions on $\T\subset[0,\infty)$ with values in $\R$. %; note that $f(t)\in\R^k$ for any $t\in\T$.
%}
\end{itemize}
Note that due to the Polish structures of 
these spaces, 
%$\A$ and $\F$,
the Borel sets of $\M$ are given by the product $\sigma$-algebra $\B(\M)=\B(\A\times\F)=\B(\A)\otimes\B(\F)=\B(\R^{k_{\A}})\otimes\B(\U^k)=\B(\R)^{k_{\A}}\otimes\B(\U)^k$. 
%, where the last component is the $k$-fold product of $\B(\U)$ with itself. 
Explicit examples of auxiliary and functional mark spaces are given in Appendix~\ref{s:MarkSpaceChoices}.
%Section \ref{s:MarkSpaceChoices} in the Appendix.

% Since both $\A$ and $\F$ are Polish spaces, with respective metrics $d_{\A}(\cdot,\cdot)$ and $d_{\F}(\cdot,\cdot)$, 
% by endowing $\M=\A\times\F$ with the supremum metric
% $$
% d_{\M}((l_1,f_1),(l_2,f_2)) = \max \{d_{\A}(l_1,l_2),d_{\F}(f_1,f_2)\},
% \quad (l_1,f_1),(l_2,f_2)\in\M=\A\times\F,
% $$
% (or any other equivalent metric) $\M$ itself becomes csm \citep[p.\ 377]{DVJ1} and its Borel sets are given by
% $\B(\M)=\B(\A\times\F)=\B(\A)\otimes\B(\F)$. %(see e.g.\ \citep[Lemma 6.4.2.]{Bogachev}). 

%{\color{blue}\bf Change $\psi$ to $\psi$ everywhere!}

Let $\Y=\X\times\M$ and let $ N_{lf}$ be the collection of all point patterns, i.e.\ locally finite subsets $\psi=\{(x_1,l_1,f_1),\ldots,(x_n,l_n,f_n)\}\subset\Y$, $n\geq0$; $n=0$ corresponds to $\psi=\emptyset$. Note that local finiteness means that the cardinality $\psi(A)=|\psi\cap A|$ is finite for any bounded Borel set $A\in\B(\Y)$. Denote the corresponding counting measure $\sigma$-algebra on $ N_{lf}$ by $\NN_{lf}$ (see \citet[Chapter 9]{DVJ2}); $\NN_{lf}$ is the $\sigma$-algebra generated by the mappings $\psi\mapsto\psi(A)\in\N_0$, $\psi\in N_{lf}$, $A\in\B(\Y)$. 
By construction, since point patterns here are defined as subsets, all $\psi\in  N_{lf}$ are simple, i.e.\ $\psi(\{(x,l,f)\})\leq\psi_{G}(\{x\})\in\{0,1\}$ for any $(x,l,m)\in\X\times\A\times\F$.

\begin{definition}
Given some probability space $(\Omega,\Sigma,\P)$, a point process $\Psi=\{(X_1,L_1,F_1),\ldots,(X_N,L_N,F_N)\}$, $N\in\N_0$, on $\Y=\X\times\M=\X\times\A\times\F$ 
%$\Psi:\Omega\rightarrow N_{lf}$ 
%, $\omega\mapsto\Psi(\cdot;\omega)$, 
is a measurable mapping from %the probability space 
$(\Omega,\Sigma,\P)$ to the space $(N_{lf},\NN_{lf})$. 

If a point process $\Psi$ on $\Y$ is such that the {\em ground/unmarked point process} $\Psi_G=\{x:(x,l,f)\in\Psi\}$ is a well defined point process in $\X$, we call $\Psi$ 
%a \emph{point process} on $\Y$. 
%When, in addition, $\Psi\in N_{lf}^*$ a.s. we call 
%$$
%\Psi 
%= \sum_{y\in\Psi}\delta_{y} 
%= \sum_{(x,l,f)\in\Psi}\delta_{(x,l,f)}
%$$
a \emph{(simple) functional marked point process (FMPP)} and when $\X\subset\R^{d-1}\times\R$, $d\geq2$, and $\Psi_G$ is a spatio-temporal point process in $\X$, we call $\Psi$ 
%a \emph{point process} on $\Y$. 
%When, in addition, $\Psi\in N_{lf}^*$ a.s. we call 
%$$
%\Psi 
%= \sum_{y\in\Psi}\delta_{y} 
%= \sum_{(x,l,f)\in\Psi}\delta_{(x,l,f)}
%$$
a \emph{spatio-temporal FMPP}. 
%When $\F$ is given by the c\`adl\`ag space we may refer to $\Psi$ as a c\`adl\`ag FMPP (CFMPP). 
%
\end{definition}

Note that $\Psi$ either may be treated as a locally finite random subset $\Psi=\{(X_i,L_i,F_i)\}_{i=1}^N\subset\Y$, 
%$F_i=\{(F_{1i}(t),\ldots,F_{ki}(t))\}_{t\in\T}$, 
%$0\leq i\leq N$, 
or as a random counting measure 
$$
\Psi(\cdot)
%{\color{red}= \sum_{y\in\Psi}\delta_{y}(\cdot)}
= \sum_{(x,l,f)\in\Psi}\delta_{(x,l,f)}(\cdot)
= \sum_{i=1}^N\delta_{(X_i,L_i,F_i)}(\cdot)
$$
%{\color{red}  it is better to remove}
on $(\Y,\B(\Y))$ with \emph{ground measure/process} 
$$
\Psi_{G}(\cdot) = \sum_{x\in\Psi_G}\delta_x(\cdot) 
=\sum_{(x,l,f)\in\Psi}\delta_{(x,l,f)}(\cdot\times\A\times\F)
= \sum_{i=1}^N\delta_{X_i}(\cdot) 
$$
on $(\X,\B(\X))$.
%with support $\Psi_{G}=\{X_i\}_{i=1}^{N}\subset\X$, is a well-defined (i.e., a.s.\ locally finite) simple point process on $\X$. 
In the spatio-temporal case, it may be convenient to write $\Psi_G=\{(X_i,T_i)\}_{i=1}^N$ to emphasize that each ground process point has a spatial component, $X_i\in\R^{d-1}$, as well as a temporal component $T_i\in\R$.

\begin{rem}
Since all of the underlying spaces are Polish, we may choose a metric $d(\cdot,\cdot)$ on $\Y$ which turns $\Y$ into a complete and separable metric space, with metric topology given by the underlying Polish topology. E.g, we may consider
%Putting all the above together, so that we endow $\Y=\X\times\M=\X\times(\A\times\F)$ with the metric 
\beann
d((x_1,l_1,f_1),(x_2,l_2,f_2))
=
% \max\{\|x_1-x_2\|,d_{\M}((l_1,f_1),(l_2,f_2))\}
% \\
% &=&
\max\{d_{\X}(x_1,x_2),d_{\A}(l_1,l_2),d_{\F}(f_1,f_2)\},
\eeann
where $d_{\X}(x_1,x_2)=\|x_1-x_2\|_d$ and the metrics $d_{\A}(\cdot,\cdot)$ and $d_{\F}(\cdot,\cdot)$ make $\A$ and $\F$ complete and separable metrics spaces \citep{VanLieshout}; when $\A=\R^{k_{\A}}$ or $\A$ is a compact subset of $\R^{k_{\A}}$ we may use $d_{\A}(l_1,l_2)=\|l_1-l_2\|_{k_{\A}}$. In the spatio-temporal case, it may be natural to consider $d_{\X}((x_1,t_1),(x_2,t_2))=\max\{\|x_1-x_2\|_{d-1},|t_1-t_2|\}$, $(x_1,t_1),(x_2,t_2)\in\X\subset\R^{d-1}\times\R=\R^{d}$ \citep{CronieLieshoutSTPP}, which is topologically equivalent to $d_{\X}((x_1,t_1),(x_2,t_2))=\|(x_1,t_1)-(x_2,t_2)\|_d$.
\end{rem}

We will write 
$P(R) = P_{\Psi}(R) = \P(\{\omega\in\Omega:\Psi(\omega)\in R\})$,
$R\in\NN_{lf}$, for the distribution of $\Psi$, i.e.\ the probability measure that $\Psi$ induces on $( N_{lf}, \NN_{lf})$. 
When $\X=\R^d$, for any $\psi\in N_{lf}$ and any $z\in\R^d$, we will write $\psi + z$ to denote 
$\sum_{(x,l,f)\in\psi}\delta_{(x+z,l,f)}$ (or $\{(x+z,l,m):(x,l,m)\in\psi\}$), i.e.\ a shift of $\psi$ in the ground space by the vector $z$.
If $\Psi + z \stackrel{d}{=} \Psi$, i.e.\ $P_{\Psi}(\cdot)=P_{\Psi+z}(\cdot)$, for any $z$, we say that $\Psi$ is {\em stationary}. 
Moreover, $\Psi$ is {\em isotropic} if $\Psi$ is rotation invariant in the ground space, i.e.\ the rotated FMPP $r\Psi=\{(rX_i, L_i, F_i)\}_{i=1}^N$ has the same distribution as $\Psi$ for any rotation $r$.

\subsection{Components of FMPPs}
\label{SectionComponents}
We emphasize that any collection of elements $\{(X_1,L_1,F_1),\ldots,(X_n,L_n,F_n)\}\subset\Psi$, $n\geq1$, consists of the combination of: 
\begin{itemize}
\item a collection of random spatial locations $X_1,\ldots,X_n\in\X$,
\item a collection $L_1,\ldots,L_n$ of random variables taking values in $\A$,
\item an $n$-dimensional random function/stochastic process $\{F_1(t),\ldots,F_n(t)\}_{t\in\T}\in(\R^k)^n$, with realisations in $\F^n$; formally, this is an unordered collection of $n$ stochastic processes in $\R^k$ with sample paths in $\F=\U^k\subset\{f|f:\T\rightarrow\R\}^k$.
\end{itemize}

%{\color{blue}\bf Check that the indexation in $L_i=(L_{1i},\ldots,L_{k_{\A}i})$ is consistent throughout the paper!}

In particular, $\Psi_{\X\times\A}=\{(X_i,L_i)\}_{i=1}^N$ is a marked point process of the usual kind, with locations in $\R^d$ and marks in $\A\subset\R^{k_{\A}}$, 
i.e.\ 
%Since $\A\subset\R^{k_{\A}}$, $k_{\A}\geq1$, 
each auxiliary mark $L_i=(L_{1i},\ldots,L_{k_{\A}i})$ is given by a $k_{\A}$-dimensional random vector. Depending on how $\A$ and the distributions of the $L_i$'s are specified, we are able to consider an array of different settings. E.g., if $\A=\{1,\ldots,k_d\}$, $k_d\geq2$, each random variable 
% the distributions of the random variables 
$L_i$ has a discrete distribution on $\A$.
% are discrete with support on the integers $\Z\cap\A$. 
Since $\Psi_{\X\times\A}$ hereby becomes a multi-type/multivariate point process in $\R^d$, 
one may 
%we 
call such FMPPs \emph{multi-type/multivariate} \citep{DVJ1,VanLieshout,Gelfand:etal:10}. In  Appendix~\ref{s:MarkSpaceChoices},
%Section \ref{s:MarkSpaceChoices} in the Appendix, 
we look closer at specific choices for $\A$. It is often convenient to write $\A=\A_d$ to emphasise when we have a discrete auxiliary mark space, such as $\A_d=\{1,\ldots,k_d\}$, and $\A=\A_c$ to emphasise when have a continuous space ((closure) of an open set), such as $\A_c=\R^{k_{\A}}$.

Within the current definition of FMPPs we may also consider the scenario where the auxiliary marks play no role, and thereby may be ignored. This may be obtained by e.g.\ setting $\A=\{c\}$ for some constant $c\in\R$, 
so that all auxiliary marks attain the value $c$, 
or equivalently, setting $L_i=c$ a.s.\ for any $i=1,\ldots,N$, assuming that $c\in\A$.

Note that when we want to consider functional marks with realisations given by functions $f(t)=(f_1(t),\ldots,f_k(t))\in\R^k$, $t\in\T$, which describe spatial paths, we let $k\geq2$. 
Often the spatial locations $X_i$ describe the initial location of such a path and it is then natural to assume that $d=k\geq2$ and $f(t)\in\X$ a.s.\ for any $t\in\T$. An application here would be that the marks describe movements of animals, living within some spatial domain $\X$; recall Figure \ref{Illu}. 

Recall that each functional mark $F_i(t)=(F_{i1}(t),\ldots,F_{ik}(t))\in\R^k$, $t\in\T\subset[0,\infty)$, $i=1,\ldots,N$, is realised in the measurable space $(\F,\B(\F))$, where $\F=\U^k$, $k\geq1$, and $\U$ are Polish function spaces (products of Polish spaces are Polish). 
By conditioning $\Psi$ on $\Psi_{\X\times\A}$, which includes conditioning on $N$, we obtain the random functional
$$
\Psi|\Psi_{\X\times\A}=\{F_1|\Psi_{\X\times\A},\ldots,F_N|\Psi_{\X\times\A}\}
=\{F_1(t)|\Psi_{\X\times\A},\ldots,F_N(t)|\Psi_{\X\times\A}\}_{t\in\T}\subset\F,
$$
%{\color{red}$\Psi_{\F}=\{F_i; (X_i,L_i,F_i)\in\Psi\}$} 
which may be regarded as a stochastic process with dimension $N$ and with the same marginal distributions for all of its components. 
Due to the inherent temporally evolving nature of the functional marks, one may further consider some filtration $\Sigma_{\T}$, and thus obtain a filtered probability space $(\Omega,\Sigma,\Sigma_{\T},\P)$, such that all $F_i=\{F_i(t)\}_{t\in\T}$, $i=1,\ldots,N$, are adapted to $\Sigma_{\T}$ (see Appendix~\ref{s:FunctionSpaces}  for more details). 
%Section \ref{s:FunctionSpaces} in the Appendix for further details on the matter). 

\begin{rem}
Formally, $\Psi|\Psi_{\X\times\A}$ may be obtained as the point process generated by the family of regular conditional probabilities obtained by disintegrating $P_{\Psi}$ with respect to the distribution of $\Psi_{\X\times\A}$ on its point pattern space \citep[Appendix A1.5.]{DVJ1}. 

% this conditioning may be carried out by means of disintegration of $P_{\Psi}$ with respect to the distribution of $\Psi_{\X\times\A}$, whereby $\Psi|\Psi_{\X\times\A}$ is generated by the corresponding family of regular probabilities \citep[Appendix]{DVJ1}.
\end{rem}

%, and that $\U$ is a csm space. 
% Due to the assumed Polish structure, 
% %Since we will consider csm spaces $\U$, 
% the Borel sets of the sample space for the functional marks satisfy $\B(\F) = \B(\U^k)=\B(\U)^k$. 

We impose the Polish assumption on $\U$ in order to carry out the usual marked point process analysis \citep{DVJ1,DVJ2}; note that $\U$ being Polish implies that $\F$ is Polish and $\B(\F)=\B(\U^k)=\B(\U)^k$. 
However, choosing a Polish function space $\U$ is a delicate matter; 
note that \citet{MateuFMPP} did not address this issue. 
%In Section \ref{s:MarkSpaceChoices} in the Appendix, 
In Appendix~\ref{s:FunctionSpaces}, 
%we look closer at this matter 
we consider functional mark spaces in more detail and there we cover the two most natural choices for $\U$, namely Skorohod spaces and $L_p$-spaces \citep{Billingsley,EthierKurtz,JacodShiryaev,Silvestrov}. 
Note that these two classes of functions are not mutually exclusive. 
Noting that, in general, the support $\supp(f)=\{t\in\T:f(t)\neq0\}\subset\T$ of a function $f\in\F$ need not be given by all of $\T$, 
in some contexts it may be natural to let $\Psi_{\X\times\A}$ govern the supports $\supp(F_i)=\{t\in\T:F_i(t)\neq0\in\R^k\}$, $i=1,\ldots,N$. 
To illustrate this idea, consider the case where $d=1$ and $\X=\T=[0,\infty)$, so that $\Psi_G=\{T_i\}_{i=1}^N\subset[0,\infty)$ is a temporal point process. In addition, assume that $k_{\A}=1$ and that each auxiliary mark $L_i$ is some non-negative random variable, such as an exponentially distributed one, which does not depend on $\Psi_G$. 
% This implies that $\Psi_{\X\times\A}$ may be treated as a marked spatio-temporal point process with event times $T_i\in\T$ and continuous marks $L_i$ in $[0,\infty)$. 
Let us think of $T_i$ and $L_i$ as a point's {\em birth time} and {\em lifetime}, respectively. Defining the corresponding {\em death time} as $D_i=T_i+L_i$, 
we may then e.g.\ let 
$$
F_i(t)|\Psi_{\X\times\A}=(F_{i1}(t)|\Psi_{\X\times\A},\ldots,F_{ik}(t)|\Psi_{\X\times\A})=0
$$ 
for all $t\notin[T_i,D_i)$ a.s., where $0$ is the $k$-dimensional vector of $0$s. Note further that there in addition to this may exist $t\in[T_i,D_i)$ such that $F_i(t)|\Psi_{\X\times\A}=0$ in some way (e.g.\ absorption), which is something governed by the distribution of $\{F_i(t)|\Psi_{\X\times\A}\}_{t\in\T}$ on $\F$. %{\color{red} mark process distribution}. 
An explicit construction to obtain this when $k=1$ would e.g.\ be $F_i(t)=\1_{[T_i,D_i)}(t)Y_i((t-T_i)\wedge0)$, $t\in\T$, for some stochastic process $Y(t)$, $t\in[0,\infty)$, which starts in $0$.

\subsection{Reference measures and reference stochastic processes}
\label{SectionRefernceMeasures}
%{\Large\color{red} Move to supplementary materials}\\
For the purpose of integration, among other things, we need a reference measure on $(\Y,\B(\Y))$. 
We let it be given by the product measure 
\begin{align}
\label{ReferenceMeasure}
\nu(C\times D\times E) 
=& [|\cdot|\otimes\nu_{\M}](C\times(D\times E))
= 
|C| \nu_{\M}(D\times E)
=
|C| [\nu_{\A}\otimes\nu_{\F}]](D\times E)
% &=&
% [|\cdot|_d\otimes[\nu_{\A}\otimes\nu_{\F}]](C\times D\times E)
% \\
% &=&
=
|C|\nu_{\A}(D)\nu_{\F}(E)
,
\end{align}
where $C\times D\times E\in\B(\Y)=\B(\X)\otimes\B(\A)\otimes\B(\F)$ and we note that, as usual, the reference measure on the ground space $\X$ is given by the Lebesgue measure $|\cdot|=|\cdot|_d$ on $\X\subset\R^d$, $d\geq1$. 
Moreover, we need $\nu_{\M}$ to be a finite measure so both $\nu_{\A}$ and $\nu_{\F}$ need to be finite measures on $(\A,\B(\A))$ and $(\F,\B(\F))$, respectively. 

Regarding the reference measure on the auxiliary mark space, in 
Appendix~\ref{s:MarkSpaceChoices} we provide a few examples based on different choices for $\A$. Most noteworthy here is that if $\A=\A_d$ is a discrete space then 
$\nu_{\A}=\nu_{\A_d}$ is a discrete measure $\nu_{\A_d}(\cdot)=\sum_{i\in\A_d}\Delta_i\delta_i(\cdot)$, $\Delta_i\geq0$ (e.g.\ the counting measure, given by $\Delta_i\equiv1$), if $\A=\A_c$ is a continuous 
space 
%((the closure) of an open set) 
then 
we may choose $\nu_{\A}=\nu_{\A_c}$ to be the $k_{\A}$-dimensional Lebesgue measure on $\A$, and if $\A$ is unbounded, e.g.\ $\A=\R^{k_{\A}}$, then we may choose $\nu_{\A}$ to be some probability measure. If $\A=\A_d\times\A_c$ is given by a product of a discrete and a continuous space, then $\nu_{\A}$ can be taken to be a product measure $\nu_{\A_d}\otimes\nu_{\A_c}$. 

Turning to the functional mark space $(\F,\B(\F))$, consider some suitable reference random function/stochastic process
\begin{align}
\label{ReferenceProcess}
X^{\F}&=(X_1^{\F},\ldots,X_k^{\F}):(\Omega,\Sigma,\P)\rightarrow(\F,\B(\F))=(\U^k,\B(\U)^k),
\\
\Omega\ni\omega\mapsto X^{\F}(\omega) &= (X_1^{\F}(\omega),\ldots,X_k^{\F}(\omega)) = \{(X_1^{\F}(t;\omega),\ldots,X_k^{\F}(t;\omega))\}_{t\in\T}\in\U^k=\F,
\nn
\end{align}
where each $X^{\F}(\omega)$ is commonly referred to as a sample path/realisation of $X^{\F}$. 
This random element induces a probability measure 
\bea
\label{e:FunctionalReferenceMeasure}
\nu_{\F}(E) = \P(\{\omega\in\Omega : X^{\F}(\omega)\in E\}),
\quad E\in\B(\F),
\eea
on $\F$, which we will employ as our reference measure on $\F$. 
%; $\nu_{\F}(\cdot)$ is the canonical reference measure under consideration. 
Note that the joint distribution on $(\F^n,\B(\F^n))$ of $n$ independent copies of $X^{\F}$ is given by $\nu_{\F}^n$, the $n$-fold product measure of $\nu_{\F}$ with itself. 
%Also, we may conversely first choose the measure $\nu_{\F}$ and then consider the corresponding process $X^{\F}$.
Moreover, if there is a suitable measure $\nu_{\U}$ on $\U$, we let $\nu_{\F}=\nu_{\U}^k$. 
Specifically, 
%
%
% For reasons which will become clear, 
$\nu_{\F}$, or $X^{\F}$, 
should be chosen such that suitable absolute continuity 
%/change-of-measure 
results can be applied. 
More specifically, the distribution $P_Y$ on $(\F^n,\B(\F^n))$, $n\geq1$, of some stochastic process $Y=\{Y(t)\}_{t\in\T}\in\F^n=(\U^k)^n$ of interest should have some (functional) density/Radon-Nikodym derivative $f_Y$ with respect to $\nu_{\F}^n$, i.e.\ $P_Y(E)=\int_{E}f_Y(f)\nu_{\F}^n(df)=\E_{\nu_{\F}^n}[\1_E f_Y]$, $E\in\B(\F^n)$. 
Note that Kolmogorov's consistency theorem allows us to specify the (abstract) distribution $P_Y$ of $Y$ through its finite dimensional distributions (on $(\R^k)^n$).

In many situations, a natural choice for $\nu_{\F}$ is a Gaussian measure on $\B(\F)$, i.e.\ one corresponding to some Gaussian process $X^{\F}$, or the distribution corresponding to a Markov process $X^{\F}:(\Omega,\Sigma,\P)\rightarrow(\F,\B(\F))$. An often natural choice, which satisfies both of these properties, is the $k$-dimensional standard Brownian motion/Wiener process 
$$
X^{\F}=W=\{W(t)\}_{t\in\T}=\{(W_1(t),\ldots,W_k(t))\}_{t\in\T}\in\F=\U^k
,
$$
which is generated by the corresponding Wiener measure $\mathcal{W}_{\F}$ on $\B(\F)$. In certain cases one speaks of an abstract Wiener space or Cameron-Martin space.
Here issues related to absolute continuity have been extensively studied, and explicit constructions of Radon-Nikodym derivatives involve e.g.\ the Cameron-Martin-Girsanov (change of measure) theorem. 
%; here the choice of function space $\U$ considered is of importance.  
For discussions, overviews and detailed accounts, see e.g.\ \citet{kallenberg2006foundations,rajput1972gaussian,maniglia2004gaussian,Skorohod} and the references therein.

% (see \citet{Skorohod} for a discussion on such densities when $\U$ is the c\`adl\`ag/Skorohod space)

Note that integration of a measurable function $h$ with respect to $\nu$ satisfies $\int_{\Y}h(x,l,f)\nu(d(x,l,f)) = \int_{\X}\int_{\A}\int_{\F}h(x,l,f)\de x \nu_{\A}(dl)\nu_{\F}(df)  = \int_{\X}\int_{\A}\int_{\U^k}h(x,l,f_1,\ldots,f_k)\de x \nu_{\A}(dl)\nu_{\U}(df_1)\cdots\nu_{\U}(df_k)$; whenever the auxiliary marks are (partially) discrete, the integral over $\A$ is (partially) replaced by a sum.

\section{FMPP examples%Mark structures
}\label{SectionMarkStructures}
The class of FMPPs provides a framework to give structure to a series of existing models and it allows for the construction of new important models and modelling frameworks, which have uses in different applications.  Below we provide some examples of explicit mark structures, which may be considered when constructing FMPP models. In Appendix~\ref{SectionClassesSTCFMPP} we further provide examples of classical point process models which are functional marked and in Appendix \ref{SectionApplications} we provide a few (further) examples of applications.

\subsection{Point processes with real valued marks}\label{SectionMPP}

%{\color{red}
Besides the fact that $\Psi_{\X\times\A}$ is already a marked point process with real valued marks, letting each $F_i$ a.s.\ take values in the class $\{f:\T\to\R^k \text{ is constant}\}\subset\F=\U^k$ the functional marks are given by the random vectors $F_i(t) \equiv \xi_i\in\R^k$, $t\in\T$, $i=1,\ldots N$, and we may replace $\Psi$ with the marked point process $\bar{\Psi}=\{(X_i,L_i,\xi_i);\, (X_i,L_i)\in\Psi_{\X\times\A}\}\subset\X\times\A\times\R^k$ with real-valued marks $(L_i,\xi_i)$. When $L_i$ is a discrete random variable which describes the point types (recall Section \ref{SectionComponents} and Appendix~\ref{s:AuxiliaryMarkSpace}), $\bar{\Psi}$ is a multi-type point process with $k$-variate real valued marks.

\subsection{Conditionally deterministic functional marks}
\label{s:DeterministicFunctionalMarks}
It may naturally be the case that $\Psi|\Psi_{\X\times\A}$ is not random, i.e.\ 
$\Psi|\Psi_{\X\times\A}=\{f_1,\ldots,f_N\}$ for some given deterministic functions $f_1,\ldots,f_N\in\F$ (obtained by letting the distribution of $\Psi|\Psi_{\X\times\A}$ be given by a product of Dirac masses on $\F$); in Appendix~\ref{s:MarkSpaceChoices}, we look closer at this scenario. 
One example of this is the \emph{growth-interaction process} \citep{Comas,MateuFMPP,Cronie,CronieSarkka,CronieForest,RenshawComas,RenshawComasMateu,RS1,RS2}, which is one of the models having given rise to a substantial part of the ideas underlying the current construction of FMPPs. In Appendix  \ref{SectionGI} we review the growth-interaction process within the setting of FMPPs and indicate some extensions for it. Note further that some of the other modelling frameworks provided below (partially) also fit into this framework.

\subsection{Marking with random closed sets -- geometric interpretation}\label{SectionDeterministicMarks}

We next illustrate how (spatio-temporal) FMPPs may be used to generate (spatio-temporal) point processes marked by random closed sets.

Consider a (spatio-temporal) FMPP $\Psi$ where the spatial locations $X_i$ are located in some subset of $\R^2$ and $k=1$, i.e.\ $\F=\U$, so that $F_i(t)=F_{i1}(t)\in\R$, $t\in\T$. 
In certain settings, such as in the forestry setting, 
one approach to visualising $\Psi$ is obtained by letting the Euclidean disk/ball 
%$B_{\R^2}[X_i,F_i(t)]=\{x\in\R^2:\|X_i-x\|\leq F_i(t)\}$, 
with centre $X_i$ and radius $F_i(t)$, illustrate the space occupied by the $i$th point of $\Psi$ at time $t\in\T$; we here use the convention that 
a ball is empty 
%$B_{\X}[X_i,r]=\emptyset$ 
if $r\leq0$. 
Now, consider the following temporally evolving random closed set 
%Boolean model 
\citep{SKM}:
\begin{align*}
\Xi(t) &= \bigcup_{i=1}^{N} B_{\X}[X_i,F_i(t)]\subset\R^2, 
\quad
t\in\T,
\\
\Xi &= \int_{\T} \Xi(dt)
= \bigcup_{i=1}^{N} \Xi_i = \bigcup_{i=1}^{N} 
\{(x,t)\in\R^2\times\T : F_i(t)>0, \|X_i-x\|\leq F_i(t)\}
.
\end{align*}
We see that whenever $\supp(F_i)$ is a.s.\ bounded, each \emph{deformed cylinder} $\Xi_i$ is a.s.\ a compact subset of $\R^2\times\R=\R^3$ if $\sup_{t\in\T}F_i(t)<\infty$ a.s.. 
We further note that we may consider the marked point processes $\{(X_i,\Xi_i)\}_{i=1}^N$ and $\{(X_i,B_{\X}[X_i,F_i(t)])\}_{i=1}^N$, which are point processes with marks given by random closed sets. Hence, FMPPs provide a way of defining e.g.\ {\em spatio-temporal Boolean models}.
Figure \ref{FigureGeometry} illustrates a realisation of such a spatio-temporal random closed set $\Xi$. 

\begin{figure}[!htbp]
\begin{center}
\includegraphics[width=0.3\textwidth]{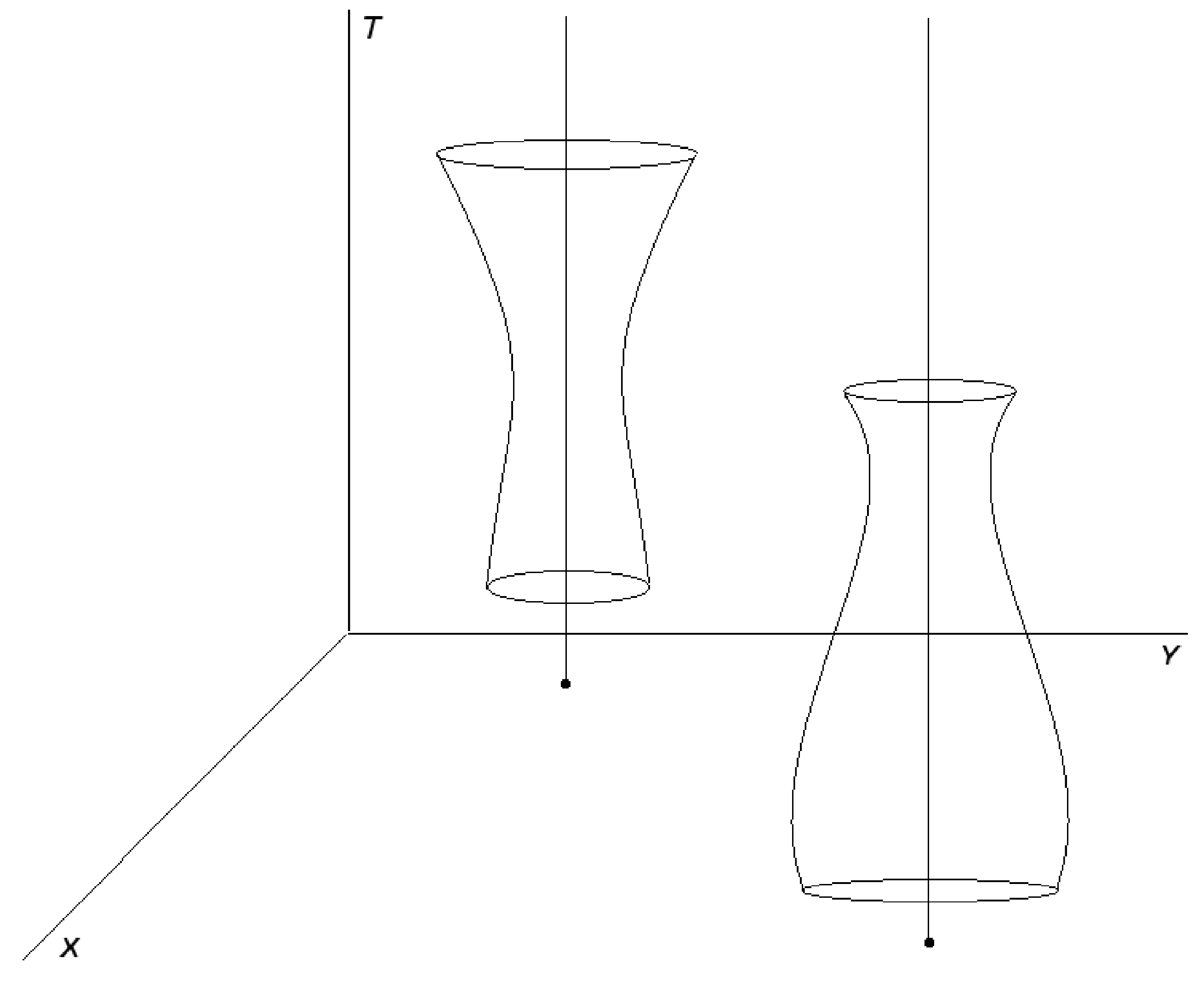}
\caption{An illustration of a realisation of a spatio-temporal random set $\Xi$.}
\label{FigureGeometry}
\end{center}
\end{figure}

The cross section of $\Xi$ at a given time $t$ gives us $\Xi(t)$; 
in the context of e.g.\ forest stand modelling, %we find that 
$\Xi(t)$ gives us the geometric representation of the cross section of the forest stand at time $t$, at some given height (usually \emph{breast height}). 
Note, in addition, that when $\X$ is bounded, depending on the form of the functional marks, we may derive geometric properties such as the expected coverage proportion $\frac{\pi}{|\X|}\sum_{n=0}^{\infty}\sum_{i=1}^{n}\E[F_i(t)^2|\Psi_{\X\times\A}]\P(N=n)$ of $\X$ at time $t$ (provided that the disks do not overlap). 
%{\color{blue}What is the difference between germ-grain model and $\Xi(t)$? See Stoyan et al (1995), page 216.}

The auxiliary marks may clearly play different roles here. E.g., we may consider a multivariate spatio-temporal random closed set $\Xi$ by setting $\A=\{1,\ldots,k_d\}$, $d\geq2$. 
In addition, recalling the discussion on birth times and lifetimes in Section \ref{SectionComponents}, assume that the ground process $\Psi_G=\{(X_i,T_i)\}_{i=1}^N\subset\R^2\times\T$ is a spatio-temporal point process and that each auxiliary mark $L_i$ is a non-negative random variable. Calling $L_i$ the lifetime and $T_i$ the birth time of the $i$th point, by defining the corresponding death time $D_i=T_i+L_i$ and assuming that $F_i(t)|\Psi_{\X\times\A}=0$ 
for all $t\notin[T_i,D_i)$, we obtain that 
\begin{align*}
\Xi(t) &= \bigcup_{i : t\in[T_i,D_i),  F_i(t)>0} B_{\X}[X_i,F_i(t)], 
\quad
t\in\T,
\\
\Xi &= \bigcup_{i=1}^{N} 
\{(x,t)\in\R^2\times[T_i,D_i) : F_i(t)>0, \|X_i-x\|\leq F_i(t)\}
.
\end{align*}
Note that depending on the assumed supports for the lifetimes (bounded/unbounded), we may also absorb $T_i$ into the auxiliary mark.

\subsection{Spatio-temporal geostatistical marking}
\label{SectionGeostatMarkin}

For a marked point process with real valued marks, one often speaks of {\em geostatistical marking/random field marking}. This is the case where, conditionally on $\Psi_G=\{X_i\}_{i=1}^N$, the associated mark is given by $Z_{X_i}$, $i=1,\ldots,N$, where $Z=\{Z_x\}_{x\in\X}$ is some suitable random field. 
%which is independent of $\Psi_G$. 
This may be regarded as \emph{sampling the random field $Z$ at  random locations} $\{X_i\}_{i=1}^{N}$. Note that this definition is slightly more general than the definition usually encountered in the literature, where one typically assumes that $Z$ is independent of $\Psi_G$ \citep{Illian,Baddeley:etal:16}. One setting which falls within this more general definition, where there is not necessarily independence between $Z$ and $\Psi_G$, is intensity-weighted marking; for more details see Section~\ref{SectionIntensityDependentMarks} and \citet{HoStoyan}. 

Within the FMPP-context, the idea of geostatistical marking  may be extended to the case where the marks are coming from a spatio-temporal random field $Z=\{Z_{x}(t)\}_{(x,t)\in\X\times\T}$. 
\begin{definition}\label{DefGeostatMarking}
Consider a spatio-temporal $k$-dimensional random field $Z_{x}(t)\in\R^k$, $(x,t)\in\X\times\T$, $k\geq1$. 
If conditionally on $\Psi_{\X\times\A}$ the functional marks of an FMPP $\Psi$ are given by $F_i = \{Z_{X_i}(t)\}_{t\in\T}\in\F=\U^k$, $i=1,\ldots,N$, we say that $\Psi$ has the \emph{spatio-temporal geostatistical marking} property, or that the spatio-temporal random field $Z$ is {\em sampled at random spatial locations}. 
\end{definition}

%When $k\geq2$ this may be referred to as {\em multivariate spatio-temporal geostatistical marking}.

%We may also refer to this type of marking as \emph{sampling a spatio-temporal random field at random spatial locations}. 

%{\sf \color{red} We should check the notation more and more times for ease of understanding and pursuit}.

%\subsubsection{Randomly generated temporal supports}
To provide an example of a model structure where we exploit spatio-temporal geostatistical marking, consider a multi-type spatio-temporal FMPP $\Psi=\{((X_i,T_i),L_i,F_i)\}_{i=1}^N$, with spatio-temporal ground process $\Psi_G=\{(X_i,T_i)\}_{i=1}^N\subset\X\subset\R^d\times[0,\infty)$ and auxiliary marks $L_i=(L_{1i},L_{2i})\in\A=\A_d\times\A_c\subset\R\times\R$, where $L_{1i}$ is a discrete random variable which takes values in $\A_d=\{1,\ldots,k_{d}\}$, $k_{d}\geq2$, and $L_{2i}$ is a continuous random variable with values in $\A_c=[0,\infty)$ (see Appendix~\ref{s:AuxiliaryMarkSpace}  for details on auxiliary mark spaces). 
In addition, viewing $T_i$ and $L_{2i}$ as the birth time and the lifetime of the $i$th point, respectively, define the death time of the $i$th point as $D_i=T_i+L_{2i}$. 
Given a.s.\ non-negative spatio-temporal random fields $Z_x^1(t),\ldots,Z_x^{k_{d}}(t)\in[0,\infty)$, $(x,t)\in\X\times\T$, i.e.\ one for each possible value of $L_{1i}$ (one for each class label), we let 
$$
F_i(t) = \1_{[T_i,D_i)}(t)\sum_{j=1}^{k_{d}}\1\{L_{1i}=j\}Z_{X_i}^j((t-T_i)\wedge0).
$$
In words, we have a population of $k_d$ different species, where for specie $j\in\{1,\ldots,k_{d}\}$, 
\begin{itemize}
\item the space-time locations are given by $\{(X_i,T_i):(X_i,T_i,(L_{1i},L_{2i}))\in\Psi_{\X\times\A}\cap\X\times\{j\}\times\A_c\}\subset\X$,
\item the size, i.e.\ the functional mark, of individual $i$ of specie $j$ is given by $F_i(t) = \1_{[T_i,D_i)}(t)Z_{X_i}^j((t-T_i)\wedge0)$, where the birth time $T_i$ determines when its size starts changing and its death time $D_i=T_i+L_{2i}$ determines when its size becomes $0$ again.
\end{itemize}
To exemplify further, in the forestry context, $Z_x^j(t)$, $(x,t)\in\X\times\T$, could model the height/diameter at breast hight of the trees of specie $j$.

% \begin{rem}
% Given (dependent) random fields $Z_j=\{Z_j(x,t)\}_{(x,t)\in\X\times\T}$, $j=1,\ldots,l$, when $\Psi$ is multivariate and $\F=\U$, natural constructions/extensions include 
% \begin{itemize}
% \item
% $F_i(t) = \sum_{j=1}^{k_{\A}}\1\{L_{i}=j\}Z_{j}(X_i,t)\in\R$, when $\A=\A_d$,
% \item
% when $L_i=(S_i,T_i,D_i)\in \A=\A_d\times\A_c=\{1,\ldots,k_{\A}\}\times[0,\infty)^2$, with $T_i<D_i$ a.s., let 
% $F_i(t) = \1_{[T_i,D_i)}(t)\sum_{j=1}^{k_{\A}}\1\{S_i=j\}Z_{j}(X_i, t-T_i)$.
% %when $\A=\A_d\times\A_c=\{1,\ldots,k_{\A}\}\times[0,\infty)$.
% \end{itemize}
% \end{rem}
% {\color{red} Unclear, Please explain me using forest data when you are in Ume\aa}.

%
\subsubsection{Spatio-temporal geostatistical prediction with sampling location errors}
%\subsubsection{Geostatistical functional data and geostatistics with uncertainty in the sampling locations}
When observations have been made of a spatio-temporal functional process, at a set of fixed known spatial locations $x_i \in \X$, $i=1,\ldots,n$, one often speaks of \emph{geostatistical functional data}. 
% The class of related data types comprise a broad family of spatially dependent functional data. 
% For a good account of these types of data, the reader is referred to 
% \Citep{DelicadoGiraldoComas,GiraldoDelicadoMateu2010,GiraldoDelicadoMateu2011}. 
More specifically, given some underlying spatio-temporal random field/functional process 
$
\left\{Z_{x}(t):\ x\in \X\subset\R^{d}, t \in\T\right\},
$
we assume that we observe a set of functions, or rather spatially located curves $\left(Z_{x_{1}}(t),\ldots,Z_{x_{n}}(t)\right)$, obtained by sampling $Z$ at locations $x_i \in \X$, $i=1,\ldots,n$, for $t \in\T=[a,b]$, which define the set of functional observations. Each function is assumed to belong to $\F=\U=L_2(\T)$. The class of related data types comprise a broad family of spatially dependent functional data. 
For a good account on these types of data, the reader is referred to 
\citet{DelicadoGiraldoComas,GiraldoDelicadoMateu2010,GiraldoDelicadoMateu2011} and the references therein. 

Consider now the scenario where one would perform some geostatistical analysis 
%, such as spatio-temporal prediction \Citep{GiraldoDelicadoMateu2010}, 
within the setting described above when, in addition, there is uncertainty in the monitoring locations $x_i$, $i=1,\ldots,n$. This positioning error may be the result of imprecise positioning instruments, positional coordinate rounding or human error, e.g.\ map reading \citep{cressie:kornak:03}. 
In the purely spatial setting and in the case of a random field $Z$ sampled at randomly perturbed locations, geostatistical inference has been treated by e.g.\ \citet{chiles:delfiner:12,cressie:kornak:03} to some extent; \citet{cressie:kornak:03} use the terms {\em coordinate positioning model} and {\em feature positioning model}.
% This kind of location error in geostatistical settings may be modelled by the so called coordinate-positioning model (see \citet{cressie:kornak:03} for details).
Note that one here samples the random field/spatial functional process $Z$ at the spatial locations $X_i=x_i+\varepsilon_i$, $i=1,\ldots,n$, where 
$\varepsilon_i$ is a $d$-dimensional random vector \citep{chiles:delfiner:12}.
When each $\varepsilon_i=\varepsilon(x_i)$ is generated through some random error field $\varepsilon(x)$, $x\in\X$, the locations $x_i$ may be dependently perturbed, whereby the sampling locations $X_i$ become spatially dependent; if $\varepsilon(x)$, $x\in\X$, is given by a white noise field, then the locations $x_i$ become independently perturbed.
We see that $\Psi_G=\{X_i\}_{i=1}^n$ constitutes a spatial point process with a fixed number of points $n$; recall that binomial point processes and simple sequential inhibition processes are examples of point processes with predetermined total point counts \citep{VanLieshout}.
Now, an FMPP is obtained by assigning $F_i=\{Z_{X_i} (t)\}$, $t\in\T$, to $X_i\in\Psi_G$ as functional mark. 
%
% for $t\in\T$, by considering $\{Z_{X_i} (t)\}_{i=1}^n$ , $i=1,\ldots,n$,  as functional mark at location $X_i\in\X$, using FMPP framework is the correct one  for analysing this type of data set.
% Here the FMPP framework is the correct one since $\{X_i\}_{i=1}^n$ constitutes a spatial point process.
%
% {\sf\color{red} We assume that the recorded positions $\{X_i\}$, $i=1,\ldots,n$ are obtained by independent random perturbation of a set of known locations $\{x_i\}$, $i=1,\ldots,n$. By this trick while we make the location of events random, but $n$ is still a known fixed value.}
%
Consequently, the geostatistical framework could be extended to incorporate such randomness in the sampling locations. 
\citet{GiraldoDelicadoMateu2010} treat the deterministic case, i.e.\ $\varepsilon_i\equiv0$, and consider the estimator 
$\widetilde{Z_{x_0}}(t)=\widetilde{Z_{x_0}}(t|x_1,\ldots,x_n;\lambda)=\sum_{i=1}^{n}\lambda(x_i,t) Z_{x_i}(t)$, where $\lambda:\X\times\T\rightarrow\R$ belongs to $L_2(\X\times\T)$, for prediction of the marginal random process $\{Z_{x_0}(t)\}_{t\in\T}$, $x_0\in\X$.
Assuming that the locations are in fact random, we obtain the predictor $\widehat{Z_{x_0}}(t)=\sum_{i=1}^{n}\lambda(X_i,t) Z_{X_i}(t)=\sum_{i=1}^{n}\lambda(X_i,t) F_i(t)$ and the associated prediction problem may now be expressed as minimising \citep{GiraldoDelicadoMateu2010}
\begin{align*}
\lambda&\mapsto\E\left[\int_{\T}(\widehat{Z_{x_0}}(t)-Z_{x_0}(t))^2 \de t \right]
=
\E\left[\int_{\T}\left(\sum_{i=1}^{n}\lambda(X_i,t) Z_{X_i}(t) - Z_{x_0}(t)\right)^2 \de t \right]
\\
&=
\int_{\X^n}
%\left.
%\underbrace{
\E^{v_1,\ldots,v_n}
\left[
\int_{\T}
\left(
\widetilde{Z_{x_0}}(t|v_1,\ldots,v_n;\lambda)
%\sum_{i=1}^{n}\lambda(v_i,t) Z_{v_i}(t) 
- Z_{x_0}(t)\right)^2 
%\right|\{X_i=v_i\}_{i=1}^n
\de t
\right]
%}_{=L(\lambda|v_1,\ldots,v_n)}
\rho_G^{(n)}(v_1,\ldots,v_n)\de v_1\cdots\de v_n
% \\
% &=
% \int_{\X^n}
% \E^{v_1,\ldots,v_n}
% \left[\int_{\T}(\widetilde{Z_{x_0}}(t)-Z_{x_0}(t))^2 \de t \right]
% \rho_G^{(n)}(v_1,\ldots,v_n)\de v_1\cdots\de v_n
% &=\E\left[\int_{\T}\left(\sum_{i=1}^{n}\lambda(x_i+\varepsilon_i,t) Z_{x_i+\varepsilon_i}(t) - Z_{x_0}(t)\right)^2 \de t \right]
% ,
% \\
% &\text{s.t. } \sum_{i=1}^n \lambda(x_i+\varepsilon_i,t)=1 \text{ for all } t\in \T,
\end{align*}
with respect to $\lambda:\X\times\T\rightarrow\R$ in $L_2(\X\times\T)$ such that $\sum_{i=1}^n \lambda(x_i+\varepsilon_i,t)=1$ for all $t\in \T$. This follows by the Campbell-Mecke formula and Fubini's theorem. Here $\rho_G^{(n)}$ is the $n$th product density of the ground process $\Psi_G=\{X_i\}_{i=1}^n$ and $\E^{x_1,\ldots,x_n}[\cdot]$ denotes expectation under the $n$-point Palm distribution of $\Psi_G$ (see Section \ref{nFoldPalm}). We interpret $\lambda\mapsto\E^{v_1,\ldots,v_n}
[\int_{\T}(\widetilde{Z_{x_0}}(t|v_1,\ldots,v_n;\lambda)-Z_{x_0}(t))^2 \de t]$ as the function to be minimised under deterministic sampling ($\varepsilon_i\equiv0$), when the spatial sampling locations are given by $v_1,\ldots,v_n\in\X$; we weight this by $\rho_G^{(n)}(v_1,\ldots,v_n)\de v_1\cdots\de v_n$, which may be interpreted as the infinitesimal probability that $\Psi_G$ has points at $v_1,\ldots,v_n$. 
%Hence, the prediction may be carried out by 

%{\color{blue} See section 11.4 of \citet{DVJ2} for random translations of a PP.}

%=
%\min_{\lambda_1,\ldots,\lambda_n\in L_2(\T)} 
%\E\left[\int_{\T}\left(\sum_{i=1}^{n}\lambda_i(t) F_i(t) - Z_{x_0}(t)\right)^2 dt \right]
%\\
%
% &=\min_{\lambda_1,\ldots,\lambda_n\in L_2(\T)} \E\left[\int_{\X}\cdots\int_{\X}\int_{\T}\left(\sum_{i=1}^{n}\lambda_{x_i+y_i}(t) Z_{x_i+y_i}(t)-Z_{x_0}(t)\right)^2 \de t \de y_1\cdots\de y_n \right]
%
% This is equivalent to minimising 
% \begin{align*}
% %\min_{\lambda_1,\ldots,\lambda_n\in L_2(\T)} 
% &\int_{\X^n}  \int_{\T}
% \E\left[\left.\left(\sum_{i=1}^{n}\lambda(v_i,t) Z_{v_i}(t) - Z_{x_0}(t)\right)^2 \right| X_1=v_1,\ldots, X_n=v_n\right] \de t
% \times
% \rho_G^{(n)}(v_1,\ldots,v_n)\, \de v_1\cdots\de v_n
% \\&=,
% \end{align*}
 
% %(see Section \ref{SectionFinitePPs}). 
% Hence, one obtains a geostatistical analysis based on FMPPs. 

\subsection{Constructed functional marks
%LISA functions
}\label{lisa}
Another important class of marks are {\em constructed
marks} which, paraphrasing \citet{Illian}, are marks reflecting the geometries of point configurations in neighbourhoods of the individual points. In particular, they are sometimes used to identify points that are different from the normal points in a point pattern \citep{Illian,StoyanStoyan}. Constructed marks are either numerical or functional and here we consider {\em constructed functional  marks} (CFMs); for further details on constructed numerical marks, see e.g.\ \citet{Illian}. 

A broad class of  CFMs can be obtained by using the idea of LISA ({\em Local Indicators of Spatial Association}) functions. Formally, a LISA function for a point $X_i\in\Psi_G$ is a statistic which describes local dependence with respect to $X_i$. Explicitly, LISA functions, which may be incorporated as functional marks, are constructed through a function $S(\cdot)$ such that $S(h,X_i;\Psi_G\setminus\{X_i\})=F_{i}(h)$, $h\in\T=[0,\infty)$, where $F_i$ (possibly with additional parameters) has sample paths in $\F=\U$. Loosely speaking, $h$ is a given distance which specifies which points $X_j\in\Psi_G\setminus\{X_i\}$ should be included in an $h$-neighbourhood of $X_i$, in order to determine the local $h$-distance dependence \citep{anselin:95}. 
 %These functions can be regarded as constructed functional marks \Citep{MateuLorenzoPorcu}. 

In the context of spatial point processes, \citet{Getis:Franklin:87}
used a local version of a Ripley $K$-function estimator, i.e.\ an estimator of the individual $K$-function at point $X_i\in\Psi_G$, given  by $F_i(h)=K_{X_i}(h)=\Psi_G(B_{\X}[X_i,h]\setminus\{X_i\})=\Psi_G(B_{\X}[X_i,h])-1$, to show that points can exhibit different behaviours when examined at different scales of analysis. 
\citet{CollinsCressie} developed {\em second order product density LISA functions} to examine the behaviour of the individual points in a point pattern in terms of their relation to the neighbouring points at several scales simultaneously. This allows for identifying points with similar neighbourhood structures. These two are examples of CFMs that can be attached to points of a point process to turn it into an FMPP. For more examples of CFMs  in terms of LISA functions, see \citet{Illian} and the references therein.

\subsection{Intensity-dependent marks}
\label{SectionIntensityDependentMarks}
A step forward in the marking of stationary unmarked point processes is to allow the distributions of
the marks to be dependent on the local intensity, as suggested by \citet{HoStoyan,MyllymakiPenttinen} in the context of stationary log-Gaussian Cox processes \citep{MollerSyversveen,Moller}.  This intensity-dependent marking assumes conditional independence to hold for the marks, given the random intensity. 
Heuristically, these models allow the marks to be large/small in areas of low/high point intensity and small/large in areas of high/low point intensity. For instance, in forest stands, where there is spatial competition for resources, small marks would mean that many trees are located close to each other. 
For log-Gaussian Cox processes, intensity-dependent marking leads to a correlation of the marks which is affected by the second order properties of the unmarked Cox process $\Psi_G$. 
The set-up developed in \citet{MyllymakiPenttinen} allows the mean and the variance of the mark distribution to be affected by the local intensity, and this setup has been employed for the marking of log Gaussian Cox processes. 
Here one may test for mark independence as well as for independence between marks and locations \citep{Grabarnik2011,Schlather2004}.

For a spatio-temporal point process $\Psi_G=\{(X_i,T_i)\}_{i=1}^N$ with intensity $\rho_G(\cdot)$ (see Section \ref{SectionPointProcessCharacteristics}), in the current FMPP context we may extend these ideas as follows. 
%, where we make use of the intensity $\lambda_G(\cdot)$ of the ground process $\Psi_G$.
\begin{definition}
A spatio-temporal FMPP $\Psi$ with ground process $\Psi_G=\{(X_i,T_i)\}_{i=1}^N\subset\X\subset\R^{d-1}\times\T$, $d\geq2$, with intensity $\rho_G(\cdot)$, is said to have \emph{spatio-temporal intensity-dependent marks} if, conditionally on $\Psi_G$ and the auxiliary marks, the functional marks $F_i(t)$, $t\in\T$, $i=1,\ldots,N$, are given as functions $t\mapsto h(\rho_G(X_i,t))$, $t\in\T$, $i=1,\ldots,N$, for some (random) function $h:\R\to\R$.
\end{definition} 

For instance, we may have
$$
F_i(t)|\Psi_{\X\times\A}=a+b\rho_G(X_i,t)+\varepsilon(X_i,t),
\quad a,b\in\R,
$$
where $\varepsilon(x,t)$ is a spatio-temporal zero mean Gaussian noise process. This can also be seen as an example of geostatistical marking. Further, note that spatio-temporal intensity-dependent marking falls in the category of conditionally deterministic functional marks 
%(see Section \ref{SectionComponents} and Appendix~\ref{s:FunctionSpaces}) 
if the function $h(\cdot)$ is deterministic. 

% and the corresponding point masses on $(\F^n,\B(\F^n))$, $n\geq1$, may or may not depend on the auxiliary marks. 
% %(recall the discussion in Section \ref{SectionAuxiliaryMarkMeasure}).
% Moreover, as in the above mentioned references, larger marks indicate where and when there is low intensity. 

\section{Moment characteristics for FMPPs}
\label{SectionPointProcessCharacteristics}

Besides illustrating the connections above, the aim of this paper is to consider different statistical approaches which allow us to analyse point pattern data with functional marks. 
For a wide range of summary statistics, the core elements are intensity functions and higher order product density functions. We next consider product densities and intensity reweighted product densities %, so-called correlation functionals, 
for FMPPs. 
%under a few  common assumptions. 
In Appendix \ref{s:MarkSpaceChoices} we look closer at what these entities look like under various auxiliary and functional mark space choices.
%In addition, to construct a marked inhomogenous $K$-function \citep{CronieLieshoutMPP} analogue for FMPPs, we will study the reduced Palm measures for these process. %with details.

%In addition, we consider a further, highly important, building block for statistics for point processes, namely the reduced Palm measures.  %and the Papangelou conditional intensities. 
%The purpose is to create the building blocks needed to be able to consider the {\em marked inhomogeneous $K$-function} of \Citep{CronieLieshoutMPP}, in the context of functional marks. 

%Recall that for both types of processes the mark space is given by $\M=\A\times\F$ and to provide a general notation, which may be used to describe both CFMPPs and STCFMPPs, we write 
%$\Psi = \sum_{(g,m)\in\Psi}\delta_{(g,m)} = \sum_{(g,l,f)\in\Psi}\delta_{(g,l,f)}$, 
%where $x=g\in\G=\X$ in the CFMPP case and $(x,t)=g\in\G=\X\times\T$ in the STCFMPP case.

%Throughout, for different measures constructed, we will use the following measure extension approach. 
%When some set function $\mu(A)$ is defined for the bounded Borel sets $A$ in some Borel space $(\mathcal{X},\B(\mathcal{X}))$, by assuming that $\mu(\cdot)$ is locally finite, $\mu(\cdot)$ becomes a finite measure on the ring of bounded Borel sets. 
%Hereby one may extend $\mu$ to a measure on the whole $\sigma$-algebra $\B(\mathcal{X})$ (see e.g.\ \Citep[Theorem A, p. 54]{Halmos}). 

\subsection{Product densities and intensity functionals}
%, correlation functionals and mark distributions}
%We first consider (factorial) moment measures and the product densities of a (ST)CFMPP $\Psi$. The construction of the product densities paves the way for the construction of certain likelihood functions and summary statistics. %the spatio-temporal (functional) marked $K$-function. 
%

Let $\Psi$ be an FMPP with ground process $\Psi_G$. Given some $n\geq1$ and some measurable functional $h:\Y^n=\X^n\times\A^n\times\F^n\rightarrow[0,\infty)$, 
%following e.g.\ \citet{penttinen:stoyan:89} 
%we assume that
consider
\begin{align}
\label{e:StoyanMeasure}
\alpha_h^{(n)}
%(A_1\times\cdots\times A_n)
&=
\E\left[
\sum\nolimits_{(x_1,l_1,f_1),\ldots,(x_n,l_n,f_n)\in\Psi}^{\neq}
h((x_1,l_1,f_1),\ldots,(x_n,l_n,f_n))
%\prod_{i=1}^n\1\{(x_i,l_i,f_i)\in A_i\}
\right]
.
\end{align}
%exists as a locally finite measure on $A_1\times\cdots\times A_n = (C_1\times D_1\times E_1)\times\cdots\times(C_n\times D_n\times E_n)\in\B(\Y^n) = \B(\X\times\M)^n=\B(\X\times\A\times\F)^n$; 
Here $\sum^{\neq}$ denotes summation over distinct $n$-tuples. We first note that the $n$th order factorial moment measure $\alpha^{(n)}(A_1\times\cdots\times A_n)$ of $\Psi$ is retrieved by letting $h$ be the indicator function for the set $A_1\times\cdots\times A_n = (C_1\times D_1\times E_1)\times\cdots\times(C_n\times D_n\times E_n)\in\B(\Y^n) = \B(\X\times\M)^n=\B(\X\times\A\times\F)^n$. Note further that $\alpha^{(n)}$ coincides with the $n$th order \emph{moment measure}  
\(
\mu^{(n)}(A_1\times\cdots\times A_n) = \E[\Psi(A_1)\cdots\Psi(A_n)]
\)
when $A_1,\ldots,A_n\in\B(\Y)$ are disjoint.

Assume next that the $n$th order \emph{(functional) product density} $\rho^{(n)}$, i.e.\ the Radon-Nikodym derivative of $\alpha^{(n)}$ with respect to the $n$-fold product of the reference measure $\nu$ in (\ref{ReferenceMeasure}) with itself, exists. We have that $\alpha^{(n)}$ and $\rho^{(n)}$ satisfy the following {\em Campbell formula} \citep{SKM}:
\begin{align}
\label{Campbell}
&
\alpha_h^{(n)}
=
\int_{\X\times\A\times\F}\cdots\int_{\X\times\A\times\F} h((x_1,l_1,f_1),\ldots,(x_n,l_n,f_n)) 
\alpha^{(n)}(d((x_1,l_1,f_1),\ldots,(x_n,l_n,f_n)))
\\
&=
\int_{\X\times\A\times\F}\cdots\int_{\X\times\A\times\F} h((x_1,l_1,f_1),\ldots,(x_n,l_n,f_n)) 
\rho^{(n)}((x_1,l_1,f_1),\ldots,(x_n,l_n,f_n)) \prod_{i=1}^{n}
\underbrace{\de x_i\nu_{\A}(dl_i)\nu_{\F}(df_i)}_{=\nu(dx_i\times dl_i\times df_i)}
.
\nn
\end{align}
%for any measurable $h:\Y^n=\X^n\times\A^n\times\F^n\rightarrow[0,\infty)$. 
%It follows that, 
Heuristically, 
$\rho^{(n)}((x_1,l_1,f_1),\ldots,(x_n,l_n,f_n)) \prod_{i=1}^{n} \nu(d(x_i,l_i,f_i))$ is interpreted as the probability of having ground process points in the infinitesimal neighbourhoods $dx_1,\ldots,dx_n\subset\X$ of $x_1,\ldots,x_n$, with associated marks belonging to the infinitesimal neighbourhoods $d(l_1,f_1),\ldots,d(l_n,f_n)\subset\A\times\F$ of the mark locations $(l_1,f_1),\ldots,(l_n,f_n)$.

Turning to the ground process $\Psi_{G}$, 
through $\alpha^{(n)}$ we may define the $n$th order \emph{ground factorial moment measure} $\alpha_{G}^{(n)}(\cdot)=\alpha^{(n)}(\cdot\times\A\times\F)$ and its Radon-Nikodym derivative $\rho_{G}^{(n)}$ with respect to the $n$-fold product $|\cdot|^n$ of the Lebesgue measure $|\cdot|$ with itself, which is called the $n$th order \emph{ground product density}. 
% Note  that by omitting the marks  the Campbell formula~\eqref{Campbell} equals to the Campbell formula of the ground process $\Psi_G$, see e.g.\ \cite{MollerWaagepetersen}. 
Note that by letting the function $h$ in \eqref{Campbell} be a function on $\X$ only, we obtain a Campbell formula for the ground process $\Psi_G$. %\citep{MollerWaagepetersen}.
Moreover, by the existence of $\rho_{G}^{(n)}$ and $\rho^{(n)}$, it follows that \citep{Heinrich2013}
\begin{align}
\label{ProductDensityCFMPP}
\rho^{(n)}((x_1,l_1,f_1),\ldots,(x_n,l_n,f_n)) 
&= 
Q_{x_1,\ldots,x_n}^{\M}((l_1,f_1),\ldots,(l_n,f_n)) 
\rho_{G}^{(n)}(x_1,\ldots,x_n)
\\
&=
Q_{(x_1,l_1),\ldots,(x_n,l_n)}^{\F}(f_1,\ldots,f_n)
Q_{x_1,\ldots,x_n}^{\A}(l_1,\ldots,l_n)
\rho_{G}^{(n)}(x_1,\ldots,x_n)
,
\nn
\end{align}
where 
\bea
\label{AuxMarkDensities}
%Q^{\A,n} &=& \{
Q_{x_1,\ldots,x_n}^{\A}:\A^n\to[0,\infty) 
,%| 
&&
x_1,\ldots,x_n\in\X
%\}
,
\\
\label{FunMarkDensities}
% Q^{\F,n} &=& 
% \{
Q_{(x_1,l_1),\ldots,(x_n,l_n)}^{\F}:\F^n=(\U^k)^n\to[0,\infty),
%|
&&
(x_1,l_1),\ldots,(x_n,l_n)\in\X\times\A
%\}
,
\eea
are densities of the families  
\begin{align}
\label{AuxiliaryMarkDistributions}
P_{x_1,\ldots,x_n}^{\A}(D_1\times\cdots\times D_n)
&=\int_{D_1\times\cdots\times D_n}
Q_{x_1,\ldots,x_n}^{\A}(l_1,\ldots,l_n)
\nu_{\A}(d l_1)\cdots\nu_{\A}(d l_n)
,
\\
%\end{align}
%$P_{x_1,\ldots,x_n}^{\A}(\cdot)$ 
% and 
% \begin{align}
\label{FunctionalMarkDistributions}
P_{(x_1,l_1),\ldots,(x_n,l_n)}^{\F}(E_1\times\cdots\times E_n)
&=
\int_{E_1\times\cdots\times E_n}
Q_{(x_1,l_1),\ldots,(x_n,l_n)}^{\F}(f_1,\ldots,f_n)
\prod_{i=1}^{n}\nu_{\F}(d f_i)
,
\end{align}
%$P_{(x_1,l_1),\ldots,(x_n,l_n)}^{\F}(\cdot)$, 
%$(x_1,l_1),\ldots,(x_n,l_n)\in\X\times\A$, 
$(D_1\times E_1),\ldots,(D_n\times E_n)\in\B(\M)=\B(\A\times\F)$, of (regular) conditional probability distributions. 
% {\color{red}are families of conditional density functions with conditional probability distributions $P_{x_1,\ldots,x_n}^{\A}(\cdot)$ and $P_{(x_1,l_1),\ldots,(x_n,l_n)}^{\F}(\cdot)$, respectively. }
We interpret $Q_{x_1,\ldots,x_n}^{\A}(\cdot)$ as the density of the conditional joint probability distribution %$P_{x_1,\ldots,x_n}^{\A}(\cdot)$ 
% \begin{align}
% \label{AuxiliaryMarkDistributions}
% P_{x_1,\ldots,x_n}^{\A}(D_1\times\cdots\times D_n)
% =\int_{D_1\times\cdots\times D_n}
% Q_{x_1,\ldots,x_n}^{\A}(l_1,\ldots,l_n)
% \nu_{\A}(d l_1)\cdots\nu_{\A}(d l_n)
% \end{align}
of $n$ auxiliary marks in $\A$, given that $\Psi$ indeed has $n$ points at the locations $x_1,\ldots,x_n\in\X$. 
Similarly, $Q_{(x_1,l_1),\ldots,(x_n,l_n)}^{\F}(\cdot)$ is interpreted as the density of the conditional joint probability distribution  %$P_{(x_1,l_1),\ldots,(x_n,l_n)}^{\F}(\cdot)$ 
% \begin{align}
% \label{FunctionalMarkDistributions}
% P_{(x_1,l_1),\ldots,(x_n,l_n)}^{\F}(E_1\times\cdots\times E_n)
% =
% \int_{E_1\times\cdots\times E_n}
% Q_{(x_1,l_1),\ldots,(x_n,l_n)}^{\F}(f_1,\ldots,f_n)
% \prod_{i=1}^{n}\nu_{\F}(d f_i)
% \end{align}
of $n$ functional marks in $\F$, given that $\Psi_G$ has points at the $n$ locations $x_1,\ldots,x_n\in\X$ with attached auxiliary marks $l_1,\ldots,l_n\in\A$. 
Recalling Sections \ref{SectionComponents} and \ref{SectionRefernceMeasures}, 
we see that $P_{(x_1,l_1),\ldots,(x_n,l_n)}^{\F}(\cdot)$ represents the probability distribution on $(\F^n,\B(\F^n))$ of $n$ components of $\Psi|\Psi_{\X\times\A}=\{F_1|\Psi_{\X\times\A},\ldots,F_N|\Psi_{\X\times\A}\}$, which may be seen as an $n$-dimensional random function/stochastic process 
$\{F_1(t)|\Psi_{\X\times\A},\ldots,F_n(t)|\Psi_{\X\times\A}\}_{t\in\T}\subset\F$. 
%$\{(F_1(t),\ldots,F_n(t))\}_{t\in\T}$, 
This distribution is absolutely continuous with respect to the reference measure $\nu_{\F}^n$, i.e.\ the distribution of an $n$-dimensional version of the reference process $X^{\F}$, with density given by (\ref{FunMarkDensities}). 
Note that $\rho^{(n)}$ is (partly) a functional since one of its component, $Q_{(x_1,l_1),\ldots,(x_n,l_n)}^{\F}(\cdot)$, is a functional; 
here, we use the term 'functional' for any mapping which 
%since the mapping $\lambda:\X\times\A\times\F\rightarrow[0,\infty)$ 
takes a function as one of its arguments.
%Since $Q_{(x_1,l_1),\ldots,(x_n,l_n)}^{\F}(\cdot)$ is a functional on $\F^n$, also , as discussed in Section \ref{SectionRefernceMeasures}.
The two regular probability distribution families \eqref{AuxiliaryMarkDistributions} and \eqref{FunctionalMarkDistributions} constitute the so-called $n$-point mark distributions \citep{SKM}:
\begin{align*}
&P_{x_1,\ldots,x_n}^{\M}((D_1\times E_1)\times\cdots\times(D_n\times E_n))
=\int_{D_1\times\cdots\times D_n}
P_{(x_1,l_1),\ldots,(x_n,l_n)}^{\F}(E_1\times\cdots\times E_n)
P_{x_1,\ldots,x_n}^{\A}(d(l_1,\ldots,l_n))
\\
&=\int_{(D_1\times E_1)\times\cdots\times(D_n\times E_n)} Q_{x_1,\ldots,x_n}^{\M}((l_1,f_1),\ldots,(l_n,f_n))
\prod_{i=1}^{n}
\nu_{\A}(dl_i)\nu_{\F}(df_i)
\nonumber
% \\
% &=\int_{D_1\times\cdots\times D_n}
% P_{(x_1,l_1),\ldots,(x_n,l_n)}^{\F}(E_1\times\cdots\times E_n)
% Q_{x_1,\ldots,x_n}^{\A}(l_1,\ldots,l_n)
% \nu_{\A}(\de l_1)\cdots\nu_{\A}(\de l_n)
%\prod_{i=1}^{n}\nu_{\A}(dl_i)
.
\end{align*}

The \emph{intensity measure} is given by $\mu(A) = \mu^{(1)}(A) = \alpha^{(1)}(A) = \E[\Psi(A)]$, $A=C\times D\times E\in\B(\Y)$, and since $\rho^{(1)}$ exists, 
\begin{align}
\label{IntensityMeasure}
\mu(A)
%= \int_{A}\rho^{(1)}(x,l,f) \nu(\de(x,l,f))
= \int_{C\times D\times E}\rho^{(1)}(x,l,f) \de x\nu_{\A}(d l)\nu_{\F}(d f)
= \int_{C\times D\times E}
Q_{(x,l)}^{\F}(f)
Q_{x}^{\A}(l)
\rho_G(x) \de x\nu_{\A}(d l)\nu_{\F}(d f)
\end{align}
and we refer to $$
\rho(x,l,f) = \rho^{(1)}(x,l,f) = Q_{(x,l)}^{\F}(f) Q_{x}^{\A}(l) \rho_G(x)
$$ 
as the {\em intensity functional} of the FMPP $\Psi$. Here $\rho_{G}(\cdot)=\rho_{G}^{(1)}(\cdot)$ is the intensity of the ground process, $\Psi_G$. 

We finally point out that $\rho_G^{(n)}$ and $\rho^{(n)}$ are in fact the intensity function and the intensity functional of the point processes 
\begin{align}
\label{FactorialProcess}
\Psi_G^{n\neq}
&=
\{(x_1,\ldots,x_n)\in\Psi^n : x_i\neq x_j \text{ if } i\neq j\}\subset \X^n,
\\
\Psi^{n\neq}&=\{((x_1,l_1,f_1),\ldots,(x_n,l_n,f_n))\in\Psi^n : (x_i,l_i,f_i)\neq(x_j,l_j,f_j) \text{ if } i\neq j\}
\nonumber
\\
&\stackrel{a.s.}{=}
\{((x_1,l_1,f_1),\ldots,(x_n,l_n,f_n))\in\Psi^n : x_i\neq x_j \text{ if } i\neq j\}\subset (\X\times\A\times\F)^n
,
\nonumber
\end{align}
respectively; the last equality follows since $\Psi$ is a marked point process.

% Hereby the intensity functional $\rho(x,l,f)$ may be expressed as
% \begin{align*}
% \rho(x,l,f) = 
% Q_{x}^{\M}(l,f) \rho_{G}(x)
% =
% Q_{(x,l)}^{\F}(f) Q_{x}^{\A}(l) \rho_{G}(x)
% ,
% \end{align*}
% where 
%
%\begin{definition}
%The \emph{intensity functional} of a (ST)CFMPP $\Psi$ is given by 
%$$
%\lambda(g,l,f) = \rho^{(1)}(g,l,f).
%$$
%\end{definition}

% {\color{red} What does happen in the stationary case? See Baddeley (2010) page 381 in the Handbook of Spatial Statistics.}

% The family $P^{\M,n}=\{P_{x_1,\ldots,x_n}^{\M}(D\times E):(x_1,\ldots,x_n)\in\X^n, D\times E\in\B(\M^n)\}$, 
% \begin{align*}
% &P_{x_1,\ldots,x_n}^{\M}((D_1\times E_1)\times\cdots\times(D_n\times E_n))
% =
% \\
% &=\int_{D_1\times\cdots\times D_n}
% P_{(x_1,l_1),\ldots,(x_n,l_n)}^{\F}(E_1\times\cdots\times E_n)
% P_{x_1,\ldots,x_n}^{\A}(\de(l_1,\ldots,l_n))
% \\
% &=\int_{D_1\times\cdots\times D_n}
% P_{(x_1,l_1),\ldots,(x_n,l_n)}^{\F}(E_1\times\cdots\times E_n)
% Q_{x_1,\ldots,x_n}^{\A}(l_1,\ldots,l_n)
% \nu_{\A}(\de l_1)\cdots\nu_{\A}(\de l_n)
% %\prod_{i=1}^{n}\nu_{\A}(dl_i)
% ,
% \end{align*}
% %
% %$P^{\M,n}$ 
% of regular probability distributions 
% is often referred to as the $n$-point mark distributions. 

\subsection{Correlation functionals}
\label{s:CorrelationFunctionals}
Pair correlation functions, which are not in fact correlations in the usual sense, are valuable tools for studying second order dependence properties of point processes. These may be generalised to arbitrary orders $n\geq2$ to characterise $n$-point interactions between the points of a point process, and here in the FMPP context we will refer to them as correlation functionals.
%As we shall see, correlation functionals, which are not in fact correlations in the usual sense, play a similar role for FMPPs. 
Assuming that $\rho$ and $\rho^{(n)}$, $n\geq1$, exist, the $n$th order {\em correlation functional} is defined as
\begin{align}
\label{nPointCorrelationFunction}
g_{\Psi}^{(n)}((x_1,l_1,f_1),\ldots,(x_n,l_n,f_n)) 
&=
\frac{\rho^{(n)}((x_1,l_1,f_1),\ldots,(x_n,l_n,f_n))}
{\rho(x_1,l_1,f_1)\cdots\rho(x_n,l_n,f_n)}
% \\
% &= 
% \frac{Q_{(x_1,l_1),\ldots,(x_n,l_n)}^{\F}(f_1,\ldots,f_n)}
% {Q_{(x_1,l_1)}^{\F}(f_1)\cdots Q_{(x_n,l_n)}^{\F}(f_n)}
% \frac{Q_{x_1,\ldots,x_n}^{\A}(l_1,\ldots,l_n)}
% {Q_{x_1}^{\A}(l_1)\cdots Q_{x_n}^{\A}(l_n)}
% \frac{\rho_G^{(n)}(x_1,\ldots,x_n)}
% {\rho_G(x_1)\cdots\rho_G(x_n)}
% \nonumber
\\
&=
% \gamma_{(x_1,l_1),\ldots,(x_n,l_n)}^{\F}
% (f_1,\ldots,f_n)
% \gamma_{x_1,\ldots,x_n}^{\A}
% (l_1,\ldots,l_n)
\gamma_{x_1,\ldots,x_n}^{\M}
((l_1,f_1),\ldots,(l_n,f_n))
g_{G}^{(n)}(x_1,\ldots,x_n),
\nn
\end{align}
where 
\begin{align}
\label{nPointMarkDensityRatios}
\gamma_{x_1,\ldots,x_n}^{\M}
((l_1,f_1),\ldots,(l_n,f_n))
&=\gamma_{(x_1,l_1),\ldots,(x_n,l_n)}^{\F}
(f_1,\ldots,f_n)
\gamma_{x_1,\ldots,x_n}^{\A}
(l_1,\ldots,l_n)
,
\\
\gamma_{x_1,\ldots,x_n}^{\A}
(l_1,\ldots,l_n)
&=
\frac{Q_{x_1,\ldots,x_n}^{\A}(l_1,\ldots,l_n)}
{Q_{x_1}^{\A}(l_1)\cdots Q_{x_n}^{\A}(l_n)},
\nonumber
\\
\gamma_{(x_1,l_1),\ldots,(x_n,l_n)}^{\F}
(f_1,\ldots,f_n)
&=
\frac{Q_{(x_1,l_1),\ldots,(x_n,l_n)}^{\F}(f_1,\ldots,f_n)}
{Q_{(x_1,l_1)}^{\F}(f_1)\cdots Q_{(x_n,l_n)}^{\F}(f_n)}
\nonumber
\end{align}
and
\[
g_{G}^{(n)}(x_1,\ldots,x_n)=
\frac{\rho_{G}^{(n)}(x_1,\ldots,x_n)}
{\rho_{G}(x_1)\cdots\rho_{G}(x_n)}
\]
is the $n$th order {\em correlation function} of the ground process, $\Psi_G$. 
Note that $\gamma_{(x_1,l_1),\ldots,(x_n,l_n)}^{\F}(\cdot)$ represents the conditional joint density of $n$ functional marks, given their associated locations and auxiliary marks, divided by the conditional marginal densities of these functional marks, given  their corresponding associated locations and auxiliary marks. An analogous interpretation holds for the second term, but then regarding the auxiliary marks instead and conditioned only on the locations. 
The particular case $n=2$, i.e., 
$
% \begin{align}
% \label{PairCorrelationFunction}
g_{\Psi}^{(2)}((x_1,l_1,f_1),(x_2,l_2,f_2)) 
= 
% \frac{\rho^{(2)}((x_1,l_1,f_1),(x_2,l_2,f_2))}{\rho(x_1,l_1,f_1)\rho(x_2,l_2,f_2)}
% %\\
% %&
% = 
\gamma_{x_1,x_2}^{\M}
((l_1,f_1),(l_2,f_2))
% \gamma_{(x_1,l_1),(x_2,l_2)}^{\F}
% (f_1,f_2)
% \gamma_{x_1,x_2}^{\A}
% (l_1,l_2)
% \frac{Q_{(x_1,l_1),(x_2,l_2)}^{\F}(f_1,f_2)}
% {Q_{(x_1,l_1)}^{\F}(f_1) Q_{(x_2,l_2)}^{\F}(f_2)}
% \frac{Q_{x_1,x_2}^{\A}(l_1,l_2)}
% {Q_{x_1}^{\A}(l_1) Q_{x_2}^{\A}(l_2)}
g_{G}^{(2)}(x_1,x_2)
%\frac{\rho_G^{(2)}(x_1,x_2)}{\rho_G(x_1)\rho_G(x_2)}
% \nn
% \\
% &= \frac{Q_{(x_1,l_1),(x_2,l_2)}^{\F}(f_1,f_2)Q_{x_1,x_2}^{\A}(l_1,l_2)}
% {Q_{(x_1,l_1)}^{\F}(f_1) Q_{x_1}^{\A}(l_1)Q_{(x_2,l_2)}^{\F}(f_2) Q_{x_2}^{\A}(l_2)}
% g_{G}(x_1,x_2)
%\nn
, 
$ %\end{align}  
is referred to as the \emph{pair correlation functional (pcf)} and we note that $g_G^{(2)}(x_1,x_2)=\rho_G^{(2)}(x_1,x_2)/(\rho_G(x_1)\rho_G(x_2))$ is the pair correlation function of the ground process \citep{BaddeleyEtAl,SKM}. 
When $n=2$, the first term on the right hand side in \eqref{nPointCorrelationFunction} may be expressed as $\gamma_{x_1,x_2}^{\A}
(l_1,l_2)Q_{(x_1,l_1),(x_2,l_2)}^{\F}(f_1|f_2)/Q_{(x_1,l_1)}^{\F}(f_1)$, where $Q_{(x_1,l_1),(x_2,l_2)}^{\F}(f_1|f_2)$ represents a conditional density on $\F$ of one functional mark, $F_1$, given another functional mark, $F_2$, as well as the associated locations and auxiliary marks. 

\section{FMPP model structures}
\label{SectionIndependentMarks}

We next look closer at a few structural distributional assumptions and model structures for FMPPs. 
In the context of the auxiliary marks we have already highlighted some effects of imposing different independence assumptions on the marks. 
Here, we mainly focus on two assumptions which will play a role in the statistical analysis: common marginal mark distributions and (location-dependent) independent marking. 
In Appendix~\ref{SectionClassesSTCFMPP} we further provide a few different functional marked classical point process models.

\subsection{Common mark distributions}
An assumption which may be realistic in a variety of different contexts is that the marks are not necessarily independent but they have the same marginal distributions. We next look closer at this setting and we note that the statements below should be understood in an almost everywhere (a.e.) setting.

\begin{definition}
Let $\Psi$ be an FMPP with ground process $\Psi_G$ and consider the following scenarios, defined conditionally on $\Psi_G$. 
\begin{itemize}
\item $\Psi$ has a {\em common (marginal) mark distribution}: The marginal 1-dimensional distributions of all marks $(L_i,F_i)$, $i=1,\ldots,N$, are the same, i.e.\ they do not depend on the spatial locations. Here the 1-point mark distributions 
$P_{x}^{\M}(D\times E) = \int_{D} P_{(x,l)}^{\F}(E) P_{x}^{\A}(dl)$, $x\in\X$, 
$D\times E\in\A\times\F$, 
satisfy
$$
P_{x}^{\M}(D\times E) 
\equiv 
P^{\M}(D\times E)
=
\int_{D\times E} Q^{\M}(l,f) \nu_{\M}(d(l,f))
=
\int_{D\times E}
Q_{l}^{\F}(f) Q^{\A}(l) \nu_{\A}(dl)\nu_{\F}(df),
$$
for some probability measure $P^{\M}(D\times E)$, which has density $Q^{\M}(l,f)=Q_{l}^{\F}(f) Q^{\A}(l)$ 
%\equiv Q_{x}^{\M}(\cdot)$, 
with respect to $\nu_{\M}=\nu_{\A}\otimes\nu_{\F}$. 
This is e.g.\ the case when $\Psi$ is {\em stationary} \citep[Thm 3.5.1.]{SchneiderWeil}; $P^{\M}(\cdot)$ is then commonly referred to as {\em the} mark distribution. 

\item $\Psi$ has a {\em common (marginal) functional mark distribution}: Each $F_i|\Psi_{\X\times\A}\in\Psi|\Psi_{\X\times\A}$, $i=1,\ldots,N$, has the same marginal distribution on $(\F,\B(\F))$, which neither depends on its spatial location nor its auxiliary mark. 
Here $P_{(x,l)}^{\F}\equiv P^{\F}$ and $Q_{(x,l)}^{\F}\equiv Q^{\F}$, $(x,l)\in\X\times\A$.

\end{itemize}

\end{definition}

% When there is a common marginal mark distribution, the 1-point mark distributions 
% $P_{x}^{\M}(D\times E) = \int_{D} P_{(x,l)}^{\F}(E) P_{x}^{\A}(dl)$, $x\in\X$, 
% $D\times E\in\A\times\F$, 
% satisfy
% $$
% P_{x}^{\M}(D\times E) 
% \equiv 
% P^{\M}(D\times E)
% =
% \int_{D\times E} Q^{\M}(l,f) \nu_{\A}(dl)\nu_{\F}(df)
% =
% \int_{D\times E}
% Q_{l}^{\F}(f) Q^{\A}(l) \nu_{\A}(dl)\nu_{\F}(df),
% $$
% for some probability measure $P^{\M}(D\times E)$, which we assume has density $Q^{\M}(l,f)=Q_{l}^{\F}(f) Q^{\A}(l)$,
% with respect to $\nu_{\A}\otimes\nu_{\F}$. 
% {\color{red} On one hand you have said that marks don't depend on locations, on the other hand you have used the subscript $x$. CONTRADICTION! We should rewrite this part in accordance with Illian et al (2008) pages 300-302 for easy readable.}

% One instance where there is a common marginal mark distribution is when $\Psi$ is {\em stationary} \citep[Thm 3.5.1.]{SchneiderWeil}; then $P^{\M}(\cdot)$ is commonly referred to as {\em the mark distribution}. 
Under the assumption of a common mark distribution, it may further be the case that the common mark distribution $P^{\M}$ coincides with the reference measure  $\nu_{\M}=\nu_{\A}\otimes\nu_{\F}$ (so $\nu_{\A}$ and $\nu_{\F}$ must be probability measures), which implies that $Q^{\M}(l,f) = Q_{l}^{\F}(f) Q^{\A}(l)\equiv1$ and the correlation functionals satisfy
\begin{align}
\label{e:CorrelationFunctionalsCommonMark}
g_{\Psi}^{(n)}((x_1,l_1,f_1),\ldots,(x_n,l_n,f_n)) 
&= 
Q_{(x_1,l_1),\ldots,(x_n,l_n)}^{\F}(f_1,\ldots,f_n)
Q_{x_1,\ldots,x_n}^{\A}(l_1,\ldots,l_n)
g_{G}^{(n)}(x_1,\ldots,x_n)
.
\end{align}
E.g., $\nu_{\A}$ may be a Bernoulli distribution with parameter $p\in[0,1]$ and $\A=\A_d=\{0,1\}$, and $\nu_{\F}$ a Wiener measure $\mathcal{W}_{\F}$, whereby (marginally) $L_i$ is a Bernoulli random variable and $F_i$ is a Brownian motion, which are independent of each other.

Under the weaker assumption that $\Psi$ has a common functional mark distribution, 
% we note that the marginal distribution of any $F_i|\Psi_{\X\times\A}\in\Psi|\Psi_{\X\times\A}$, $i=1,\ldots,N$, does not depend either the spatial location or the auxiliary mark of the $i$th point; $P_{(x,l)}^{\F}\equiv P^{\F}$ and $Q_{(x,l)}^{\F}\equiv Q^{\F}$, $(x,l)\in\X\times\A$. 
% Moreover, 
recalling the reference process $X^{\F}$ in (\ref{ReferenceProcess}), which has $\nu_{\F}$ as distribution, when additionally $P^{\F}=\nu_{\F}$ we here obtain that, marginally, each component $F_i|\Psi_{\X\times\A}$, $i=1,\ldots,N$, has the same distribution as $X^{\F}$. 
%
% A slightly weaker assumption than the above is obtained by assuming that only the functional marks have a common mark distribution  and do not depend on auxiliary marks, i.e.\ $P_{(x,l)}^{\F}\equiv P^{\F}$ and $Q_{(x,l)}^{\F}\equiv Q^{\F}$, $(x,l)\in\X\times\A$. 
% Note that a common (marginal) functional mark distribution is obtained as soon as the functional marks do not have an explicit spatial distribution but rather an implicit one, through the dependence of other functional marks; 
To provide an example for this setting, note e.g.\ that for the (stochastic) growth-interaction model, conditionally on $N=1$, i.e.\ $\Psi=\{(X_1,L_1,F_1)\}$, we have that the distribution of $F_1|\Psi_{\X\times\A}=\{F_1(t)|(X_1,L_1)\}_{t\in\T}$ does not change with $(X_1,L_1)$.

%$F_i\stackrel{d}{=}X^{\F}$ for all $i=1,\ldots,N$. 

%we have that $Q^{\F}(\cdot)\equiv1$ and, marginally, each component $F_1|\Psi_{\X\times\A}$, $i=1,\ldots,N$, is given by a copy of the reference stochastic process $X^{\F}$ found in \eqref{ReferenceProcess}.

\begin{rem}
Note that when $\Psi$ has a common functional mark distribution we do not necessarily assume that there is a {\em common (marginal) auxiliary mark distribution}, i.e.\ that $\Psi_{\X\times\A}$ has a common mark distribution. Under such an assumption, all $L_i|\Psi_G$, $i=1,\ldots,N$, have the same marginal distributions, which do not depend on the spatial locations, whereby $P_{x}^{\A}\equiv P^{\A}$ and $Q_{x}^{\A}\equiv Q^{\A}$, $x\in\X$.
%If $\Psi$ has a common auxiliary mark distribution, we have that 
%
% we may also impose the assumption that the auxiliary marks have the same marginal distribution, so that 
Hence, if there is a common auxiliary mark distribution as well as a common functional mark distribution, it follows that $P_{x}^{\M}(D\times E) \equiv P^{\M}(D\times E)=P^{\F}(E)P^{\A}(D)$, $D\times E\in\A\times\F$, $x\in\X$, i.e.\ $L_i$ and $F_i$ are conditionally independent for any $i=1,\ldots,N$. This is a stronger assumption than the assumption of a common mark distribution and it holds e.g.\ when $P^{\M}=\nu_{\M}=\nu_{\A}\otimes\nu_{\F}$.

\end{rem}

\subsection{Location-dependent independent marking and random labelling}
\label{s:IndependentMarking}
We next turn to two common notions of mark independence: location-dependent independent marking and random labelling. 

%Considering the multivariate distributions of the marks, 
\begin{definition}
We say that an FMPP $\Psi$ is {\em (location-dependent) independently marked} if, conditional on its ground process $\Psi_G$, all marks $(L_i, F_i)$, $i=1,\ldots,N$, are independent but not necessarily identically distributed \citep[Definition 6.4.III]{DVJ1}. 

%{\color{red} At the title of this subsection and also in the above definition, we have used the term "Location-dependent independent marking". If we would like to use the term "independent marking" instead, we have two options, either remove the term "Location-dependent" everywhere in the paper and mention in  a suitable place that this also called "Location-dependent independent marking", or after the definition we can mention that hereafter we will use the term "independent marking".}

By further adding the assumption of a common marginal mark distribution to 
independent marking, 
so that the marks become independent and identically distributed as well as independent of the ground process $\Psi_G$, we obtain the definition of {\em random labelling}.
\end{definition}

Hereinafter, we will use the shorter term 'independent marking', thus leaving out the part 'location-dependent', in keeping with \citet{DVJ1}. 
%Consequently, 
Under 
independent marking, 
each mark $(L_i, F_i)$ may depend on its associated spatial location and it follows that
%each mark $(L_i, F_i)$ is independent of $\Psi_G\setminus \{X_i\}$ and of any other marks $(L_j,F_j)$, $i\neq j$, whereby 
\begin{align}
\label{e:IndependentMarking}
P_{x_1,\ldots,x_n}^{\M}((D_1\times E_1)\times\cdots\times(D_n\times E_n))
&=
\prod_{i=1}^n
P_{x_i}^{\M}(D_i\times E_i)
=
\prod_{i=1}^n
\int_{D_i}
P_{(x_i,l_i)}^{\F}(E_i)
P_{x_i}^{\A}(d l_i)
\\
&=\int_{D_1\times E_1}
\cdots
\int_{D_n\times E_n}
\prod_{i=1}^n
\underbrace{
Q_{(x_i,l_i)}^{\F}(f_i)
Q_{x_i}^{\A}(l_i)
}_{=Q_{x_i}^{\M}(l_i,f_i)}
\nu_{\A}(dl_i)\nu_{\F}(df_i)
% \\
% &=
% \int_{D_1\times E_1}
% \cdots
% \int_{D_n\times E_n}
% \prod_{i=1}^n
% Q_{x_i}^{\M}(l_i,f_i) \nu_{\A}(\de l_i)\nu_{\F}(\de f_i)
\nonumber
\end{align}
for any $D_i\times E_i\in\B(\A\times\F)$, $i=1,\ldots,n$, and any $n\geq1$. 
Furthermore, under random labelling, expression \eqref{e:IndependentMarking} reduces to 
$$
\prod_{i=1}^n
P^{\M}(D_i\times E_i)
=
\prod_{i=1}^n
\int_{D_i\times E_i} Q^{\M}(l_i,f_i) 
\nu_{\M}(d(l_i,f_i))
=
\prod_{i=1}^n
\int_{D_i\times E_i} 
Q_{l}^{\F}(f_i) Q^{\A}(l_i) \nu_{\A}(dl_i)\nu_{\F}(df_i)
,
$$
which further reduces to $\prod_{i=1}^n\nu_{\A}(D_i)\nu_{\F}(E_i)$ if the common mark distribution coincides with the reference measure $\nu_{\M}=\nu_{\A}\otimes\nu_{\F}$; this additionally implies that the auxiliary and functional marks are (conditionally) independent of each other. 
Under 
%location-dependent 
independent marking 
it clearly follows that the correlation functionals satisfy
% In particular, under the hypothesis of (conditional) independence among the auxiliary and the functional marks, 
\[
g_{\Psi}^{(n)}((x_1,l_1,f_1),\ldots,(x_n,l_n,f_n)) 
= 
g_{G}^{(n)}(x_1,\ldots,x_n), \quad n\geq1.
\]
Hence, if e.g.\ the pair correlation functional coincides with the pair correlation function of the ground process, then the  auxiliary and functional marks are pairwise conditionally independent.

It is not always the case that one wants to have both the auxiliary and the functional marks being independent. 
We next turn to the case where the functional marks are independent.
%, we may impose different levels of independence. 
\begin{definition}
If all the components of $\Psi|\Psi_{\X\times\A}=\{F_1|\Psi_{\X\times\A},\ldots,F_N|\Psi_{\X\times\A}\}$ are independent, 
we say that $\Psi$ has \emph{(location- and auxiliary mark-dependent) independent functional marks}.

When $\Psi$ has both 
% location- and auxiliary (mark-dependent) 
independent functional marks and a common marginal functional mark distribution, we say that $\Psi$ has \emph{randomly labelled functional marks}.

\end{definition}
Here 
% Note that when $\Psi$ has location- and auxiliary mark-dependent independent functional marks, 
it follows that (recall \eqref{nPointMarkDensityRatios})
\begin{align*}
P_{(x_1,l_1),\ldots,(x_n,l_n)}^{\F}(E_1\times\cdots\times E_n) 
&= \prod_{i=1}^{n} P_{(x_i,l_i)}^{\F}(E_i)
=
\prod_{i=1}^n
\int_{E_i}
Q_{(x_i,l_i)}^{\F}(f_i)
\nu_{\F}(df_i)
, 
\quad 
E_1,\ldots,E_n\in \B(\F),
\\
g_{\Psi}^{(n)}((x_1,l_1,f_1),\ldots,(x_n,l_n,f_n)) 
&=
\gamma_{x_1,\ldots,x_n}^{\A}
(l_1,\ldots,l_n)
g_{G}^{(n)}(x_1,\ldots,x_n)
,
\quad
\quad n\geq1
.
\end{align*}
%for any $n\geq1$. 
Moreover, if $\Psi$ has randomly labelled functional marks then $P^{\F}_{(x,l)}=P^{\F}$ and, if additionally $P^{\F}$ coincides with $\nu_{\F}$, then 
$
P_{(x_1,l_1),\ldots,(x_n,l_n)}^{\F}(E_1\times\cdots\times E_n) 
= \prod_{i=1}^{n} \nu_{\F}(E_i)
$ 
%$E_1,\ldots,E_n\in \B(\F)$, $n\geq1$, 
and the functional marks $F_1,\ldots,F_N$ are independent copies of the reference stochastic process $X^{\F}$ in (\ref{ReferenceProcess}).
% When $\Psi$ has location- and auxiliary mark-dependent independent functional marks, then the correlation functional (\ref{nPointCorrelationFunction}) satisfies (recall \eqref{nPointMarkDensityRatios}) 
% \[
% g_{\Psi}^{(n)}((x_1,l_1,f_1),\ldots,(x_n,l_n,f_n)) 
% =
% % \frac{Q_{x_1,\ldots,x_n}^{\A}(l_1,\ldots,l_n)}
% % {Q_{x_1}^{\A}(l_1)\cdots Q_{x_n}^{\A}(l_n)}
% \gamma_{x_1,\ldots,x_n}^{\A}
% (l_1,\ldots,l_n)
% g_{G}^{(n)}(x_1,\ldots,x_n)
% .
% \]
%Note that when $\Psi$ has randomly labelled functional marks, $F_1,\ldots,F_N$ are independent copies of the reference process $X^{\F}$ in (\ref{ReferenceProcess}). 

Further, given that $\Psi$ has 
%location- and auxiliary mark-dependent 
independent functional marks, if we additionally  assume that the auxiliary marks are conditionally independent, so that 
%the regular probabilities (\ref{AuxMarkProb}) satisfy 
\(
P_{x_1,\ldots,x_n}^{\A}(D_1\times\cdots\times D_n) = \prod_{i=1}^{n} P_{x_i}^{\A}(D_i)
=
\prod_{i=1}^n
\int_{D_i}
Q_{x_i}^{\A}(l_i)
\nu_{\A}(dl_i)
, 
\)
$D_1,\ldots,D_n\in \B(\A)$, 
for any $n\geq1$, we retrieve the classical definition of 
%location-dependent 
independent marking for real valued marks \cite[Definition 6.4.III]{DVJ1}, and consequently that of random labelling by assuming that they are also identically distributed.

\begin{rem}
A weaker form of location- and auxiliary mark-dependent independent functional marking, \emph{conditional independent functional marking}, may be obtained by assuming that 
$$
P_{(x_1,l_1),\ldots,(x_n,l_n)}^{\F}(E_1\times\cdots\times E_n) 
= \prod_{i=1}^{n} P_{(x_1,l_1),\ldots,(x_n,l_n)}^{\F}(E_i),
\quad 
E_1,\ldots,E_n\in \B(\F),
$$
for any $n\geq1$ and some family $\{P_{(x_1,l_1),\ldots,(x_n,l_n)}^{\F}(E):(x_1,l_1),\ldots,(x_n,l_n)\in\X\times\A, E\in\B(\F)\}$ of regular probability distributions. Note that here the distribution of a functional mark may depend on all the spatial locations and auxiliary marks.
\end{rem}

\subsection{Poisson processes}

Poisson processes \citep{DVJ1,SKM}, the most well known point process models, are the benchmark/reference models for  representing lack of spatial interaction and constructing other, more sophisticated models. 
Given a positive locally finite measure $\mu$ on $\B(\Y)=\B(\X\times\A\times\F)$, 
a 
%\emph{(spatio-temporal) c\`adl\`ag
\emph{functional marked Poisson process} $\Psi$, with intensity measure $\mu$, is simply a Poisson process on $\Y$ with the additional assumption that $\Psi_G$ is well-defined. 
%In other words, for any disjoint sets $A_1,\ldots,A_n\in\B(\Y)$, $n\geq1$, the random variables $\Psi(A_1),\ldots,\Psi(A_n)$ are independent and Poisson distributed with means $\mu(A_i)$, $i=1,\ldots,n$ (the latter whenever $A_i$ is bounded). 
When $\Psi$ has a well-defined intensity functional $\rho(\cdot)$, i.e.\ when the intensity measure in (\ref{IntensityMeasure}) satisfies $\mu(A)=\int_{A}\rho(x,l,f)\nu(d(x,l,f))$, it follows that $\rho^{(n)}((x_1,l_1,f_1),\ldots,(x_n,l_n,f_n))=\prod_{i=1}^{n}\rho(x_i,l_i,f_i)$, whereby $g_{\Psi}^{(n)}((x_1,l_1,f_1),\ldots,(x_n,l_n,f_n)) \equiv 1$ 
% \begin{align*}
% \rho^{(n)}((x_1,l_1,f_1),\ldots,(x_n,l_n,f_n))
% &= 
% Q_{x_1,\ldots,x_n}^{\M}((l_1,f_1),\ldots,(l_n,f_n)) 
% \rho_{G}^{(n)}(x_1,\ldots,x_n)
% \\
% &=
% \prod_{i=1}^{n} 
% Q_{(x_i,l_i)}^{\F}(f_i)
% Q_{x_i}^{\A}(l_i)
% \prod_{i=1}^{n} 
% \rho_{G}(x_i)
% =\prod_{i=1}^{n}\rho(x_i,l_i,f_i)
% ,
% \\
% g_{\Psi}^{(n)}((x_1,l_1,f_1),\ldots,(x_n,l_n,f_n)) 
% &=
% \gamma_{x_1,\ldots,x_n}^{\M}
% ((l_1,f_1),\ldots,(l_n,f_n))
% g_{G}^{(n)}(x_1,\ldots,x_n)
% \equiv 1,
% % \\
% % \rho_G^{(n)}(x_1,\ldots,x_n)
% % &= \prod_{i=1}^{n} \rho_{G}(x_i),
% % \\
% % g_G^{(n)}(x_1,\ldots,x_n) 
% % &\equiv 1,
% \end{align*}
for any $n\geq1$. 
Note that, formally, not every (functional marked) Poisson process is actually a marked point process; we may not necessarily have that $\Psi_G$ is a well-defined point process in $\X$ \citep[p.\ 8]{VanLieshout}. That being said, we here clearly have an example of independent marking. 
%whereby the pair correlation functional satisfies $g_{\Psi}((x_1,l_1,f_1),(x_2,l_2,f_2)) = 1$. 
% We note also that through (\ref{RelationPalmPapangelou}), since for all $(x,l,f)\in\Psi$ the Palm measures satisfy $P^{!(g,l,f)}(\cdot)=P(\cdot)$, its Papangelou conditional intensity satisfies $\lambda(g,l,f;\Psi)=\lambda(g,l,f)$. 
When there is a common functional mark distribution, all of the functional marks are given by independent copies of the reference process $X^{\F}$ in \eqref{ReferenceProcess}. In particular, if the reference measure $\nu_{\F}$ is given by a Wiener measure $\mathcal{W}_{\F}$ on $\F$, then the functional marks are iid Brownian motions.
Moreover, when $\Psi$ has a common mark distribution, 
%due to Corollary \ref{CorMarkDistrStationarity}, 
it becomes randomly labelled and 
$
\rho^{(n)}((x_1,l_1,f_1),\ldots,(x_n,l_n,f_n)) = \rho_G^n>0$ if the common mark distribution coincides with $\nu_{\M}$. 

When we condition on $N=n$, we obtain a {\em Binomial point process}, which is simply a random (iid) sample $\{(X_i,L_i,F_i)\}_{i=1}^n$ of size $n$, with density $f(x,l,f)=\rho(x,l,f)/n$.

% where $\rho_{G}(\cdot)$ is the ground intensity function. 
% Note that here $g_G(x_1,x_2)\equiv1$, whereby the pair correlation functional satisfies 
% $g_{\Psi}((x_1,l_1,f_1),(x_2,l_2,f_2)) 
% = Q_{x_1,x_2}^{\A}(l_1,l_2)/(Q_{x_1}^{\A}(l_1)Q_{x_2}^{\A}(l_2))$ if $\Psi$ has location- and auxiliary mark-dependent independent functional marks  
% and 
% $g_{\Psi}((x_1,l_1,f_1),(x_2,l_2,f_2)) = 1$ if in addition $\Psi$ has location-dependent independent auxiliary marked. 

\section{Reference measure %{\color{red} weighted or averaged}
averaged 
reduced Palm distributions}
\label{s:Palm}
In the statistical analysis we 
will need to consider Palm conditioning with respect to a given mark set $(D\times E)\in\B(\A\times\F)$; we interpret this as conditioning on the null-event that there is a point of $\Psi_G$ at a given location, under the assumption that the mark associated to this point belongs to $(D\times E)$. 
To be able to do so, 
% we first consider the $n$-point reduced Palm distributions $\P^{!(x_1,l_1,f_1),\ldots,(x_n,l_n,f_n)}(\Psi\in R)=P^{!(x_1,l_1,f_1),\ldots,(x_n,l_n,f_n)}(R)$, $R\in\NN_{lf}^n$, $n\geq1$, of $\Psi$, which we define through the $n$-point reduced Campbell-Mecke formula: 
% For any measurable functional $h:(\X\times\A\times\F\times  N_{lf})^n\rightarrow[0,\infty)$, 
% \begin{align*} 
% &\E\left[
% \sum\nolimits_{(x_1,l_1,f_1),\ldots,(x_n,l_n,f_n)\in\Psi}^{\neq}
% h((x_1,l_1,f_1),\ldots,(x_n,l_n,f_n),\Psi^n\setminus\{(x_1,l_1,f_1),\ldots,(x_n,l_n,f_n)\})
% \right]
% =
% \\
% =&
% \int_{(\X\times\A\times\F)^n}
% \int_{N_{lf}^n}
% h((x_1,l_1,f_1),\ldots,(x_n,l_n,f_n),\psi)
% P^{!(x_1,l_1,f_1),\ldots,(x_n,l_n,f_n)}(d\psi)
% \times
% \\
% &\times
% \rho^{(n)}((x_1,l_1,f_1),\ldots,(x_n,l_n,f_n))
% \prod_{i=1}^n
% \de x_i\nu_{\A}(dl_i)\nu_{\F}(df_i)
% \\
% =&
% \int_{(\X\times\A\times\F)^n}
% \E^{!(x_1,l_1,f_1),\ldots,(x_1,l_1,f_1))}\left[h((x_1,l_1,f_1),\ldots,(x_n,l_n,f_n),\Psi^n)\right]
% \times
% \\
% &\times
% \rho^{(n)}((x_1,l_1,f_1),\ldots,(x_n,l_n,f_n))
% \prod_{i=1}^n
% \de x_i\nu_{\A}(dl_i)\nu_{\F}(df_i)
% .
% \end{align*}
% Recall that, heuristically, $P^{!(x_1,l_1,f_1),\ldots,(x_1,l_1,f_1)}(\cdot)$ is interpreted as the conditional distribution of $\Psi^n$, the $n$-fold product of $\Psi$ with itself, given that $\Psi^n$ has a point at $((x_1,l_1,f_1),\ldots,(x_n,l_n,f_n))$ which we neglect. 
% Moreover, letting $h((x_1,l_1,f_1),\ldots,(x_n,l_n,f_n),\psi)=h^*(x_1,\ldots,x_n,\psi)\prod_{i=1}^n\1\{(l_i,f_i)\in\A\times\F\}$ for some measurable $h^*:\X\times\A\times\F\times  N_{lf}^n\rightarrow[0,\infty)$, 
%
we follow \citet{MCmarked,CronieLieshoutMPP} and  define the {\em $\nu_{\M}$-averaged reduced Palm distribution} with respect to $(D\times E)\in\B(\A\times\F)$.

\begin{definition}
Given an FMPP $\Psi$, its family 
$\P_{D\times E}^{!x}(\Psi\in\cdot) 
= P_{D\times E}^{!x}(\cdot)$, $x\in\X$, of 
{\em $\nu_{\M}$-averaged reduced Palm distributions} with respect to $(D\times E)\in\B(\A\times\F)$, 
are defined as the probability measures
\beann
\label{MarkedPalm}
% \P_{D\times E}^{!x}(\Psi\in R) 
% = 
P_{D\times E}^{!x}(R) 
= 
\frac{\int_{D\times E} P^{!(x,l,f)}(R) \nu_{\M}(d(l,f))}
{\nu_{\M}(D\times E)}
=
\frac{\int_{D\times E}\E^{!(x,l,f)}[\1_R(\Psi)] \nu_{\A}(dl)\nu_{\F}(df)}{\nu_{\A}(D)\nu_{\F}(E)}
,
\quad R\in\NN_{lf},
\eeann
where 
%$x\in\X$, 
% $\P^{!(x,l,f)}(\Psi\in R)=P^{!(x,l,f)}(R)$, $R\in\NN_{lf}$,
$\P^{!(x,l,f)}(\Psi\in \cdot)=P^{!(x,l,f)}(\cdot)$ 
denotes the {\em reduced Palm distribution of $\Psi$} at $(x,l,f)\in\X\times\A\times\F$.
\end{definition}
Recall that $P^{!(x,l,f)}(R)$, $R\in\NN_{lf}$, may be defined through the \emph{reduced Campbell-Mecke formula} \citep[Section 13.1]{DVJ2}: 
For any measurable functional $h:\X\times\A\times\F\times  N_{lf}\rightarrow[0,\infty)$, 
\begin{align} 
\label{reducedCMMarked}
\E\left[\sum_{(x,l,f)\in\Psi} h(x,l,f,\Psi\setminus\{(x,l,f)\})\right]
&=
\int_{\X\times\A\times\F}\int_{ N_{lf}}
h(x,l,f,\psi)
P^{!(x,l,f)}(d\psi)
\rho(x,l,f)\de x\nu_{\M}(d(l,f))
%\mu(d(x,l,f))
\nn
\\
&=
\int_{\X\times\A\times\F}
\E^{!(x,l,f)}\left[h(x,l,f,\Psi)\right]
\rho(x,l,f)\de x\nu_{\A}(dl)\nu_{\F}(df)
.
\end{align}
Since $P^{!(x,l,f)}(\cdot)$ is the distribution of the {\em reduced Palm process} $\Psi^{!(x,l,f)}$, heuristically, $P^{!(x,l,f)}(\cdot)$ is the conditional distribution of $\Psi$, given that $\Psi$ has a point at $(x,l,f)$ which we neglect. 
Moreover, the probability measure $\P_{D\times E}^{!x}(\cdot)$ has expectation
$$
\E_{D\times E}^{!x}[\cdot]
=
% \int
% h(\psi)
% P_{D\times E}^{!x}(d\psi)
% =
\frac{1}{\nu_{\A}(D)\nu_{\F}(E)}
\int_{D\times E} 
\E^{!(x,l,f)}[\cdot]
% \int
% h(\psi)
% P^{!(x,l,f)}(d\psi) 
\nu_{\A}(dl)\nu_{\F}(df)
$$
by Fubini's theorem.
%Note that without the requirement that $\nu_{\A}$ and $\nu_{\F}$ are finite measures, we cannot ensure that we do not have $P_{D\times E}^{!x}(R)\leq P_{D\times E}^{!x}(N_{lf})=0$ for all $R\in \NN_{lf}$. 
% 

In particular, for a Poisson process on $\X\times\A\times\F$, by Slivnyak's theorem \citep{SKM},
\begin{align}
\label{e:MarkedMeasurePoisson}
\P_{D\times E}^{!x}(\Psi\in\cdot) 
= 
\frac{\int_{D\times E} P(\cdot) \nu_{\A}(dl)\nu_{\F}(df)}
{\nu_{\A}(D)\nu_{\F}(E)}
=
\P(\Psi\in\cdot),
\end{align}
the (unconditional) distribution of $\Psi$. 
Moreover, for a 
multivariate FMPP with $\A=\{1,\ldots,k_d\}$, we obtain 
\[
\P_{\{i\}\times E}^{!x}(\Psi\in\cdot) 
%= P_{\{i\}\times E}^{!x}(\cdot)
= 
\frac{\nu_{\A}(\{i\})\int_{E} P^{!(x,i,f)}(\cdot) \nu_{\F}(df)}
{\nu_{\A}(\{i\})\nu_{\F}(E)}
= 
\frac{\int_{E} P^{!(x,i,f)}(\cdot) \nu_{\F}(df)}
{\nu_{\F}(E)}
,
\quad i\in\A,
\]
i.e.\ the $\nu_{\F}$-averaged reduced Palm distribution of $\Psi_i=\{(x,f):(x,l,f)\in\Psi\cap\X\times\{i\}\times\F\}$ with respect to $E\in\B(\F)$, which is independent of the choice of auxiliary reference measure $\nu_{\A}$.
When $\Psi$ has a common mark distribution which coincides with the reference measure, i.e.\ $P_{x}^{\M}(D\times E)
\equiv P^{\M}(D\times E)=\nu_{\M}(D\times E)=\nu_{\A}(D)\nu_{\F}(E)$, $x\in\X$, we obtain a non-stationary and redcued version of the {\em Palm distribution of $\Psi$ with respect to the mark set $D\times E$} found in \citet[p.\ 135]{SKM}:
\[
P_{D\times E}^{!x}(\cdot) 
= 
\frac{1}{P^{\A}(D)P^{\F}(E)} \int_{D\times E}
P^{!(x,l,f)}(\cdot)
P^{\A}(dl)P^{\F}(df)
=
\frac{\int_{D\times E}
P^{!(x,l,f)}(\cdot)
P^{\A}(dl)P^{\F}(df)}{\int_{D\times E}
P^{!(x,l,f)}(N_{lf})
P^{\A}(dl)P^{\F}(df)}
.
\]
This may now be interpretated as the conditional distribution of $\Psi$, given that it has a point with location $x$ with a mark belonging to $D\times E$. 
Note further that 
% actually has the same interpretation as $P^{!(x,l,f)}(\cdot)$, but with the additional condition that the point with location $x$ has a mark that belongs to $D\times E$
% which we may refer to as the {\em reduced Palm distribution with respect to $D\times E$}; . 
%Note that 
under stationarity we have that $P^{!(x,l,f)}(\cdot)\equiv P^{!(0,l,f)}(\cdot)$ for any $x\in\X=\R^d$ so the reduced Palm distributions with respect to $D\times E$ all satisfy $P_{D\times E}^{!x}(\cdot)\equiv P_{D\times E}^{!0}(\cdot)$.

%and defining the {\em averaged reduced Palm distribution} $\bar P^{!x}(R)=\int_{\A\times\F} Q_x^{\M}(l,f)P^{!(x,l,f)}(R) \nu_{\A}(dl)\nu_{\F}(df)$, $R\in \NN_{lf}$, by letting 

To connect the above distributions to the reduced Palm distributions $P_G^{!x}(\cdot)$, $x\in\X$, of the ground process, 
let $h$ in the reduced Campbell-Mecke formula \eqref{reducedCMMarked} depend only on the ground location and the FMPP: 
\begin{align*} 
\E\left[\sum_{(x,l,f)\in\Psi} h(x,\Psi\setminus\{(x,l,f)\})\right]
&=
\int_{\X}
\int_{ N_{lf}}
h(x,\psi)
\underbrace{
\int_{\A\times\F}
Q_x^{\M}(l,f)
P^{!(x,l,f)}(d\psi)
\nu_{\A}(dl)\nu_{\F}(df)
}_{=\bar P^{!x}(d\psi)}
\rho_G(x)\de x
% \\
% &=
% \int_{\X}
% \int_{ N_{lf}}
% h(x,\psi)
% \bar P^{!x}(d\psi)
% \rho_G(x)\de x
% =
% \int_{\X}
% \bar\E^{!x}[h(x,\Psi)]
% \rho_G(x)\de x
,
\end{align*}
where $\bar P^{!x}(\cdot)$ may be interpreted as an average Palm distribution of $\Psi$, given that it has a point at $x$ with unspecified mark \citep[(13.1.13)]{DVJ2}. 
The measure $\bar P^{!x}(\cdot)$ is a distribution on the space $(N_{lf},\NN_{lf})$ of marked point patterns but by projecting it onto the corresponding measurable space of unmarked point patterns, 
% i.e.\ the collection of all $\{x_1,\ldots,x_n\}\subset\X$ such that $\{(x_1,l_1,f_1),\ldots,(x_n,l_n,f_n)\}\in N_{lf}$ with the corresponding counting measure $\sigma$-algebra, 
we obtain the reduced Palm distribution $P_G^{!x}(\cdot)$ of $\Psi_G$ at $x\in\X$ \citep[p.\ 279]{DVJ2}. For any non-negative and measurable function $h$ on the product of the ground space and the space of all unmarked point patterns, 
\begin{align*} 
\E\left[\sum_{x\in\Psi_G} h(x,\Psi_G\setminus\{x\})\right]
=
\int_{\X}
\E_G^{!x}[h(x,\Psi_G)]
\rho_G(x)\de x
,
\end{align*}
where $\E_G^{!x}[\cdot]$ denotes expectation under $P_G^{!x}(\cdot)$. 
Moreover, when $\Psi$ has a common mark distribution which coincides with the reference measure, we obtain that $P_{\A\times\F}^{!x}(\cdot)=\bar P^{!x}(\cdot)$. Hence, under this assumption, the projection of $P_{\A\times\F}^{!x}(\cdot)$ onto the space of unmarked point patterns is simply $P_G^{!x}(\cdot)$.

\subsection{
%$n$-point 
Higher order 
reduced Palm distributions}\label{nFoldPalm}
Similarly, $n$-point reduced Palm distributions $P^{!(x_1,l_1,f_1),\ldots,(x_n,l_n,f_n)}(\cdot)$ on $(N_{lf}^n,\NN_{lf}^n)$, of arbitrary order $n\geq1$ may be obtained -- they are defined as the reduced Palm distributions of the point processes $\Psi^{n\neq}$, $n\geq1$, in expression 
\eqref{FactorialProcess}. 
The interpretation here is that we instead condition on $\Psi$ having distinct marked points at  $(x_1,l_1,f_1),\ldots,(x_n,l_n,f_n)\in\X\times\A\times\F$, which we neglect. 
The associated reduced Palm process $\Psi^{!(x_1,l_1,f_1),\ldots,(x_n,l_n,f_n)}$, i.e.\ the point process with distribution $P^{!(x_1,l_1,f_1),\ldots,(x_n,l_n,f_n)}(\cdot)$, has intensity function \citep{PalmTutorial}
\begin{align}
\label{e:ReducedPalmIntensity}
\rho^{!(x_1,l_1,f_1),\ldots,(x_n,l_n,f_n)}(x,l,f)
=&\frac{\rho^{(n+1)}((x_1,l_1,f_1),\ldots,(x_n,l_n,f_n),(x,l,f))}
{\rho^{(n)}((x_1,l_1,f_1),\ldots,(x_n,l_n,f_n))}
% \\
% &\times\1\{\rho^{(n)}((x_1,l_1,f_1),\ldots,(x_n,l_n,f_n))>0\}
% \nn
\end{align}
provided that the denominator is positive; it is 0 otherwise.
Note in particular that $\rho^{!(x_1,l_1,f_1)}(x_2,l_2,f_2) = \rho(x_2,l_2,f_2) g^{(2)}_{\Psi}((x_1,l_1,f_1),(x_2,l_2,f_2))$ and sometimes, in the literature this quantity is called {\em conditional intensity} and is interpreted as the intensity at the point $(x_2,l_2,f_2)$ conditional on the information that there is a point at $(x_1,l_1,f_1)$; see e.g.\ \citet[page 57]{Diggle2013}.

Having defined the $n$-point reduced Palm distributions, one may in an analogous fashion define $\nu$-averaged reduced Palm distributions $P_{D_1\times E_1,\ldots,D_n\times E_n}^{!x,\ldots,x_n}$, $x_i\in\X$, with respect to mark sets $(D_i\times E_i)\in\B(\A\times\F)$, $i=1,\ldots,n$, which have an analogous interpretation.

We may similarly define $n$-point reduced Palm distributions $\P^{!x_1,\ldots,x_1}(\Psi_G\in R)$, $R\in\NN_{lf}$, $n\geq1$, for the ground process $\Psi_G$, which are the reduced Palm distributions of $\Psi_G^{n\neq}$ in expression 
\eqref{FactorialProcess}. The interpretation here is that we condition on $\Psi_G$ having points at the distinct locations $x_1,\ldots,x_1\in\X$.

It should finally be mentioned that ordinary (non-reduced) $n$-point Palm distributions of $\Psi$ and $\Psi_G$ may be obtained as 
\begin{align*}
P^{(x_1,l_1,f_1),\ldots,(x_n,l_n,f_n)}(R)
&=
P^{!(x_1,l_1,f_1),\ldots,(x_n,l_n,f_n)}
(\{\psi\cup\{((x_1,l_1,f_1),\ldots,(x_n,l_n,f_n))\}:\psi\in R\}),
\\
P^{x_1,\ldots,x_n}(R)
&=
P^{!x_1,\ldots,x_n}
(\{\psi\cup\{(x_1,\ldots,x_n)\}:\psi\in R\}).
\end{align*}

% We finally mention that there are higher order $n$-point Palm distributions $P^{!(x_1,l_1,f_1),\ldots,(x_1,l_1,f_1)}(R)$, $R\in\NN_{lf}^n$, $n\geq1$, which may be used to define $n$-point $\nu$-averaged reduced Palm distribution with respect to mark sets $(D_i\times E_i)\in\B(\A\times\F)$, $i=1,\ldots,n$: 
% For any measurable functional $h:(\X\times\A\times\F\times  N_{lf})^n\rightarrow[0,\infty)$, 
% \begin{align*} 
% &\E\left[
% \sum\nolimits_{(x_1,l_1,f_1),\ldots,(x_n,l_n,f_n)\in\Psi}^{\neq}
% h((x_1,l_1,f_1),\ldots,(x_n,l_n,f_n),\Psi^n\setminus\{(x_1,l_1,f_1),\ldots,(x_n,l_n,f_n)\})
% \right]
% =
% \\
% &=
% \int_{(\X\times\A\times\F)^n\times N_{lf}^n}
% h(x,l,f,\psi)
% P^{!(x_1,l_1,f_1),\ldots,(x_1,l_1,f_1)}(d(\psi_1,\ldots,\psi_n))
% \rho(x,l,f)\de x\nu_{\A}(dl)\nu_{\F}(df)
% %\mu(d(x,l,f))
% %\nn
% \\
% &=
% \int_{(\X\times\A\times\F)^n}
% \E^{!(x_1,l_1,f_1),\ldots,(x_1,l_1,f_1))}\left[h((x_1,l_1,f_1),\ldots,(x_n,l_n,f_n),\Psi^n)\right]
% \rho(x,l,f)\de x\nu_{\A}(dl)\nu_{\F}(df)
% .
% \end{align*}

\section{Marked intensity reweighted moment stationarity}
\label{s:MIRS}

%{\color{blue}
To be able to treat the summary 
%{\color{blue}
statistics 
%of 
considered 
%}
in this paper, we first have to introduce the notion of {\em $k$th order marked intensity reweighted stationarity ($k$-MIRS)} (cf.\ \citet{CronieLieshoutMPP,Iftimi}). %; we here define it slightly differently than originally done by \citet{CronieLieshoutMPP} and we have added the term 'marked' in their names.
\begin{definition}
\label{def:IRS}
An FMPP $\Psi$ with $\Psi_G\subset\X=\R^d$ 
%and $\underline\rho_{\A\times\F}>0$ 
is called {\em $k$th order marked intensity reweighted stationary ($k$-MIRS)}, $k\in\{1,2,\ldots\}$,  
if $\inf_{(x,l,f)\in\X\times \A\times\F}\rho(x,l,f)>0$ %for any $D\times E\in\B(\A\times\F)$ 
and the $n$th order correlation functionals (recall expression \eqref{nPointCorrelationFunction}) satisfy
\[
g_{\Psi}^{(n)}((x_1,l_1,f_1),\ldots,(x_n,l_n,f_n)) 
\stackrel{a.e.}{=}
g_{\Psi}^{(n)}((x+x_1,l_1,f_1),\ldots,(x+x_n,l_n,f_n))
,\quad n=1,\ldots,k,
\]
for any $x\in\R^d$ 
%and any $n\leq k$ 
(recall that $g_{\Psi}^{(1)}(\cdot)\equiv1$). In particular, the case $k=2$ is referred to as $\Psi$ being 
{\em second order marked intensity reweighted stationary (SOMIRS)} \citep{CronieLieshoutMPP,Iftimi}. 
% If this holds for any order $k$, we say that $\Psi$ is {\em marked intensity reweighted moment stationary (MIRMS)} \citep{CronieLieshoutMPP}.
% {\em $k$th order marked intensity reweighted stationary} and, in particular, the case $k=2$ is referred to as $\Psi$ being 
% {\em marked second order intensity reweighted stationary (MSOIRS)}.
% {\em marked intensity reweighted moment stationary (MIRMS)} 
\end{definition}

Note that, loosely speaking, this  definition essentially states that after having scaled away the effects of the varying intensity, the dependence structure, which is reflected by the product densities, only depends on the distance between the points. Note further that we have implicitly assumed that the product densities up to order $k$ exist. 
%and that the definition of MSOIRS is a marked extension of the definition proposed by \citet{BaddeleyEtAl}. 
A few things are worth pointing out here:
\begin{itemize}
\item For $k$-MIRS to hold, we see that it is required to have both translation invariance of the correlation functions
$g_{G}^{(n)}(\cdot)$, $n\leq k$,
of the ground process, i.e.\ $g_{G}^{(n)}(x_1,\ldots,x_n) \stackrel{a.e.}{=} g_{G}^{(n)}(x+x_1,\ldots,x+x_n)$ for any $x\in\R^d$, as well as 
\begin{align*}
\gamma_{(x_1,l_1),\ldots,(x_n,l_n)}^{\F}
(f_1,\ldots,f_n)
&\stackrel{a.e.}{=}
\gamma_{(x+x_1,l_1),\ldots,(x+x_n,l_n)}^{\F}
(f_1,\ldots,f_n)
,
\quad x\in\R^d, n\leq k,
\\
\gamma_{x_1,\ldots,x_n}^{\A}
(l_1,\ldots,l_n)
&\stackrel{a.e.}{=}
\gamma_{x+x_1,\ldots,x+x_n}^{\A}
(l_1,\ldots,l_n)
,
\end{align*}
% \begin{align*}
% Q_{x_1+x,\ldots,x_n+x}^{\A}(l_1,\ldots,l_n)
% /
% (Q_{x_1+x}^{\A}(l_1)\cdots Q_{x_n+x}^{\A}(l_n))
% &\stackrel{a.e.}{=}
% Q_{x_1,\ldots,x_n}^{\A}(l_1,\ldots,l_n)
% /
% (Q_{x_1}^{\A}(l_1)\cdots Q_{x_n}^{\A}(l_n))
% ,
% \\
% \frac{Q_{(x_1+x,l_1),\ldots,(x_n+x,l_n)}^{\F}(f_1,\ldots,f_n)}
% {Q_{(x_1+x,l_1)}^{\F}(f_1)\cdots Q_{(x_n+x,l_n)}^{\F}(f_n)}
% &\stackrel{a.e.}{=}
% \frac{Q_{(x_1,l_1),\ldots,(x_n,l_n)}^{\F}(f_1,\ldots,f_n)}
% {Q_{(x_1,l_1)}^{\F}(f_1)\cdots Q_{(x_n,l_n)}^{\F}(f_n)}
% ,
% \end{align*}
%for any $x\in\R^d$ and any $n\leq k$, 
for the functions in 
\eqref{nPointMarkDensityRatios}. 
Moreover, assuming that there is a common mark distribution which coincides with the reference measure, the latter reduces to 
$Q_{x_1,\ldots,x_n}^{\A}
(l_1,\ldots,l_n)
\stackrel{a.e.}{=}
Q_{x+x_1,\ldots,x+x_n}^{\A}
(l_1,\ldots,l_n)$ 
and 
$Q_{(x_1,l_1),\ldots,(x_n,l_n)}^{\F}
(f_1,\ldots,f_n)
\stackrel{a.e.}{=}
Q_{(x+x_1,l_1),\ldots,(x+x_n,l_n)}^{\F}
(f_1,\ldots,f_n)$ for any $x\in\R^d$ and any $n\leq k$.

\item Stationarity implies $k$-MIRS for any order $k\geq1$. 
%MIRMS which in turn implies $k$-MIRS.
\item A Poisson process on $\R^d\times\A\times\F$ with intensity bounded away from $0$ is $k$-MIRS for any order $k\geq1$ since $g_{\Psi}^{(n)}(\cdot)\equiv1$ for any $n\geq1$.

\item Under the assumption of independent marking, $k$-MIRS for any order $k\geq1$ and SOMIRS coincide with the definitions of intensity reweighted moment stationarity (IRMS) \citep{VanLieshoutJfunction} and second order intensity reweighted stationarity (SOIRS) \citep{BaddeleyEtAl}, respectively, because under this assumption we have $g_{\Psi}^{(n)}((x_1,l_1,f_1),\ldots,(x_n,l_n,f_n))
= g_G^{(n)}(x_1,\ldots,x_n)$. It should be emphasised that the literature nowhere presents examples of models which are SOIRS but not IRMS \citep{VanLieshoutJfunction,Zhao}; examples include certain Cox, Poisson and Gibbs processes.

\item An illustrative example of a $k$-MIRS for any order $k\geq1$ FMPP is provided by assuming that its ground process is IRMS, the auxiliary marks are independent of the spatial locations and the functional marks are sampled from a suitable stationary spatio-temporal random field. 

\end{itemize}

\section{Summary statistics%: Weighted marked reduced moment measures, marked correlation functionals and marked inhomogeneous $J$-functions
}
\label{s:MarkedCorrelationFunctionals}

Having provided various moment characteristics (Section \ref{SectionPointProcessCharacteristics}) and notions of intensity reweighted moment stationarity (Section \ref{s:MIRS}) for FMPPs, we may now look closer at how these can be exploited to study dependence structures in FMPPs. 
Characterising dependence in marked point processes can, in general, be done in various different ways. There are, however, essentially two main approaches which are studied: 
\begin{enumerate}
\item Spatial interaction between groups of points of $\Psi_G$, based on different classifications of the marks.
\item Dependence between the marks, conditionally on the ground process.
\end{enumerate}
The former approach may be carried out by means of {\em marked second order reduced moment measures/$K$-functions}, {\em marked inhomogeneous nearest neighbour distance distribution functions}, {\em marked inhomogeneous empty space functions} and {\em marked inhomogeneous $J$-functions}, which are defined in \citet{Iftimi,CronieLieshoutMPP,MCmarked}. The last three of these are full-distribution summary statistics and require that the point process is 
%{\color{blue}
$k$-MIRS for any order $k\geq1$, whereas the first two are second order statistics which require SOMIRS.
%
%, is covered in Appendix~\ref{s:Jfunctions}.
%
%
%Turning to the latter approach, 
%We here study different types of weighted correlations among the marks, conditionally on the associated locations of the points.
We here study the second approach and, 
to this end, we define some new summary statistics and, as we shall see, they generalise most existing finite order (marked) inhomogeneous summary statistics.
%, which we refer to as weighted marked $n$th order reduced moment measures and 
% 
%We next turn to analysing the marks conditionally on the associated locations of the points. 

Drawing inspiration from \citet{CronieLieshoutMPP,Iftimi,penttinen:stoyan:89}, we have the following definition.

\begin{definition}
\label{def:K_measure}
Assuming that $2\leq n\leq k$, let $\Psi$ be $k$-MIRS and consider some test function $t=t_n$, by which we mean a measurable mapping $t:\M^n=(\A\times\F)^n\to[0,\infty)$. 

Given some $W\in\B(\R^d)$ with $|W|>0$  and $D\times E\in\B(\M)=\B(\A\times\F)$ with $\nu_{\M}(D\times E)=\nu_{\A}(D)\nu_{\F}(E)>0$, the corresponding {\em $t$-weighted marked $n$th order reduced moment measure} is defined as 
\begin{align}
\label{eq:K_measure}
&\mathcal K_t^{(D\times E)
\bigtimes_{i=1}^{n-1}(D_i\times E_i)
%(D_i\times E_i)_{i=1}^{n-1}
}(C_1\times\cdots\times C_{n-1})
=
\mathcal K_t^{(D\times E)(D_1\times E_1)\cdots(D_{n-1}\times E_{n-1})}(C_1\times\cdots\times C_{n-1})
\nonumber
\\
&=
\E\Bigg[
\sum_{(x,l,f)\in\Psi\cap W\times D\times E}
\mathop{\sum\nolimits\sp{\ne}}_{(x_1,l_1,f_1),\ldots,(x_{n-1},l_{n-1},f_{n-1})\in \Psi\setminus\{(x,l,f)\}}
t((l,f),(l_1,f_1),\ldots,(l_{n-1},f_{n-1}))
\nonumber
\\
&
\times
\frac{1}{\rho(x,l,f)}
\prod_{i=1}^{n-1}
\frac{\1\{x_i-x\in C_i\}\1\{(l_i,f_i)\in D_i\times E_i\}}{\rho(x_i,l_i,f_i)} \Bigg]
\frac{1}{|W|\nu_{\M}(D\times E)
\prod_{i=1}^{n-1}\nu_{\M}(D_i\times E_i)}
%\nonumber
% \\
% =&
% \frac{1}{|W|\nu(D\times E)\prod_{i=1}^{n-1}\nu(D_i\times E_i)}
% \int_{W\times D\times E}
% \E^{!(x,l,f)}\Bigg[
% \mathop{\sum\nolimits\sp{\ne}}_{(x_1,l_1,f_1),\ldots,(x_{n-1},l_{n-1},f_{n-1})\in \Psi}
% t((l,f),(l_1,f_1),\ldots,(l_{n-1},f_{n-1}))
% \times
% \nonumber
% \\
% &\times
% \prod_{i=1}^{n-1}
% \frac{\1 \{x_i-x\in C_i\}\1\{(l_i,f_i)\in D_i\times E_i\}}{\rho(x_i,l_i,f_i)}
% \Bigg]
% \de x
% \nu_{\A}(dl)\nu_{\F}(df)
% \nonumber
% \\
% =&
% \frac{1}{|W|\prod_{i=1}^{n-1}\nu(D_i\times E_i)}
% \int_W
% \E_{D\times E}^{!x}\Bigg[
% \mathop{\sum\nolimits\sp{\ne}}_{(x_1,l_1,f_1),\ldots,(x_{n-1},l_{n-1},f_{n-1})\in \Psi}
% t((L(x),F(x)),(l_1,f_1),\ldots,(l_{n-1},f_{n-1}))
% \times
% \nonumber
% \\
% &\times
% \prod_{i=1}^{n-1}
% \frac{\1 \{x_i-x\in C_i\}\1\{(l_i,f_i)\in D_i\times E_i\}}{\rho(x_i,l_i,f_i)}
% \Bigg]
% \de x
% \nonumber
\end{align} 
% {
% \color{red} The notation $(D_i\times E_i)_{i=1}^{n-1}$  is not  a well-known notation in mathematics. Can we use $\times_{i=1}^{n-1}(D_i\times E_i)$ instead?
% }
for $C_i\times (D_i\times E_i)\in\B(\R^d)\times\B(\M)=\B(\R^d\times\A\times\F)$, $\nu_{\M}(D_i\times E_i)=\nu_{\A}(D_i)\nu_{\F}(E_i)>0$, $i=1,\ldots,n-1$. 
%, where the second equality follows from the Campbell-Mecke formula and $(L(x),F(x))$ denotes the mark associated to the reduced Palm conditioning under $\P_{D\times E}^{!x}(\cdot)$. 
We further refer to 
$$
K_t^{(D\times E)\bigtimes_{i=1}^{n-1}(D_i\times E_i)}(r_1,\cdots,r_{n-1})
=
\mathcal K_t^{(D\times E)\bigtimes_{i=1}^{n-1}(D_i\times E_i)}(B_{\R^d}[0,r_1]\times\cdots\times B_{\R^d}[0,r_{n-1}]),
\quad r_1,\ldots,r_{n-1}\geq0,
$$
as the {\em $t$-weighted $n$th order marked inhomogeneous $K$-function}; when $r_1=\cdots=r_{n-1}=r\geq0$, write $K_t^{(D\times E)\bigtimes_{i=1}^{n-1}(D_i\times E_i)}(r)$.
\end{definition}

The interpretation of \eqref{eq:K_measure} is essentially provided by Lemma \ref{LemmaWeighted} below. Having scaled away the individual intensity contributions of all points of $\Psi$, conditionally on $\Psi$ having a point at an arbitrary location $z\in\R^d$ with associated mark $(L(z),F(z))\in D\times E$, which is neglected (in a reduced Palm sense), \eqref{eq:K_measure} provides the mean of $t((L(z),F(z)),(L_1,F_1),\ldots,(L_{n-1},F_{n-1}))\prod_{i=1}^{n-1}\1\{(L_i,F_i)\in D_i\times E_i\}$, where the locations $X_1,\ldots,X_{n-1}$ of the points associated to $n-1$ other marks $(L_1,F_1),\ldots,(L_{n-1},F_{n-1})$ belong to the respective sets $z+C_i$, $i=1,\ldots,n-1$.

\begin{rem}
We could just as well have chosen to absorb the indicator $\prod_{i=1}^n\1\{(l_i,f_i)\in D_i\times E_i\}$ into the test function $t$ in \eqref{eq:K_measure}. The current choice has been made to emphasise the connection with the summary statistics in \citet{CronieLieshoutMPP,Iftimi}.
\end{rem}

In order to give a feeling for how the mark sets in \eqref{eq:K_measure} may be specified here in the FMPP context, consider a bivariate FMPP, i.e.\ $\A=\{1,2\}$, where $k=1$, so that $F_i:\T\to\R$. Next, let $n=2$ and let $D=\{1\}$, $D_1=\{2\}$, $E=\{f\in\F=\U:\sup_{t\in\T}|f(t)| > c\}$ and $E_1=\{f\in\F=\U:\sup_{t\in\T}|f(t)|\leq c\}$, for some positive constant $c$. Here we would thus restrict the $t$-weighted correlation provided by \eqref{eq:K_measure} to only be between points of different types and, moreover, to be between the two classes of functional marks which either exceed the threshold $c$ or not (see Section \ref{s:TestFunctions} for examples of test functions). For instance, in the forestry context $\A$ would represent the two species under consideration while $c$ would be the threshold diameter at breast height of the trees; if we would instead set $D=D_1=\A$, we would ignore the species and simply study the interaction between large and small trees, irrespective of the trees' species. Hence, we are able to study how large trees affect the survival of small trees, which is something of interest in ecology \citep{platt:etal:88,moeller:etal:16}. 
We emphasize that it should be checked that the chosen sets $E_i$, $i=1,\ldots,n-1$, are indeed measurable, given the chosen function space $(\F,\B(\F))$. 

% measure \eqref{eq:K_measure} essentially provides an intensity reweighted $n$th order $t$-weighted correlation of sort, between $n$ distinct marks, which are taken to belong to $D\times E,D_1\times E_1,\ldots,D_{n-1}\times E_{n-1}$, and have location separation vectors within sets $C_1,\ldots,C_{n-1}$.

We will see that \eqref{eq:K_measure} is closely related to the {\em $n$th order reduced moment measure} of the ground process (cf. \citet[Section 4.1.2]{Moller}),
\begin{align*}
&\mathcal K_G(C_1\times\cdots\times C_{n-1})
=
\frac{1}{|W|}
\E\left[
\sum_{x\in\Psi_G\cap W}
\mathop{\sum\nolimits\sp{\ne}}_{x_1,\ldots,x_{n-1}\in\Psi_G\setminus\{x\}}
\frac{1}{\rho_G(x)}
\prod_{i=1}^{n-1}
\frac{\1 \{x_i-x\in C_i\}}{\rho_G(x_i)} \right]
\\
% &=
% \frac{1}{|W|}
% \int_W
% \int_{x+C_1}\cdots\int_{x+C_{n-1}}
% g_G^{(n)}(x,x_1,\ldots,x_{n-1})
% \de x
% \de x_1\cdots\de x_{n-1}
% \\
% &=
% \frac{1}{|W|}
% \int_W
% \int_{x+C_1}\cdots\int_{x+C_{n-1}}
% g_G^{(n)}(0,x_1-x,\ldots,x_{n-1}-x)
% \de x
% \de x_1\cdots\de x_{n-1}
% \\
&=
\int_{C_1\times\cdots\times C_{n-1}}
g_G^{(n)}(0,x_1,\ldots,x_{n-1})
\de x_1\cdots\de x_{n-1}
=
\frac{1}{|W|}
\int_W
\E_G^{!x}\left[
\mathop{\sum\nolimits\sp{\ne}}_{x_1,\ldots,x_{n-1}\in\Psi_G}
\prod_{i=1}^{n-1}
\frac{\1\{x_i-x\in C_i\}}{\rho_G(x_i)} \right]
\de x
;
\end{align*}
the last two equalities follow from the Campbell formula, the imposed $n$th order intensity reweighted stationarity of $\Psi_G$ (which follows from $\Psi$ being $k$-MIRS) and the Campbell-Mecke formula.
%where the second equality follows from the Campbell formula and the third equality follows from the imposed $n$th order intensity reweighted stationarity of $\Psi_G$ (which follows from $\Psi$ being $k$-MIRS). 
%$g_G^{(n)}(0,x_1,\ldots,x_{n-1})$ 
% In addition, 
% \[
% \mathcal K_G(C_1\times\cdots\times C_{n-1})
% =
% \frac{1}{|W|}
% \int_W
% \E_G^{!x}\left[
% \mathop{\sum\nolimits\sp{\ne}}_{x_1,\ldots,x_{n-1}\in\Psi_G}
% \prod_{i=1}^{n-1}
% \frac{\1\{x_i-x\in C_i\}}{\rho_G(x_i)} \right]
% \de x
% \]
% by the Campbell-Mecke formula; 
An $n$-point generalisation of the inhomogeneous $K$-function $K_{\rm inhom}(r)=K_{\rm inhom}^{(2)}(r)$ of \citet{BaddeleyEtAl} to the $n$th order intensity reweighted stationary setting is obtained by considering 
$K_{\rm inhom}^{(n)}(r) = \mathcal K_G(B_{\R^d}[0,r]^{n-1})$, %$r\geq0$,
% \begin{align*}
% K_{\rm inhom}^{(n)}(r)
% =
% \mathcal K_G(B_{\R^d}[0,r]^{n-1})
% % =
% % \int_{B_{\R^d}[0,r]^{n-1}}
% % g_G^{(n)}(0,x_1,\ldots,x_{n-1})
% % \de x_1\cdots\de x_{n-1}
% ,
% \quad 
% r\geq0,
% \end{align*}
where $B_{\R^d}[0,r]$ denotes the closed origin-centred ball with radius $r\geq0$.
Note further that stationarity implies that 
$$
\alpha_G^{(n)}(C_1\times\cdots\times C_{n-1})
=
\rho_G^{n-1}\mathcal K_G(C_1\times\cdots\times C_{n-1})
=
\E_G^{!0}\left[
\mathop{\sum\nolimits\sp{\ne}}_{x_1,\ldots,x_{n-1}\in\Psi_G}
\1\{x_1\in C_1,\ldots,x_{n-1}\in C_{n-1}\}\right]
$$
and, clearly, in this case 
$K_{\rm inhom}^{(n)}(r)$, $r\geq0$,  %stationarity 
yields an $n$-point generalisation of the $K$-function of \citet{RipleyK}.

% , {\color{red} for some function $g_G^0: (\R^2)^{n-1}\rightarrow [0,\infty)$, setting $h_i=x_i-x, i=1,\ldots, n-1$,
% \begin{align*}
% K_{\rm inhom}^{(n)}(r)
% =
% \mathcal K_G(B_{\R^d}[0,r]^{n-1})
% % =
% % \int_{B_{\R^d}[0,r]^{n-1}}
% % g_G^0(h_1\ldots,h_{n-1})
% % \de h_1\cdots\de h_{n-1}
% ,
% \quad r\geq0,
% \end{align*}
% }
% \begin{align*}
% K_{\rm inhom}^{(n)}(r)
% =
% \mathcal K_G(B_{\R^d}[0,r]^{n-1})
% =
% \int_{B_{\R^d}[0,r]^{n-1}}
% g_G^{(n)}(0,x_1,\ldots,x_{n-1})
% \de x_1\cdots\de x_{n-1}
% ,
% \quad r\geq0,
% \end{align*}
% {\color{red}
% For $n$th order reweighted stationary ground process:
% What does e.g.\, $g_G^{(3)}(0,x_1,x_2)$ mean? Does it mean $g_G^{(3)}(x_1-0,x_2-0)$? We know that for $K_{\rm inhom}$ in \citet{BaddeleyEtAl},  $g_G^{(2)}(0,x_1)=g^{(2)}_G(x_1-0)$.
% }

In addition, we will see in Lemma \ref{LemmaWeighted} below that \eqref{eq:K_measure} is also related to the following kernel (recall \eqref{nPointMarkDensityRatios}).

\begin{definition}
The ($n$th order) {\em intensity reweighted $t$-correlation measure} (at $x_1,\ldots,x_n\in\R^d$) is defined as 
\begin{align}
\label{e:MarkDependenceFunction}
&\kappa_t^{\bigtimes_{i=1}^{n}(D_i\times E_i)}
(x_1,\ldots,x_n)
=
% \\
% =&
\int_{(D_1\times E_1)\times\cdots\times(D_n\times E_n)}
\gamma_{x_1,\ldots,x_n}^{\M}
((l_1,f_1),\ldots,(l_n,f_n))
\nu_t(d(l_1,f_1)\times\cdots\times d(l_n,f_n))
\nonumber
\\
&=
\int_{(D_1\times E_1)\times\cdots\times(D_n\times E_n)}
t((l_1,f_1),\ldots,(l_n,f_n))
\gamma_{x_1,\ldots,x_n}^{\M}
((l_1,f_1),\ldots,(l_n,f_n))
\nu_{\M}(d(l_1,f_1))
\cdots
\nu_{\M}(d(l_n,f_n))
\end{align}
for $x_i\in\R^d$ and $D_i\times E_i\in\B(\A\times\F)$, $i=1,\ldots,n$, where the measure $\nu_t$ is given by
\begin{align*}
&
\nu_t(M)
=
\int_{M}
t((l_1,f_1),\ldots,(l_n,f_n))
\nu_{\M}(d(l_1,f_1))
\cdots
\nu_{\M}(d(l_n,f_n))
,
\quad M\in\B((\A\times\F)^n).
\end{align*}

\end{definition}

% \begin{align}
% \label{e:MarkDependenceFunction}
% &\kappa_t^{(D_i\times E_i)_{i=1}^{n}}
% (x_1,\ldots,x_n)
% =
% \\
% =&
% \int_{(D_1\times E_1)\times\cdots\times(D_n\times E_n)}
% \gamma_{x_1,\ldots,x_n}^{\A\times\F}
% ((l_1,f_1),\ldots,(l_n,f_n))
% \nu_t(d(l_1,f_1)\times\cdots\times d(l_n,f_n))
% \nonumber
% \\
% =&
% \int_{(D_1\times E_1)\times\cdots\times(D_n\times E_n)}
% t((l_1,f_1),\ldots,(l_n,f_n))
% \gamma_{(x_1,l_1),\ldots,(x_n,l_n)}^{\F}
% (f_1,\ldots,f_n)
% \gamma_{x_1,\ldots,x_n}^{\A}
% (l_1,\ldots,l_n)
% % \times
% % \nonumber
% % \\
% % &\times
% \prod_{i=1}^n
% \nu_{\A}(dl_i)\nu_{\F}(df_i)
% \nonumber
% \end{align}

In other words, $\kappa_t^{\cdot}$ is a spatially dependent weighting of $\nu_t(\cdot)$ and we interpret it as the expectation of the random variable $t((L_1, F_1),\ldots,(L_n,F_n))\prod_{i=1}^{n}1\{(L_i, F_i)\in D_i\times E_i\}$, conditionally on $X_i=x_i$, $i=1,\ldots,n$, having scaled away the individual mark density contributions. Note that since $\Psi$ is simple, \eqref{e:MarkDependenceFunction} vanishes whenever $x_i=x_j$ for any $i\neq j$ and, moreover, 
by the imposed $n$th order marked intensity reweighted stationarity, we further have that 
$\kappa_t^{(D_i\times E_i)_{i=1}^{n}}
(x_1,\ldots,x_n) 
=
\kappa_t^{(D_i\times E_i)_{i=1}^{n}}
(x+x_1,\ldots,x+x_n)$ for a.e.\ $x\in\R^d$.
To highlight the connections with \citet{penttinen:stoyan:89}, we refer to 
\begin{align}
\label{e:tCorrelationFunctionalIntRev}
\kappa_t^{\M^n}(x_1,\ldots,x_n)
=\kappa_t^{(\A\times\F)\bigtimes_{i=1}^{n-1}(\A\times\F)}(x_1,\ldots,x_n)
,
\end{align}
i.e.\ \eqref{e:MarkDependenceFunction} with all mark sets set to $\A\times\F$, 
as the ($n$th order) {\em intensity reweighted $t$-correlation functional}; it is interpreted as the expectation of the random variable $t((L_1, F_1),\ldots,(L_n,F_n))$, conditionally on $X_i=x_i$, $i=1,\ldots,n$, having scaled away the individual mark density contributions.

Lemma \ref{LemmaWeighted} below, to which the proof can be found in 
%Section \ref{s:Proofs} in the
Appendix~\ref{s:Proofs}, gives reduced Palm and $\nu_{\M}$-averaged reduced Palm distribution representations of \eqref{eq:K_measure}. It also expresses \eqref{eq:K_measure} through \eqref{e:MarkDependenceFunction} and $\mathcal K_G$, and it tells us that \eqref{eq:K_measure} is independent of the choice $W\in\B(\R^d)$. From a statistical point of view, the main importance of Lemma \ref{LemmaWeighted} is related to non-parametric estimation -- instead of repeated sampling to estimate \eqref{eq:K_measure}, we can simply estimate \eqref{eq:K_measure} by sampling over each point of the point pattern, which is an effect of the imposed $k$-MIRS.

\begin{lemma}\label{LemmaWeighted}
The $t$-weighted marked $n$th order reduced moment measure in \eqref{eq:K_measure} satisfies
\begin{align*}
&
\prod_{i=1}^{n-1}\nu_{\M}(D_i\times E_i)
\mathcal K_t^{(D\times E)\bigtimes_{i=1}^{n-1}(D_i\times E_i)}(C_1\times\cdots\times C_{n-1})
=
\\
=&
\frac{1}{\nu_{\M}(D\times E)}
\int_{C_1\times\cdots\times C_{n-1}}
\kappa_t^{(D\times E)\bigtimes_{i=1}^{n-1}(D_i\times E_i)}
(0,x_1,\ldots,x_{n-1})
\mathcal K_G(dx_1\times\cdots\times dx_{n-1})
\\
=&
\frac{1}{\nu_{\M}(D\times E)}
\int_{D\times E}
\E^{!(z,l,f)}\Bigg[
\mathop{\sum\nolimits\sp{\ne}}_{(x_1,l_1,f_1),\ldots,(x_{n-1},l_{n-1},f_{n-1})\in \Psi}
t((l,f),(l_1,f_1),\ldots,(l_{n-1},f_{n-1}))
\times
\\
&\times
\prod_{i=1}^{n-1}
\frac{\1 \{x_i-z\in C_i\}\1\{(l_i,f_i)\in D_i\times E_i\}}{\rho(x_i,l_i,f_i)}
\Bigg]
\nu_{\A}(dl)\nu_{\F}(df)
\\
=&
\E_{D\times E}^{!z}\Bigg[
\mathop{\sum\nolimits\sp{\ne}}_{(x_1,l_1,f_1),\ldots,(x_{n-1},l_{n-1},f_{n-1})\in \Psi}
t((L(z),F(z)),(l_1,f_1),\ldots,(l_{n-1},f_{n-1}))
\times
\\
&\times
\prod_{i=1}^{n-1}
\frac{\1 \{x_i-z\in C_i\}\1\{(l_i,f_i)\in D_i\times E_i\}}{\rho(x_i,l_i,f_i)}
\Bigg]
\end{align*}
for almost every $z\in\R^d$, where $(L(z),F(z))$ denotes the mark associated with the reduced Palm conditioning under $\P_{D\times E}^{!z}(\cdot)$ 
% and 
% \begin{align}
% \label{e:MarkDependenceFunction}
% &\kappa_t^{(D\times E)(D_i\times E_i)_{i=1}^{n-1}}
% (0,x_1,\ldots,x_{n-1})
% =
% \\
% =&
% \int_{(D\times E)\times(D_1\times E_1)\times\cdots\times(D_{n-1}\times E_{n-1})}
% t((l,f),(l_1,f_1),\ldots,(l_{n-1},f_{n-1}))
% \gamma_{(0,l),(x_1,l_1),\ldots,(x_{n-1},l_{n-1})}^{\F}
% (f,f_1,\ldots,f_{n-1})
% \times
% \nonumber
% \\
% &\times
% \gamma_{0,x_1,\ldots,x_{n-1}}^{\A}
% (l,l_1,\ldots,l_{n-1})
% \nu_{\A}(dl)\nu_{\F}(df)
% \prod_{i=1}^{n-1}
% \nu_{\A}(dl_i)\nu_{\F}(df_i)
% \nonumber
% %
% % \frac{Q_{0,x_1,\ldots,x_{n-1}}^{\A}(l,l_1,\ldots,l_{n-1})}
% % {Q_{0}^{\A}(l)Q_{x_1}^{\A}(l_1)\cdots Q_{x_{n-1}}^{\A}(l_{n-1})}
% % \times
% % \nonumber
% % \\
% % &\times
% % \frac{Q_{(0,l),(x_1,l_1),\ldots,(x_{n-1},l_{n-1})}^{\F}(f,f_1,\ldots,f_{n-1})}
% % {Q_{(0,l)}^{\F}(f)Q_{(x_1,l_1)}^{\F}(f_1)\cdots Q_{(x_{n-1},l_{n-1})}^{\F}(f_{n-1})}
% ,
% \nonumber
% \end{align}
% recalling \eqref{nPointMarkDensityRatios}. 

\end{lemma}

% Applying the reduced Campbell-Mecke formula to \eqref{eq:K_measure}, we obtain that
% \begin{align*}
% &\mathcal K_t^{(D\times E)(D_i\times E_i)_{i=1}^{n-1}}(C_1\times\cdots\times C_{n-1})
% =
% \\
% =&
% \frac{1}{\prod_{i=1}^{n-1}\nu(D_i\times E_i)}
% \frac{1}{\nu(D\times E)}
% \times
% \\
% &\times
% \int_W
% \int_{D\times E}
% \E^{!(x,l,f)}\left[
% \mathop{\sum\nolimits\sp{\ne}}_{(x_1,l_1,f_1),\ldots,(x_{n-1},l_{n-1},f_{n-1})\in \Psi}
% \prod_{i=1}^{n-1}
% \frac{\1 \{x_i-x\in C_i\}\1\{(l_i,f_i)\in D_i\times E_i\}}{\rho(x_i,l_i,f_i)} \right]
% \de x\nu_{\A}(dl)\nu_{\F}(df)
% \\
% =&
% \frac{1}{\prod_{i=1}^{n-1}\nu(D_i\times E_i)}
% \int_C
% \E_{D\times E}^{!x}\left[
% \mathop{\sum\nolimits\sp{\ne}}_{(x_1,l_1,f_1),\ldots,(x_{n-1},l_{n-1},f_{n-1})\in \Psi}
% \prod_{i=1}^{n-1}
% \frac{\1 \{x_i-x\in C_i\}\1\{(l_i,f_i)\in D_i\times E_i\}}{\rho(x_i,l_i,f_i)} \right]
% \de x
% \\
% =&\ldots
% \\
% =&
% \int_C\int_{C_1}\cdots\int_{C_{n-1}}
% \end{align*}

Hence, \eqref{eq:K_measure} may be expressed as a spatial dependence-scaling (reflected by $\mathcal K_G$) of the spatially dependent mark-dependence function 
\eqref{e:MarkDependenceFunction}. 
%$\kappa_t^{(D\times E)(D_i\times E_i)_{i=1}^{n-1}}$. 

Looking closer at Lemma \ref{LemmaWeighted}, we see that normalising \eqref{eq:K_measure} by $\mathcal K_G$ can reveal features of the marking structure, conditionally on the locations.

\begin{definition}
The {\em normalised $t$-weighted marked $n$th order reduced moment measure} is defined as 
\begin{align}
\label{eq:normalisedK_measure}
&\bar{\mathcal K}_t^{(D\times E)\bigtimes_{i=1}^{n-1}(D_i\times E_i)}(C_1\times\cdots\times C_{n-1})
=
\frac{
\mathcal K_t^{(D\times E)\bigtimes_{i=1}^{n-1}(D_i\times E_i)}(C_1\times\cdots\times C_{n-1})
}
{\mathcal K_G(C_1\times\cdots\times C_{n-1})}
\\
&=
\int_{C_1\times\cdots\times C_{n-1}}
\frac{
\kappa_t^{(D\times E)\bigtimes_{i=1}^{n-1}(D_i\times E_i)}
(0,x_1,\ldots,x_{n-1})
}{\nu_{\M}(D\times E)\prod_{i=1}^{n-1}\nu_{\M}(D_i\times E_i)}
\frac{\mathcal K_G(d(x_1,\ldots, x_{n-1}))}
{\mathcal K_G(C_1\times\cdots\times C_{n-1})}
,
\nonumber
\end{align}
where the normalisation of $\mathcal K_G$ in the last term is a probability measure on $C_1\times\cdots\times C_{n-1}$.
\end{definition}

% {\color{red}
% While we can define the normalized version of the measure $\nu_t$ by dividing it by $\nu_1$ (setting test function $t(\cdot)=1$) or equivalently by dividing it by $\nu_{\M}(\A\times\F)^n$ 
% as follows: 
% $$
% \bar\nu_{t}((\A\times\F)^n)=
% \frac{\nu_{t}((\A\times\F)^n)}
% {\nu_{\M}(\A\times\F)^n}
% =
% \int_{(\A\times\F)^n}
% t_n((l_1,f_1),\ldots,(l_n,f_n))\frac{
% \prod_{i=1}^n
% \nu_{\M}(d(l_i,f_i))}{\nu_{\M}(\A\times\F)^n},
% $$
% but it is better to ignore if we really will not use it.

% A further important observation is that when $t_n$ and $t_{n-1}$ are test functions such that ... and $\nu_{\M}$ is a probability measure, the ratio 
% $$
% \frac{\nu_{t_n}((\A\times\F)^n)}
% {\nu_{t_{n-1}}((\A\times\F)^{n-1})}
% =
% \frac{
% \int_{(\A\times\F)^n}
% t_n((l_1,f_1),\ldots,(l_n,f_n))
% \prod_{i=1}^n
% \nu_{\M}(d(l_i,f_i))
% }{
% \int_{(\A\times\F)^{n-1}}
% t_{n-1}((l_1,f_1),\ldots,(l_{n-1},f_{n-1}))
% \prod_{i=1}^{n-1}
% \nu_{\M}(d(l_i,f_i))
% }
% $$ 
% represents a conditional expectation.

% $L$-function?
% }

\subsection{Special cases}
We next study how our new summary statistics behave and reduce under various assumptions on the underlying point process $\Psi$.

\subsubsection{Independent marking and Poisson processes}
When $\Psi$ is independently marked then $\kappa_t^{\cdot}(x_1,\ldots,x_n)$ coincides with $\nu_t(\cdot)$ for any $x_1,\ldots,x_n\in\R^d$, whereby 
%In particular, if $\Psi$ is independently marked then 
\begin{align}
\label{e:MarkDependenceFunctionIndepMarks}
%\kappa_t^{(D\times E)(D_i\times E_i)_{i=1}^{n-1}}(0,x_1,\ldots,x_{n-1})
\bar{\mathcal K}_t^{(D\times E)\bigtimes_{i=1}^{n-1}(D_i\times E_i)}(C_1\times\cdots\times C_{n-1})
=&
\frac{
\nu_t((D\times E)\times(D_1\times E_1)\times\cdots\times(D_{n-1}\times E_{n-1}))
}{\nu_{\M}(D\times E)\prod_{i=1}^{n-1}\nu_{\M}(D_i\times E_i)}
,
% \\
% =&
% \int_{(D\times E)\times(D_1\times E_1)\times\cdots\times(D_{n-1}\times E_{n-1})}
% t((l,f),(l_1,f_1),\ldots,(l_{n-1},f_{n-1}))
% \nu_{\A}(dl)\nu_{\F}(df)
% \prod_{i=1}^{n-1}
% \nu_{\A}(dl_i)\nu_{\F}(df_i)
% \nonumber
\end{align}
i.e., it does not depend on $C_1,\ldots,C_{n-1}$, 
and if $\Psi$ has independent functional marks only then
\begin{align*}
&\kappa_t^{\bigtimes_{i=1}^{n}(D_i\times E_i)}
(x_1,\ldots,x_n)
=
\int_{(D_1\times E_1)\times\cdots\times(D_n\times E_n)}
t((l_1,f_1),\ldots,(l_n,f_n))
%\gamma_{(x_1,l_1),\ldots,(x_n,l_n)}^{\F}(f_1,\ldots,f_n)
\gamma_{x_1,\ldots,x_n}^{\A}
(l_1,\ldots,l_n)
% \times
% \nonumber
% \\
% &\times
\prod_{i=1}^n
\nu_{\A}(dl_i)\nu_{\F}(df_i)
.
\nonumber
\end{align*}

%\subsubsection{Ground Poisson processes}
If we relax the Poisson process assumption slightly to only concern the ground process, we say that an FMPP $\Psi$ is a \emph{FM ground Poisson process}. 
% if 
% its ground process $\Psi_G$ constitutes a simple Poisson process on $\X$. 
By \eqref{ProductDensityCFMPP}, it follows that 
%in light of Proposition \ref{PropositionProdDens}, 
\begin{align*}
\rho^{(n)}((x_1,l_1,f_1),\ldots,(x_n,l_n,f_n))
&=
Q_{(x_1,l_1),\ldots,(x_n,l_n)}^{\F}(f_1,\ldots,f_n)
Q_{x_1,\ldots,x_n}^{\A}(l_1,\ldots,l_n)
\prod_{i=1}^{n}\rho_{G}(x_i),
\\
g_{\Psi}^{(n)}((x_1,l_1,f_1),\ldots,(x_n,l_n,f_n)) 
&=
\gamma_{x_1,\ldots,x_n}^{\M}
((l_1,f_1),\ldots,(l_n,f_n)).
\end{align*}
The latter clearly reduces to $\gamma_{x_1,\ldots,x_n}^{\A}
(l_1,\ldots,l_n)$ when $\Psi$ has independent functional marks and we obtain the usual Poisson case when $\Psi$ has independent marks.
When $\Psi$ is a FM ground Poisson process, 
$\mathcal K_G(C_1\times\cdots\times C_{n-1})=\prod_{i=1}^{n-1}|C_i|$, 
whereby 
\begin{align*}
&\bar{\mathcal K}_t^{(D\times E)\bigtimes_{i=1}^{n-1}(D_i\times E_i)}(C_1\times\cdots\times C_{n-1})
=\\
&
% \mathcal K_t^{(D\times E)(D_i\times E_i)_{i=1}^{n-1}}(C_1\times\cdots\times C_{n-1})
% \\
=
\frac{\int_{C_1\times\cdots\times C_{n-1}}
\kappa_t^{(D\times E)\bigtimes_{i=1}^{n-1}(D_i\times E_i)}
(0,x_1,\ldots,x_{n-1})
\de x_1\cdots \de x_{n-1}}
{ \nu_{\M}(D\times E)\prod_{i=1}^{n-1}\nu_{\M}(D_i\times E_i)\prod_{i=1}^{n-1}|C_i|}
\end{align*}
% \begin{align*}
% &\bar{\mathcal K}_t^{(D\times E)(D_i\times E_i)_{i=1}^{n-1}}(C_1\times\cdots\times C_{n-1})
% =
% \mathcal K_t^{(D\times E)(D_i\times E_i)_{i=1}^{n-1}}(C_1\times\cdots\times C_{n-1})
% \\
% &=
% \int_{C_1\times\cdots\times C_{n-1}}
% \frac{
% \kappa_t^{(D\times E)(D_i\times E_i)_{i=1}^{n-1}}
% (0,x_1,\ldots,x_{n-1})
% }{\nu_{\M}(D\times E)\prod_{i=1}^{n-1}\nu_{\M}(D_i\times E_i)}
% \de x_1\cdots \de x_{n-1}
% \end{align*}
and by additionally assuming independent marking, $\bar{\mathcal K}_t^{(D\times E)\bigtimes_{i=1}^{n-1}(D_i\times E_i)}(C_1\times\cdots\times C_{n-1})$ is given by \eqref{e:MarkDependenceFunctionIndepMarks} and ${\mathcal K}_t^{(D\times E)\bigtimes_{i=1}^{n-1}(D_i\times E_i)}(C_1\times\cdots\times C_{n-1})$ is given by \eqref{e:MarkDependenceFunctionIndepMarks} multiplied by $\prod_{i=1}^{n-1}|C_i|$.
% and a Poisson process on $\R^d\times\A\times\F$ satisfies $\bar{\mathcal K}_t^{(D\times E)(D_i\times E_i)_{i=1}^{n-1}}(C_1\times\cdots\times C_{n-1})
% =
% \nu_t((D\times E)\times(D_1\times E_1)\times\cdots\times(D_{n-1}\times E_{n-1}))$. 

Note that these observations may be used to statistically test independent (functional) marking and Poisson assumptions.

\subsubsection{Common mark distributions}
When we assume that there is a common mark distribution $P^{\M}(\cdot)$, with density $Q^{\M}(l,f) 
=
Q_{l}^{\F}(f) Q^{\A}(l)$, then the product $\prod_{i=1}^n Q^{\M}(l_i,f_i)$ in the denominator of $\gamma_{x_1,\ldots,x_n}^{\M}((l_1,f_1),\ldots,(l_n,f_n))$ may be absorbed into the test function $t(\cdot)$ and we may define the test function
$$
t^*((l,f),(l_1,f_1),\ldots,(l_{n-1},f_{n-1}))
=
t((l,f),(l_1,f_1),\ldots,(l_{n-1},f_{n-1}))
\prod_{i=1}^n Q^{\M}(l_i,f_i) 
%Q_{l_i}^{\F}(f_i) Q^{\A}(l_i)
$$
% {\color{red}together with the ($n$th order) {\em $t^*$-correlation measure}
% \begin{align*}
% &
% \kappa_{t^*}^{(D_i\times E_i)_{i=1}^{n}}
% (x_1,\ldots,x_n)
% =
% % \\
% % =&
% \int_{(D_1\times E_1)\times\cdots\times(D_n\times E_n)}
% Q_{x_1,\ldots,x_n}^{\M}
% ((l_1,f_1),\ldots,(l_n,f_n))
% \nu_t(d(l_1,f_1)\times\cdots\times d(l_n,f_n))
% \\
% =&
% \int_{(D_1\times E_1)\times\cdots\times(D_n\times E_n)}
% t((l_1,f_1),\ldots,(l_n,f_n))
% \underbrace{
% Q_{x_1,\ldots,x_n}^{\M}
% ((l_1,f_1),\ldots,(l_n,f_n))
% \nu_{\M}(d(l_1,f_1))
% \cdots
% \nu_{\M}(d(l_n,f_n))
% }_{=P_{x_1,\ldots,x_n}^{\M}
% (d(l_1,f_1)\times\ldots\times d(l_n,f_n))}
% % \\
% % &=
% % \int_{(D_1\times E_1)\times\cdots\times(D_n\times E_n)}
% % t((l_1,f_1),\ldots,(l_n,f_n))
% % P_{x_1,\ldots,x_n}^{\M}
% % (d(l_1,f_1)\times\ldots\times d(l_n,f_n))
% .
% \end{align*}
%
% }
together with the ($n$th order) {\em $t$-correlation measure}
\begin{align*}
&
k_t^{\bigtimes_{i=1}^{n}(D_i\times E_i)}
(x_1,\ldots,x_n)
=
\kappa_{t^*}^{\bigtimes_{i=1}^{n}(D_i\times E_i)}
(x_1,\ldots,x_n)
=
\\
=&
\int_{(D_1\times E_1)\times\cdots\times(D_n\times E_n)}
Q_{x_1,\ldots,x_n}^{\M}
((l_1,f_1),\ldots,(l_n,f_n))
\nu_t(d(l_1,f_1)\times\cdots\times d(l_n,f_n))
\\
=&
\int_{(D_1\times E_1)\times\cdots\times(D_n\times E_n)}
t((l_1,f_1),\ldots,(l_n,f_n))
\underbrace{
Q_{x_1,\ldots,x_n}^{\M}
((l_1,f_1),\ldots,(l_n,f_n))
\nu_{\M}(d(l_1,f_1))
\cdots
\nu_{\M}(d(l_n,f_n))
}_{=P_{x_1,\ldots,x_n}^{\M}
(d(l_1,f_1)\times\ldots\times d(l_n,f_n))}
% \\
% &=
% \int_{(D_1\times E_1)\times\cdots\times(D_n\times E_n)}
% t((l_1,f_1),\ldots,(l_n,f_n))
% P_{x_1,\ldots,x_n}^{\M}
% (d(l_1,f_1)\times\ldots\times d(l_n,f_n))
.
\end{align*}
We interpret $k_t^{\bigtimes_{i=1}^{n}(D_i\times E_i)}
(x_1,\ldots,x_n)
\stackrel{a.e.}{=}
k_t^{\bigtimes_{i=1}^{n}(D_i\times E_i)}
(z+x_1,\ldots,z+x_n)$, $z\in\R^d$, as the expectation of the random variable $t((L_1, F_1),\ldots,(L_n,F_n))\prod_{i=1}^n\1\{(L_i, F_i)\in\A\times\F\}$, conditionally on $X_i=x_i$, $i=1,\ldots,n$, and when $D_i\times E_i=\A\times\F$, $i=1,\ldots,n$, it yields the {\em $t$-correlation functional} $k_t^{\M^n}
(x_1,\ldots,x_n)$, which is an $n$-point FMPP version of the correlation functions in \citet{penttinen:stoyan:89}.  
% setting $h((x_1,l_1,f_1),\ldots,(x_n,l_n,f_n))=\1\{x_1\in C_1,\ldots,x_n\in C_n\}
% t((l_1,f_1),\ldots,(l_n,f_n))$ in equations \eqref{e:StoyanMeasure} and \eqref{Campbell}, we obtain
% \begin{align*}
% \alpha_t^{(n)}(C_1\times\cdots\times C_n)
% =
% \alpha_h^{(n)}
% =
% \int_{C_1}\cdots\int_{C_n}
% k_t^{\M^n}(x_1,\ldots,x_n)
% \rho_G^{(n)}(x_1,\ldots,x_n)
% \de x_1\cdots\de x_n
% \end{align*}
% for $C_1,\ldots,C_n\in\B(\R^d)$. 
Moreover, 
\begin{align}
\label{eq:K_measure_common_mark}
&
\mathcal K_{t^*}^{(D\times E)\bigtimes_{i=1}^{n-1}(D_i\times E_i)}(C_1\times\cdots\times C_{n-1})
=
\\
=&
\frac{1}{\prod_{i=1}^{n-1}\nu_{\M}(D_i\times E_i)}
\E_{D\times E}^{!z}\Bigg[
\mathop{\sum\nolimits\sp{\ne}}_{(x_1,l_1,f_1),\ldots,(x_{n-1},l_{n-1},f_{n-1})\in \Psi}
t((L(z),F(z)),(l_1,f_1),\ldots,(l_{n-1},f_{n-1}))
\times
\nonumber
\\
&\times
\prod_{i=1}^{n-1}
\frac{\1 \{x_i-z\in C_i\}\1\{(l_i,f_i)\in D_i\times E_i\}}{\rho_G(x_i)}
\Bigg] 
\nonumber
\end{align}
for almost any $z\in\R^d$, 
where (recalling the observations in Section \ref{s:Palm}) $\E_{D\times E}^{!z}[\cdot]$ now properly may be interpreted as a reduced Palm expectation, conditionally on the reduced Palm point having a mark belonging to $D\times E$. 
Note that the connection between the correlation functions in \citet{penttinen:stoyan:89} and Palm distributions has been mentioned (without additional details) by \citet[page 134]{SKM}. 

When the common marginal mark distribution $P^{\M}(\cdot)$ coincides with the reference measure $\nu_{\M}(\cdot)$, so that $Q^{\M}(\cdot) 
\equiv1$ and $\rho(x,l,f)=\rho_G(x)$, we have that $t(\cdot)=t^*(\cdot)$ and 
\begin{align*}
&k_t^{\bigtimes_{i=1}^{n}(D_i\times E_i)}
(x_1,\ldots,x_n)
=
\kappa_t^{\bigtimes_{i=1}^{n}(D_i\times E_i)}
(x_1,\ldots,x_n)
=
\\
=&
\int_{(D_1\times E_1)\times\cdots\times(D_n\times E_n)}
t((l_1,f_1),\ldots,(l_n,f_n))
Q_{x_1,\ldots,x_n}^{\M}
((l_1,f_1),\ldots,(l_n,f_n))
P^{\M}(d(l_1,f_1))
\cdots
P^{\M}(d(l_n,f_n))
\\
=&\E_{P^{\M}}\left[
t((L_1^*, F_1^*),\ldots,(L_n^*,F_n^*))
Q_{x_1,\ldots,x_n}^{\M}
((L_1^*,F_1^*),\ldots,(L_n^*,F_n^*))
\prod_{i=1}^{n}\1\{(L_i^*,F_i^*)\in D_i\times E_i\}
\right]
,
\end{align*}
where $(L_1^*,F_1^*),\ldots,(L_n^*,F_n^*)$ are iid random elements in $\A\times\F$ and $\E_{P^{\M}}[\cdot]$ denotes expectation under their common distribution $P^{\M}=\nu_{\M}=\nu_{\A}\otimes\nu_{\F}$. Note that under random labelling we have that $Q_{x_1,\ldots,x_n}^{\M}
((L_1^*,F_1^*),\ldots,(L_n^*,F_n^*))=1$, so by setting $D_i\times E_i=\M=\A\times\F$, $i=1,\ldots,n-1$, we obtain $k_t^{\bigtimes_{i=1}^{n}(D_i\times E_i)}
(x_1,\ldots,x_n)=\E_{P^{\M}}[
t((L_1^*, F_1^*),\ldots,(L_n^*,F_n^*))]$ and we are in the setting of \citet{penttinen:stoyan:89} under independent marking.
% When $D_i\times E_i=\A\times\F$, $i=1,\ldots,n$, it yields the {\em $t$-correlation functional} $k_t^{\M^n}
% (x_1,\ldots,x_n)$, which is an $n$-point FMPP version of the $f$-correlation function in \citet{penttinen:stoyan:89}; we interpret it as the expectation of $t((L_1, F_1),\ldots,(L_n,F_n))$, conditionally on $X_i=x_i$, $i=1,\ldots,n$. 
% 
% Moreover, 
% when the common marginal mark distribution $P^{\M}(\cdot)$ coincides with the reference measure $\nu(\cdot)$, so that $Q^{\M}(\cdot) 
% \equiv1$ and $\rho(x,l,f)=\rho_G(x)$, 
% \begin{align*}
% &\kappa_t^{(D_i\times E_i)_{i=1}^{n}}
% (x_1,\ldots,x_n)
% =
% % \\
% % =&
% \int_{(D_1\times E_1)\times\cdots\times(D_n\times E_n)}
% \gamma_{x_1,\ldots,x_n}^{\M}
% ((l_1,f_1),\ldots,(l_n,f_n))
% \nu_t(d(l_1,f_1)\times\cdots\times d(l_n,f_n))
% \\
% &=
% \int_{(D_1\times E_1)\times\cdots\times(D_n\times E_n)}
% t((l_1,f_1),\ldots,(l_n,f_n))
% \gamma_{x_1,\ldots,x_n}^{\M}
% ((l_1,f_1),\ldots,(l_n,f_n))
% \nu(d(l_1,f_1))
% \cdots
% \nu(d(l_n,f_n))
% ,
% \end{align*}
% 
% Lemma \ref{LemmaWeighted} changes accordingly. 
In particular, 
% $k_t^{\M^n}
% =
% \kappa_{t}^{\M^n}$ 
% and 
\begin{align*}
&
\mathcal K_t^{(D\times E)\M^{n-1}}(C_1\times\cdots\times C_{n-1})
=
\mathcal K_t^{(D\times E)\bigtimes_{i=1}^{n-1}(\A\times\F)}(C_1\times\cdots\times C_{n-1})
=
\\
=&
\frac{1}{P^{\M}(D\times E)}
\int_{C_1\times\cdots\times C_{n-1}}
k_t^{(D\times E)\bigtimes_{i=1}^{n-1}(\A\times\F)}
(0,x_1,\ldots,x_{n-1})
\mathcal K_G(dx_1\times\cdots\times dx_{n-1})
% \\
% =&
% \frac{1}{\nu(D\times E)}
% \int_{D\times E}
% \E^{!(z,l,f)}\Bigg[
% \mathop{\sum\nolimits\sp{\ne}}_{(x_1,l_1,f_1),\ldots,(x_{n-1},l_{n-1},f_{n-1})\in \Psi}
% t((l,f),(l_1,f_1),\ldots,(l_{n-1},f_{n-1}))
% \times
% \\
% &\times
% \prod_{i=1}^{n-1}
% \frac{\1 \{x_i-z\in C_i\}\1\{(l_i,f_i)\in D_i\times E_i\}}{\rho(x_i,l_i,f_i)}
% \Bigg]
% \nu_{\A}(dl)\nu_{\F}(df)
\\
=&
\E_{D\times E}^{!z}
\left[
\mathop{\sum\nolimits\sp{\ne}}_{x_1,\ldots,x_{n-1}\in \Psi_G}
t((L(z),F(z)),(L(x_1),F(x_1)),\ldots,(L(x_{n-1}),F(x_{n-1})))
\prod_{i=1}^{n-1}
\frac{\1 \{x_i-z\in C_i\}}{\rho_G(x_i)}
\right]
% \Bigg[
% \mathop{\sum\nolimits\sp{\ne}}_{(x_1,l_1,f_1),\ldots,(x_{n-1},l_{n-1},f_{n-1})\in \Psi}
% t((L(z),F(z)),(l_1,f_1),\ldots,(l_{n-1},f_{n-1}))
% % \times
% % \\
% % &\times
% \prod_{i=1}^{n-1}
% \frac{\1 \{x_i-z\in C_i\}}{\rho_G(x_i)}
% \Bigg]
% ,
% \\
% &\mathcal K_t^{(\A\times\F)(\A\times\F)_{i=1}^{n-1}}(C_1\times\cdots\times C_{n-1})
% =
% \\
% =&
% \E_G^{!z}\left[
% \mathop{\sum\nolimits\sp{\ne}}_{x_1,\ldots,x_{n-1}\in \Psi_G}
% t((L(z),F(z)),(L(x_1),F(x_1)),\ldots,(L(x_{n-1}),F(x_{n-1})))
% \prod_{i=1}^{n-1}
% \frac{\1 \{x_i-z\in C_i\}}{\rho_G(x_i)}
% \right]
,
\end{align*}
where $(L(x),F(x))$ denotes the marking random element associated to the location $x\in\R^d$ and $z\in\R^d$ is arbitrary. In particular, when $D\times E=\M=\A\times\F$ we have $P^{\M}(D\times E)=P^{\M}(\M)=1$ and we recall from Section \ref{s:Palm} than the expectation $\E_{D\times E}^{!z}[\cdot]$ above becomes the reduced Palm distribution $\E_G^{!z}[\cdot]$ of the ground process. This is a $n$-point mark-weighted version of the inhomogeneous $K$-function of \citet{BaddeleyEtAl}. 

We finally note that when we have homogeneity in combination with a common mark distribution (being implied by stationarity), we replace $\rho_G(x_i)$ in \eqref{eq:K_measure_common_mark} by the constant ground intensity $\rho_G>0$. In particular, 
\begin{align*}
&
\mathcal K_t^{(D\times E)\M^{n-1}}(C_1\times\cdots\times C_{n-1})
=
\\
=&
\frac{1}{\rho_G^{n-1}}
\E_{D\times E}^{!z}
\left[
\mathop{\sum\nolimits\sp{\ne}}_{x_1,\ldots,x_{n-1}\in \Psi_G}
t((L(z),F(z)),(L(x_1),F(x_1)),\ldots,(L(x_{n-1}),F(x_{n-1})))
\prod_{i=1}^{n-1}
\1 \{x_i-z\in C_i\}
\right]
\end{align*}
for almost any $z\in\R^d$.

\subsection{Choosing test functions -- analysing dependent functional data}
\label{s:TestFunctions}

By choosing different test functions $t(\cdot)$, we may extract different features from the marks. 
In practice, in a statistical context, it is most likely that one will focus only on the case $n=2$; note the connections with \citet{MateuFMPP}.
Note in particular that when $n=2$, if we ignore the functional marks and set $t((l_1, f_1),(l_2,f_2))= l_1 l_2$, \eqref{eq:K_measure} yields an intensity reweighted version of the classical {\em mark correlation function} for the auxiliary marks. If, instead, $t((l_1, f_1),(l_2,f_2))= (l_1-l_2)^2/2$, we obtain the classical {\em mark variogram} for the auxiliary marks \citep{Illian}. 
The question that remains is how we should choose sensible tests functions $t(\cdot)$ which include also the functional marks.

Starting with the simple case $t(\cdot)\equiv1$, we obtain $\nu_t=\nu_{\M}^n$ and %\eqref{e:MarkDependenceFunction} becomes
\begin{align*}
&\kappa_t^{(D\times E)\bigtimes_{i=1}^{n-1}(D_i\times E_i)}
(0,x_1,\ldots,x_{n-1})
=
\\
=&
\int_{(D\times E)\times(D_1\times E_1)\times\cdots\times(D_{n-1}\times E_{n-1})}
\gamma_{x_1,\ldots,x_n}^{\M}
((l,f),(l_1,f_1),\ldots,(l_{n-1},f_{n-1}))
\nu_{\M}(d(l,f))
\prod_{i=1}^{n-1}
\nu_{\M}(d(l_i,f_i))
.
\nonumber
\end{align*}
By additionally letting $n=2$ in \eqref{eq:K_measure}, we retrieve the {\em marked second order reduced moment measure} $\mathcal K^{(D\times E)(D_1\times E_1)}(C)$ of \citet{Iftimi}, which measures intensity reweighted interactions between points with marks in $D\times E$ and points with marks in $D_1\times E_1$, when their separation vectors belong to $C\in\B(\R^d)$. We stress that this measure, and thereby also \eqref{eq:K_measure}, is non-symmetric in the mark sets, i.e., $\mathcal K^{(D\times E)(D_1\times E_1)}(\cdot)\neq \mathcal K^{(D_1\times E_1)(D\times E)}(\cdot)$ in general \citep{Iftimi}. 
In particular, choosing $C_1$ to be the closed origin-centred ball $B_{\R^d}[0,r]$ of radius $r\geq0$, we obtain the marked inhomogeneous $K$-function $K_{\rm inhom}^{(D\times E)(D_1\times E_1)}(r)$ of \citet{CronieLieshoutMPP}, which measures pairwise intensity reweighted spatial dependence within distance $r$ between points with marks in $D\times E$ and points with marks in $D_1\times E_1$. 
Moreover, setting 
$C_1=\{a(\cos v, \sin v) : a\in[0,r], v\in[\phi,\psi] \text{ or } v\in[\pi+\phi,\pi+\psi]\}$ 
for $\phi\in[-\pi/2,\pi/2)$ and $\psi\in(\phi,\phi+\pi]$, we obtain a marked inhomogeneous directional version which may be used to study departures from isotropy, and setting $C_1=\{(x,s):\|x\|\leq r\text{ and }|s|\leq t\}$ when $\Psi$ is a spatio-temporal FMPP, we obtain a spatio-temporal version $K_{\rm inhom}^{(D\times E)(D_1\times E_1)}(r,t)$, $r,t\geq0$, of $K_{\rm inhom}^{(D\times E)(D_1\times E_1)}(r)$ \citep{Iftimi}. 

Hence, for an arbitrary $n$, setting $t(\cdot)\equiv1$ in \eqref{eq:K_measure} we would obtain a definition of a {\em marked $n$th order reduced moment measure}, $\mathcal K^{(D\times E)\bigtimes_{i=1}^{n-1}(D_i\times E_i)}(C_1\times\cdots\times C_{n-1})$, which has an analogous interpretation; it measures intensity reweighted spatial interaction between an arbitrary point with mark in $D\times E$ and distinct $(n-1)$-tuples of other points where, respectively, the separation vectors between these points and the $D\times E$-marked point belong to $C_i$, $i=1,\ldots,n-1$, and these points have marks belonging to $D_i\times E_i$, $i=1,\ldots,n-1$. 
Moreover, it may be of particular interest to choose all $C_i$, $i=1,\ldots,n-1$, to be the same set $C_1$. E.g., $C_i=B_{\R^d}[0,r]$, $i=1,\ldots,n-1$, $r\geq0$, would yield an $n$-point version, $K_{\rm inhom}^{(D\times E)\bigtimes_{i=1}^{n-1}(D_i\times E_i)}(r)$, of the marked inhomogeneous $K$-function of \citet{CronieLieshoutMPP}, which may be used to analyse intensity reweighted interactions between a point with mark in $D\times E$ and $n-1$ of its $r$-close neighbours, which have marks belonging to the respective sets $D_i\times E_i$, $i=1,\ldots,n-1$. 

%The independent marking scenario is an immediate consequence of the observations in Section \ref{s:IndependentMarking} and the Poisson ground process scenario is an immediate consequence of the observations in Section \ref{s:CorrelationFunctionals}.

We mention that when $t(\cdot)\equiv1$ under independent marking, $\bar{\mathcal K}_t^{(D\times E)(D_i\times E_i)_{i=1}^{n-1}}(\cdot)\equiv1$, which may be used to statistically test for independent marking.

% We further see that the normalisation

% reveals features of the marking structure, conditionally on the locations. 

We next turn to test functions which include the functional marks and we here only consider the case $n=2$. A natural starting point, we argue, is to consider metrics (distances) between the (functional) marks. 
There are various choices to be considered (see e.g.\ \cite{deza:deza:09} and the references therein) and each may reflect different features of the functional marks' properties; although it may be natural to use the metric having generated the assumed Polish topology of the function space $\F$, we may naturally consider different choices here. We here choose to consider the following metrics as test functions: $L_p$-metrics as defined in \eqref{Lp} in the Appendix, i.e.\ $t(f_1,f_2)=d_{L_p}(f_1,f_2)=(\int_{\T}|f_1(t)-f_2(t)|^p\de t)^{1/p}$, $1\leq p<\infty$, the uniform metric (or $L_{\infty}$-metric) $t(f_1,f_2)=d_{\infty}(f_1,f_2)=\sup_{t\in\T}|f_1(t)-f_2(t)|$ (see also Section \ref{s:FunctionSpaces}), and the symmetrised Kullback-Leibler divergence, 
$$
t(f_1,f_2)=KL(t_1,t_2)=\int_{\T}\log\left(\frac{f_1(t)}{f_2(t)}\right)f_1(t)\de t +\int_{\T}\log\left(\frac{f_2(t)}{f_1(t)}\right)f_2(t)\de t.
$$
A further choice is to consider angles, or rather inner products; $t(f_1,f_2)=\langle f_1,f_2\rangle=\int_{\T}f_1(t)f_2(t)\de t$. 
In the literature on functional clustering a common measure of proximity between two functions is \Citep{ferraty:vieu:06}
\[
t(f_1,f_2)
=
d_{L_2}\left((df_1/dt)^k,(df_2/dt)^k\right)
=
\left(\int_{\T}|(df_1(t)/dt)^k - (df_2(t)/dt)^k|^2\de t\right)^{1/2}
,
\quad k\geq1,
\]
provided that the $k$th derivatives $df_i(t)/dt$, $t\in\T$, $i=1,2$, exist. 
When, conditionally on $\Psi_{\X\times\A}$, all the functional marks have the same mean $\bar F(t)=\E[F_i(t)|\Psi_{\X\times\A}]$, $t\in\T$, which e.g.\ is the case when there is a common functional mark distribution, we may consider a functional mark counterpart of the test function for the classical variogram, 
\begin{align}\label{test:vario}
 t(f_1,f_2)=t_{v}(f_1,f_2)=\int_a^b\left(f_1(t)-\bar F(t)\right)\left(f_2(t)-\bar F(t)\right)\de t,   
\end{align}
where, in practice, $\bar F(t)$ may be estimated by means of $(1/n)\sum_{i=1}^nf_i(t)$, i.e.\ the average functional mark at time $t$ for the observed functional part of the point pattern.
Note that for each of the above choices we may reduce the interval $\T$ to some smaller interval $[a,b]\subset\T$. Moreover, we may consider combinations of them by summing them up.

When we want to consider test functions which include both functional and auxiliary marks, we may exploit metric preserving properties of certain operations \citep[p.\ 8]{VanLieshout}, such as summation and maximum, and apply these to the above mentioned test functions (metrics) for the functional marks and the metrics provided by \citet[page 343]{Illian} for auxiliary marks in order to define a test function for general purposes. When $n=2$, one may e.g.\ consider the following two test functions:
\begin{align*}
t((l_1,f_1),(l_{2},f_{2}))
&=
d_{\F}(f_1,f_2) + 
%\|l_1-l_2\|\,
l_1l_2,
\\
t((l_1,f_1),(l_{2},f_{2}))
&=
\max\{d_{\F}(f,f_i),
\|l_1-l_2\|\},
% \min\{d_{\F}(f,f_i),
% \|l_1-l_2\|\},
\end{align*}
% $$
% t((l,f),(l_1,f_1),\ldots,(l_{n-1},f_{n-1}))
% =\sum_{i=1}^{n-1}
% d_{\F}(f,f_i) + \|l-l_i\|
% $$
% or 
% $$
% t((l,f),(l_1,f_1),\ldots,(l_{n-1},f_{n-1}))
% =
% \max\{\max_{i=1,\ldots,n-1} d_{\F}(f,f_i),
% \max_{i=1,\ldots,n-1} \|l-l_i\|\},
% $$
where $d_{\F}(\cdot,\cdot)$ is a metric on function space $\F$ as mentioned above. For general $n$, we will follow the same procedure.

\subsection{Non-parametric statistical inference}

We next turn to the non-parametric estimation of our summary statistics. Specifically, we here assume that we observe an FMPP $\Psi$ within a bounded spatial domain $W\in\B(\R^d)$, $|W|>0$, i.e., we sample $\Psi\cap W\times\M = \Psi\cap W\times\A\times\F$.

Theorem \ref{LemmaUnbiased} below provides a non-parametric estimator of the $t$-weighted marked $n$th order reduced moment measure, and it provides a condition for edge corrections to render it unbiased. 
%Its proof can be found in Section \ref{s:Proofs} in the Appendix.
Its proof can be found in Appendix \ref{s:Proofs}.

\begin{thm}
\label{LemmaUnbiased}
Consider a $k$-MIRS FMPP $\Psi$ and a test function $t=t_n:\M^n=(\A\times\F)^n\to[0,\infty)$, $2\leq n\leq k$. Moreover, let $D\times E\in\B(\M)=\B(\A\times\F)$, $\nu_{\M}(D\times E)>0$, and $D_i\times E_i\in\times\B(\M)=\B(\A\times\F)$, $\nu_{\M}(D_i\times E_i)>0$, $i=1,\ldots,n-1$. 
The estimator 
\begin{align}
\label{eq:K_measure_est}
&\widehat{\mathcal K}_t^{(D\times E)\bigtimes_{i=1}^{n-1}(D_i\times E_i)}(C_1\times\cdots\times C_{n-1})
=
\widehat{\mathcal K}_t^{(D\times E)\bigtimes_{i=1}^{n-1}(D_i\times E_i)}(C_1\times\cdots\times C_{n-1};\Psi, W,\A,\F)
\\
&=
\frac{1}{\nu_{\M}(D\times E)\prod_{i=1}^{n-1}\nu_{\M}(D_i\times E_i)}
\sum_{(x,l,f)\in\Psi}
\mathop{\sum\nolimits\sp{\ne}}_{(x_1,l_1,f_1),\ldots,(x_{n-1},l_{n-1},f_{n-1})\in \Psi\setminus\{(x,l,f)\}}
w(x,x_1,\ldots,x_{n-1})
\times
\nonumber
\\
&
\times
t((l,f),(l_1,f_1),\ldots,(l_{n-1},f_{n-1}))
\frac{\1\{(x,l,f)\in W\times D\times E\}}{\rho(x,l,f)}
\prod_{i=1}^{n-1}
\frac{%\1\{x_i\in \}
\1\{(x_i,l_i,f_i)\in (W\cap(x + C_i))\times D_i\times E_i\}}
{\rho(x_i,l_i,f_i)}
\nonumber
\end{align}
is an of unbiased estimator of $\mathcal K_t^{(D\times E)\bigtimes_{i=1}^{n-1}(D_i\times E_i)}(C_1\times\cdots\times C_{n-1})$, $C_1\times\cdots\times C_{n-1}\in \B(\R^d)^{n-1}$, provided that the intensity function $\rho(\cdot)$ is known and that the edge correction function $w(\cdot)$ satisfies
\[
\int_{W}
\prod_{i=1}^{n-1}\1\{(x_i+x)\in W\}
w(x,x_1+x,\ldots,x_{n-1}+x)
\de x
=
1
\]
for almost any $x_i\in C_i$, $i=1,\ldots,n-1$. Note that $\prod_{i=1}^{n-1}\1\{(x_i+x)\in W\}=\1\{\bigcap_{i=1}^{n-1}\{(x_i+x)\in W\}\}=\1\{x\in\bigcap_{i=1}^{n-1}(W-x_i)\}$.
\end{thm}
% {\color{red} What is $W_{-x_i}$??? We also need to go through the proof to check that the result is correct. }

Here three relevant questions immediately arise: 
Which edge correction methods satisfy the condition in Theorem \ref{LemmaUnbiased}, and are there other (biased) edge correction methods which still work well in practice? 
How do we deal with the rather abstract reference measure $\nu_{\M}=\nu_{\A}\otimes\nu_{\F}$ in \eqref{eq:K_measure_est}? 
How should we deal with the unknown true intensity $\rho(\cdot)$ in \eqref{eq:K_measure_est}?

Regarding the edge correction function $w(\cdot)$, 
letting $t(\cdot)\equiv1$ as well as assuming that $\Psi$ has a common mark distribution which coincides with the reference measure, we obtain the estimator
\begin{align*}
&
\widehat{\mathcal K}_G(C_1\times\cdots\times C_{n-1})
=
\widehat{\mathcal K}_1^{\M^n}(C_1\times\cdots\times C_{n-1})
=
\\
&=
\sum_{x\in\Psi_G}
\mathop{\sum\nolimits\sp{\ne}}_{x_1,\ldots,x_{n-1}\in \Psi_G\setminus\{x\}}
w(x,x_1,\ldots,x_{n-1})
\frac{\1\{x\in W\}}{\rho_G(x)}
\prod_{i=1}^{n-1}
\frac{
\1\{x_i\in W\}
\1\{x_i\in (x + C_i)\}
}
{\rho_G(x_i)}
\end{align*}
of $\mathcal K_G(C_1\times\cdots\times C_{n-1})$, based on $\Psi_G\cap W$, 
and by looking closer at the case $n=2$ in the literature (see e.g.\ \citet{CronieLieshoutMPP}, \citet[Appendix 1]{Gabriel2014} and \citet{baddeley:98}) we get guidance in identifying suitable edge corrections. We obtain that the following choices satisfy the condition of Theorem \ref{LemmaUnbiased}; the proof of Corollary \ref{CorUnbiased} is provided in Appendix \ref{s:Proofs}.

\begin{comment}
Regarding the edge correction function $w(\cdot)$, 
%note first that the choice 
%{\color{red}$w(\cdot)\equiv1/|W|$} 
%corresponds to no edge correction.
%Moreover, 
from e.g.\ \citet{CronieLieshoutMPP}, \citet[Appendix 1]{Gabriel2014} and \citet{baddeley:98} we obtain that the following choices satisfy the condition of Theorem \ref{LemmaUnbiased}; the proof of Corollary \ref{CorUnbiased} is provided in Appendix \ref{s:Proofs}. 
\end{comment}

\begin{cor}\label{CorUnbiased}
The {\em minus sampling} edge correction
$$
w_{\ominus}(x,x_1+x,\ldots,x_{n-1}+x)
=
\1\left\{x\in \bigcap_{i=1}^{n-1}W\ominus C_i\right\}
\left/\left|\bigcap_{i=1}^{n-1}W\ominus C_i\right|\right., 
$$
where $\ominus$ denotes Minkowski subtraction, and the {\em translational} edge correction 
$$
w_{\cap}(x,x+x_1,\ldots,x+x_{n-1})
=
1\left/\left|\bigcap_{i=1}^{n-1}
% W_{x+x_i}
(W+(x+x_i))
\cap 
(W+x)
%W_x
\right|\right.
$$
both yield that the estimator in Theorem \ref{LemmaUnbiased} is unbiased. 
%{\color{red}I think that this should work for any $d\geq1$ by letting $\ell$ be the $d-1$-dimensional Hausdorff measure in $\R^d$. Do we see any way of combining $\partial B_{\R^d}[x, \|x_1\|],\ldots,\partial B_{\R^d}[x, \|x_{n-1}\|]$ to get it to work for general $n\in\{2,\ldots,k\}$?}
Moreover, when the ground space is given by $\R^d$, $d=2,3$, and $n=2$, also the {\em isotropic} or {\em rotational} edge correction
$$
w_{\partial}(x,x+x_1)
=
\frac{\ell(\partial B_{\R^d}[x, \|(x+x_1)-x\|])}{\ell(\partial B_{\R^d}[x,\|(x+x_1)-x\|]\cap W)}
=
\frac{\ell(\partial B_{\R^d}[x, \|x_1\|])}{\ell(\partial B_{\R^d}[x,\|x_1\|]\cap W)}
$$
yields an unbiased estimator \eqref{eq:K_measure_est}; here $\ell$ denotes length in $\R^2$ or surface area in $\R^3$ and $\partial$ is used to denote the boundary of a set.

\end{cor}

There are clearly other edge correction methods such as  {\em rigid motion correction} which do not satisfy the condition in Theorem \ref{LemmaUnbiased} but still work well in practice.  
%{\color{red} Which other edge corrections should we consider? Riplyl's edge correction.}

Turning to the second question, in analogy with \citet{BaddeleyEtAl,CronieLieshoutMPP,Iftimi,Zhao}, define the random measures
\begin{align*}
%\label{e:IntensityRewMeasureGround}
\Xi_G(C;\rho_G)&= \sum_{x\in\Psi_G\cap C}
\frac{1}{\rho_G(x)},
\\
%\label{e:IntensityRewMeasure}
\Xi(C\times D\times E;\rho)
&=
\sum_{(x,l,f)\in\Psi\cap C\times D\times E}
\frac{1}{\rho(x,l,f)}
,
\qquad C\times D\times E\in\B(\R^d\times\A\times\F),
\nonumber
\end{align*}
and note that 
\[
\E\left[
\Xi(W\times D\times E;\rho)
%\sum_{(x,l,f)\in\Psi\cap W\times D\times E}
%\frac{1}{\rho(x,l,f)}
\right]
/
\E\left[
\Xi_G(W;\rho_G)
% \sum_{x\in\Psi_G\cap W}
% \frac{1}{\rho_G(x)}
\right]
=
|W|\nu_{\M}(D\times E)/|W|
= \nu_{\M}(D\times E)
\]
by the Campbell formula. Hence, $\Xi_G(C;\rho_G)$  %\eqref{e:IntensityRewMeasureGround} 
is an unbiased estimator of $|W|$ and 
$
\widehat{\nu_{\M}}(D\times E;\rho,\rho_G)
=\Xi(W\times D\times E;\rho)/\Xi_G(W;\rho_G)
% \widehat{\nu_{\M}}(D\times E)
% =
% \frac{\widehat{|W|\nu_{\M}(D\times E)}}{\widehat{|W|}}
% =
% \left.
% \sum_{(x,l,f)\in\Psi\cap W\times D\times E}
% \frac{1}{\rho(x,l,f)}
% \right/
% \sum_{x\in\Psi_G\cap W}
% \frac{1}{\rho_G(x)}
$ 
is a ratio-unbiased estimator of $\nu_{\M}(D\times E)$, $D\times\E\in\B(\A\times\F)$. Following a suggestion by \citet{StoyanStoyanRatio}, in 
\eqref{eq:K_measure_est} it is advised to replace 
%$|W|$ and 
$\nu_{\M}(D\times E)\prod_{i=1}^{n-1}\nu_{\M}(D_i\times E_i)$ 
by the corresponding estimator 
%Theorem \ref{LemmaUnbiased} 
%product of estimators of the above kind 
to obtain a ratio-unbiased estimator which yields better estimates in practice. 
This approach is referred to as the Hamilton principle. 
%Note further that $|W|$ in \eqref{eq:K_measure_est} may be replaced by $\sum_{x\in\Psi_G\cap W}1/\rho_G(x)$. 
Moreover, in the case of the minus sampling edge correction, the arguments above should be applied to $|W\ominus\bigcap_{i=1}^{n-1}C_i|$ instead of $|W|$

These observations directly connect to the third question, which is how we deal with the fact that the true intensity function is unknown in practice. The most common and natural approach is to replace $\rho(\cdot)$ in Theorem \ref{LemmaUnbiased} by a plug-in estimator $\widehat\rho(x,l,f)$, $(x,l,f)\in W\times\A\times\F$. This, however, connects back to the problem of specifying $\nu_{\M}$ because to estimate $\rho(\cdot)$ we need to know $\nu_{\M}$ -- the intensity function is a Radon-Nikodym derivative with respect to the reference measure. A pragmatic and (we argue) not so restrictive approach is to assume that there is a common functional mark distribution which coincides with the functional mark reference measure $\nu_{\F}$. By doing so, any intensity estimator is of the form $\widehat\rho(x,l,f)=\widehat\rho_{W\times\A}(x,l)=\widehat Q_{x}^{\A}(l) \widehat \rho_G(x)$, $(x,l,f)\in W\times\A\times\F$, i.e., it does not depend on the functional mark values. In other words, we are in the land of estimating intensity functions for point processes with real valued marks or/and multivariate point processes. Hence, we may consider the estimator
\begin{align}
\label{eq:K_measure_est_common_mark}
&\widehat{\mathcal K}_t^{(D\times E)\bigtimes_{i=1}^{n-1}(D_i\times E_i)}(C_1\times\cdots\times C_{n-1})
=
\frac{1}{
\nu_{\M}(D\times E)\prod_{i=1}^{n-1}\nu_{\M}(D_i\times E_i)
}
\times
\\
&
\times
\sum_{(x,l,f)\in\Psi}
\mathop{\sum\nolimits\sp{\ne}}_{(x_1,l_1,f_1),\ldots,(x_{n-1},l_{n-1},f_{n-1})\in \Psi\setminus\{(x,l,f)\}}
w(x,x_1,\ldots,x_{n-1})
t((l,f),(l_1,f_1),\ldots,(l_{n-1},f_{n-1}))
\times
\nonumber
\\
&
\times
\frac{\1\{(x,l,f)\in W\times D\times E\}}{\widehat\rho_{W\times\A}(x,l)}
\prod_{i=1}^{n-1}
\frac{\1 \{x_i\in W\cap(x + C_i)\}\1\{(l_i,f_i)\in D_i\times E_i\}}{\widehat\rho_{W\times\A}(x_i,l_i)}
\nonumber
\end{align}
of $\mathcal K_t^{(D\times E)(D_i\times E_i)_{i=1}^{n-1}}(C_1\times\cdots\times C_{n-1})$. 
%and 
Moreover, taking the Hamilton principle into account, we %may 
would here replace the reference measure related parts in \eqref{eq:K_measure_est_common_mark} by the estimators $\widehat{|W|}=\Xi(W;\widehat\rho_G)$, $\widehat{\nu_{\M}}(D\times E;\widehat\rho_{W\times\A},\widehat\rho_G)$ and $\widehat{\nu_{\M}}(D_i\times E_i;\widehat\rho_{W\times\A},\widehat\rho_G)$, $i=1,\ldots,n-1$.  
% \begin{align*}
% % \widehat{|W|}=\Xi(W;\widehat\rho_G)
% % % &=
% % % \sum_{x\in\Psi_G\cap W}
% % % \frac{1}{\widehat\rho_G(x)}
% % ,
% % \qquad
% % \\
% \widehat{\nu_{\M}}(D\times E;\widehat\rho_{W\times\A},\widehat\rho_G)
% ,
% \qquad
% \widehat{\nu_{\M}}(D_i\times E_i;\widehat\rho_{W\times\A},\widehat\rho_G), \quad i=1,\ldots,n-1
% % &=
% % \left.
% % \sum_{(x,l,f)\in\Psi\cap W\times D\times E}
% % \frac{1}{\widehat\rho_{W\times\A}(x,l)}
% % \right/
% % \widehat{|W|}_{\widehat\rho_G(\cdot)}
% .
% \end{align*}
This is indeed quite remarkable -- we may estimate a statistic based on something as abstract as a measure on a Polish function space, as well as a Radon-Nikodym derivative with respect to it, without ever having to know or consider any of these entities. 
Now, it should be noted that the Hamilton principle reference measure estimators may be ignored for certain intensity estimators since these estimators already satisfy $\widehat{|W|}=\Xi(W;\widehat\rho_G)=|W|$ and $\widehat{\nu_{\M}}(D\times E;\widehat\rho_{W\times\A},\widehat\rho_G)=\nu_{\M}(D\times E)$ \citep{CvL,Moradi}.
Note finally that if we impose the stronger assumption that there is a common mark distribution $P^{\M}$ (auxiliary and functional marks) which coincides with $\nu_{\M}$, or if we do not consider any auxiliary marks, we simply replace $\widehat\rho_{W\times\A}(\cdot)$ above by $\widehat\rho_G(\cdot)$.

% {\color{red} Mimic the consistency result in \citet{Iftimi} -- Note that we have to assume that the measure in \eqref{e:IntensityRewMeasure} is both stationary and ergodic.}

How to choose appropriate mark sets and test functions completely depends on the specific context in which the data are studied as well as the underlying scientific questions. Section \ref{s:TestFunctions} points to a few different choices which may be of general interest, in particular for spatio-temporal (functional) marked point processes.

% {\color{red}
% We
% let $W \subset\R^d$ denote a bounded observation window of Lebesgue measure $|W| > 0$,
% and assume that a realization of $\Psi$ on $W$ is observed.
% Following \citet{penttinen:stoyan:89}, for $n=2$, assuming that $\Psi$ is MSOIRS, we define   $K_t$-functional 
% as follows,
% \begin{align*}
% K_t(r)=\E\sum_{\substack{(x_1,l_1,f_1)\\(x_2,l_2,f_2)}\in\Psi}^{\ne} \frac{\1(\|x_1-x_2\|\le r) t((l_1,f_1),(l_2,f_2))}{\rho_G(x_1)\rho_G(x_2)}.
% \end{align*}
%  This function is the Ripley's $K$-function  counterpart 
% in FMMP setting. For non-parametric estimation of $K_t$-functional, if $x \in W$,  a
% straightforward calculation shows that 
% \begin{align*}
% \hat K_t(r)=\sum_{\substack{(x_1,l_1,f_1)\\(x_2,l_2,f_2)}\in\Psi}^{\ne} \frac{\1(\|x_1-x_2\|\le r) t((l_1,f_1),(l_2,f_2))}{\rho_G(x_1)\rho_G(x_2)|W_{x_1}\cap W_{x_2}|}
% \end{align*}
% is an unbiased estimator of $K_t(r)$.  In practice  an estimate of $\rho_G(\cdot)$ should be plugged in the above estimator. 
% }

\begin{rem}
We here briefly indicate how one could exploit our new summary statistics to perform minimum contrast estimation \citep{Baddeley:etal:16,Diggle2013} when the distribution $P_{\theta_0}$ of $\Psi$ belongs to some parametric family $P_{\theta}$, $\theta\in\Theta\subset\R^v$, $v\geq1$, of distributions.

% , which depends on some parameter vector $\theta\in\Theta\subset\R^v$, $v\geq1$. 
Assume that we are able to explicitly derive 
$K_t^{(D\times E)\bigtimes_{i=1}^{n-1}(D_i\times E_i)}(r)$ in 
Definition \ref{def:K_measure} for some $n\geq2$, some test function $t$, some choice of mark sets $D\times E$, $D_i\times E_i$, $i=1,\ldots,n-1$, any $r\geq0$ and any $\theta\in\Theta$. 
Denoting this by $K_t^{(D\times E)\bigtimes_{i=1}^{n-1}(D_i\times E_i)}(r;\theta)$, we may may obtain an estimate $\widehat\theta$ of $\theta_0$ by minimising e.g.%\ either of
\begin{align*}
% \theta
% &\mapsto\int_{r_{min}}^{r_{max}}
% \left|
% K_t^{(D\times E)\bigtimes_{i=1}^{n-1}(D_i\times E_i)}(r;\theta)
% -
% \widehat K_t^{(D\times E)\bigtimes_{i=1}^{n-1}(D_i\times E_i)}(r)
% \right|^p\de r,
% \\
\theta
&\mapsto\int_{r_{min}}^{r_{max}}
\left|
K_t^{(D\times E)\bigtimes_{i=1}^{n-1}(D_i\times E_i)}(r;\theta)^q
-
\widehat K_t^{(D\times E)\bigtimes_{i=1}^{n-1}(D_i\times E_i)}(r)^q
\right|^p\de r,
\end{align*}
for some suitable $p,q>0$ and $0 \leq r_{min} < r_{max}<\infty$; the non-parametric estimator $\widehat K_t^{(D\times E)\bigtimes_{i=1}^{n-1}(D_i\times E_i)}(r)$ is obtained through Theorem \ref{LemmaUnbiased} by setting $C_1=\cdots= C_{n-1}=B_{\R^d}[0,r]$, $r\geq0$. 

\end{rem}

\section{Data analysis}\label{sec:DataAnalysis}
%Here, we numerically illustrate how our proposed setting and methodologies applies to the real data. 
Here, we numerically illustrate how our proposed setting and methods may be applied to real data. 
In particular, we will focus on the summary statistics and show their potential usefulness for 
extracting 
%the 
features 
in Spanish province population growth; 
see the discussion around 
%as explained in the top panels of 
Figure~\ref{Illu}.
%The boundary and the centre coordinates data
The boundary and centre coordinate data
of the provinces of Spain are extracted as shapefiles from the R package {\em raster} \citep{hijmans:19} and the statistical information about the population is taken from the web page of the {\em Spanish Institute of Statistics} (www.ine.es).

\subsection{Spatial variation of population characteristics in Spain}
To better understand the structure and dynamics of populations, two key points are having information about i) the spatial distribution of and the magnitude variation in the demography, and ii) the population growth rate. 
% Providing information about i) the distribution/pattern and the magnitude of the spatial variation in the demography and ii) the population growth rate are key points in better understanding the structure and dynamics of populations. 
In anthropology and demography, demographical evolution and sex-ratio are two important population characteristics which can change over time because of e.g.\ birth and death rates, economical situations or migration. 
% the human sex ratio is the ratio of males to females in a population. Demographic evolution  is the change in total population because of the   birth and death rates or developing from a pre-industrial to an industrialized economic system. 
However, it is natural to expect that
%, probably due to nutrition and atmospheric conditions and immigration, 
these indices are much more similar in  %some 
neighbouring regions/provinces than in distant regions/provinces. 
As 
%itemized 
highlighted 
in Section~\ref{SectionIntroduction}, 
one of the most important aspects of the analysis is to deduce whether the functional marks, i.e.\ the demographic evolution and sex ratio, are spatially dependent. %over time.
%the most important question in analyzing FMPP is to figure out  if  the functional marks, i.e.\ demographic evolution and sex ratio, are spatially dependent over time.

% The objectives of studying point processes with a functional mark are essentially the same as in other marked point processes. The most important question is that of knowing if there is spatial dependence in the functional marks. 

For both the demographic evolution and sex ratio curves, we use the test function \eqref{test:vario} in the estimator in expression \eqref{eq:K_measure_est_common_mark}; note that we here assume that there is a common mark distribution and that there are no auxiliary marks present. In both cases, we observed the functions for 20 distinct years, starting from 1998. Hence, each such observed function $f_i$ can be represented as the collection $f_i(t_1),\ldots,f_i(t_{20})$, $i=1,\ldots,n$. As a result, the distance function \eqref{test:vario} for any two observed functions $f_1$ and $f_2$ is approximated by
\begin{align*}
 \tilde{t}_{v}(f_1,f_2)=\frac{b-a}{20}\sum_{j=1}^{20}\left(f_1(t_j)-\bar F(t)\right)\left(f_2(t_j)-\bar F(t)\right),   
\end{align*}
where $a=1998$ and $b=2017$.
Hence, we focus on pairwise interactions and we let $C_1$ be given by the balls $B_{\R^2}[0,r]$, $r\geq0$, whereby we obtain a  weighted $K$-function, where we use Ripley's isotropic edge correction (recall Corollary \ref{CorUnbiased}) to correct for edge effects. 
%{\color{blue}
%We assume that there is a common mark distribution and that there are no auxiliary marks present. 
Moreover, we estimate the intensity function of the ground process non-parametrically utilising the {\em density.ppp()} function of the \textrm{R} package {\em spatstat} \citep{Baddeley:etal:16}. 
%} 

\begin{figure}[H]
\begin{center}
\begin{tabular}{cccc}
 \includegraphics*[width=0.3\textwidth,height=0.3\textwidth]{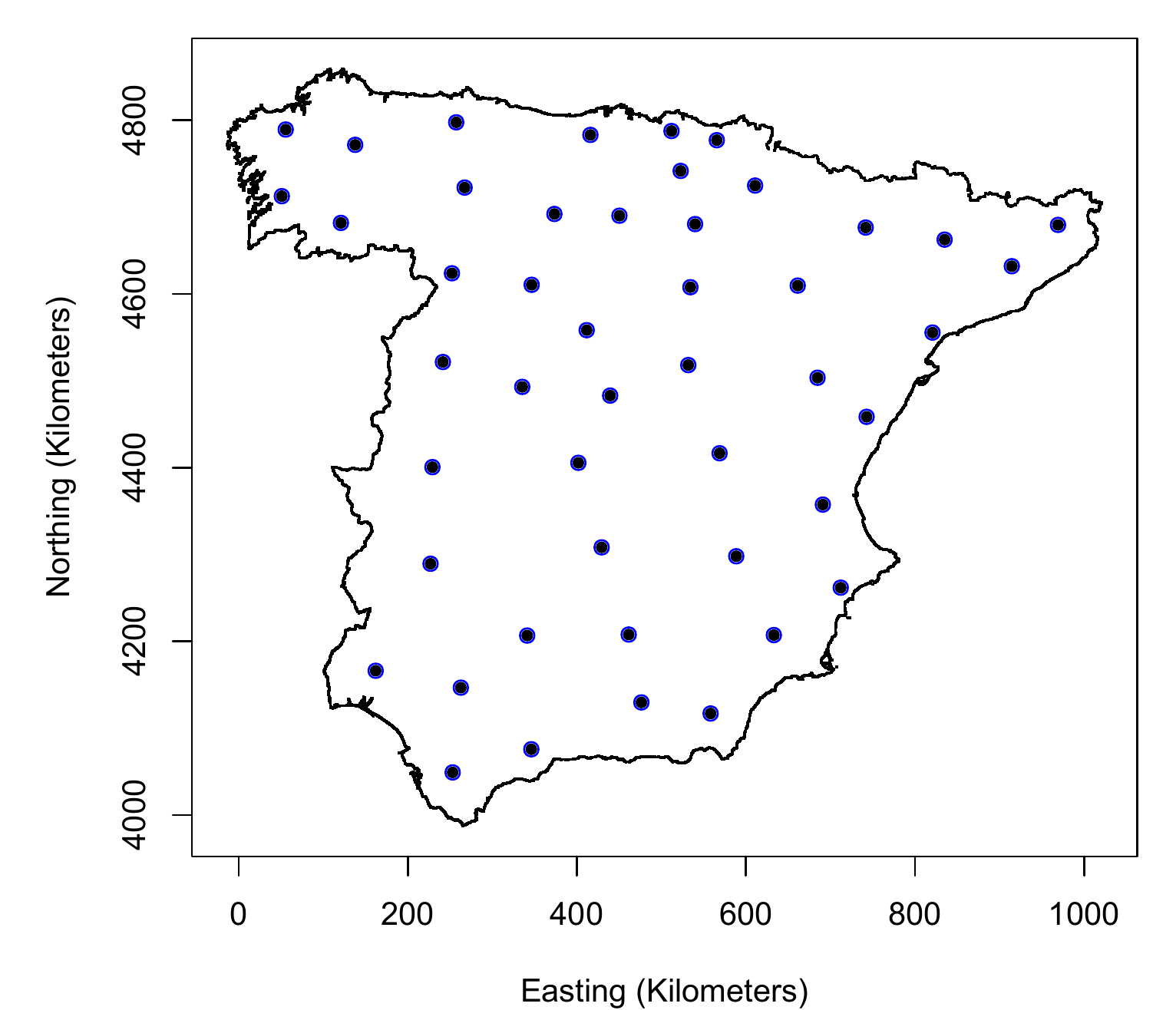}
&
 \includegraphics*[width=0.3\textwidth,height=0.3\textwidth]{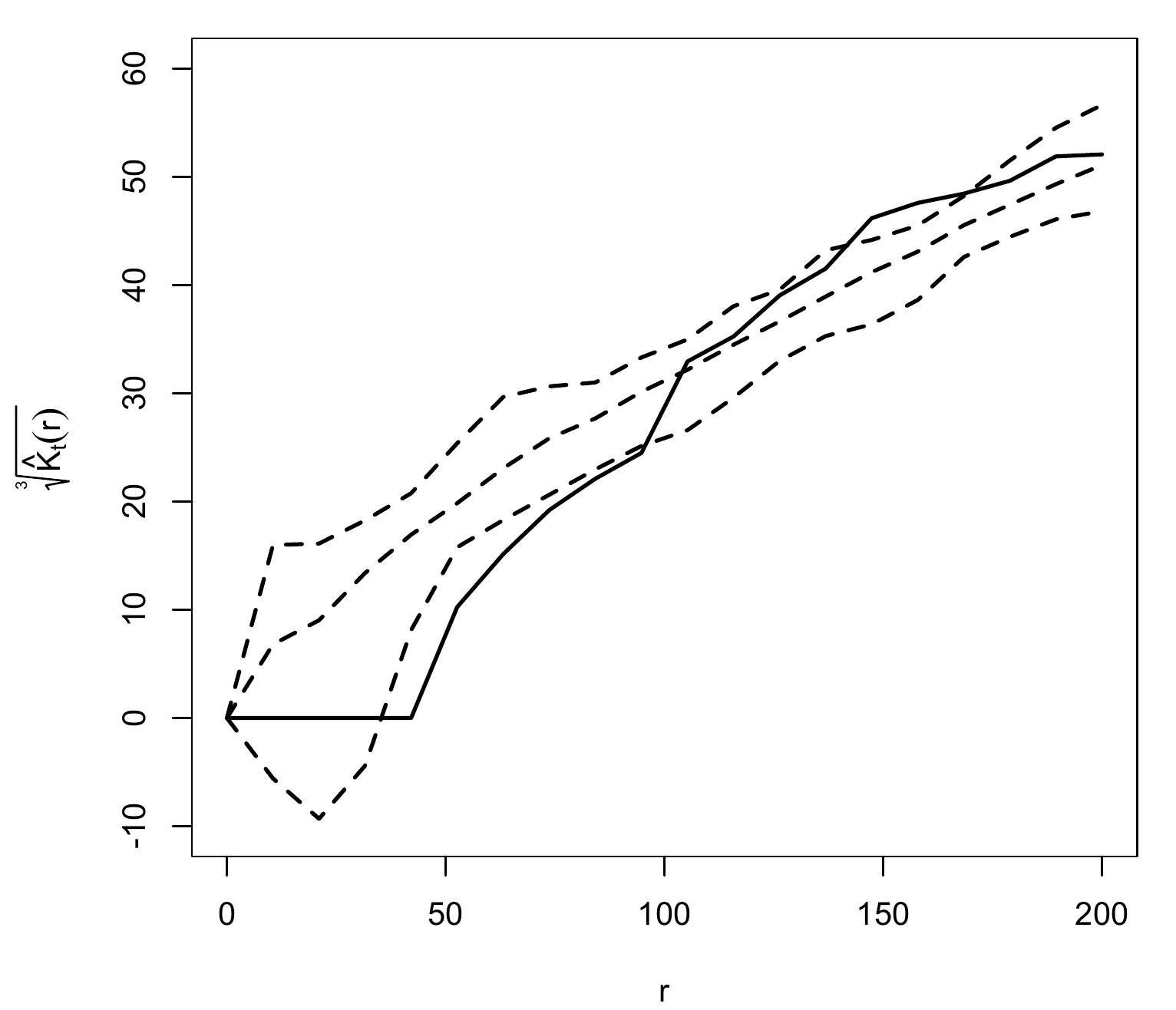}
 \\
  \includegraphics*[width=0.3\textwidth,height=0.3\textwidth]{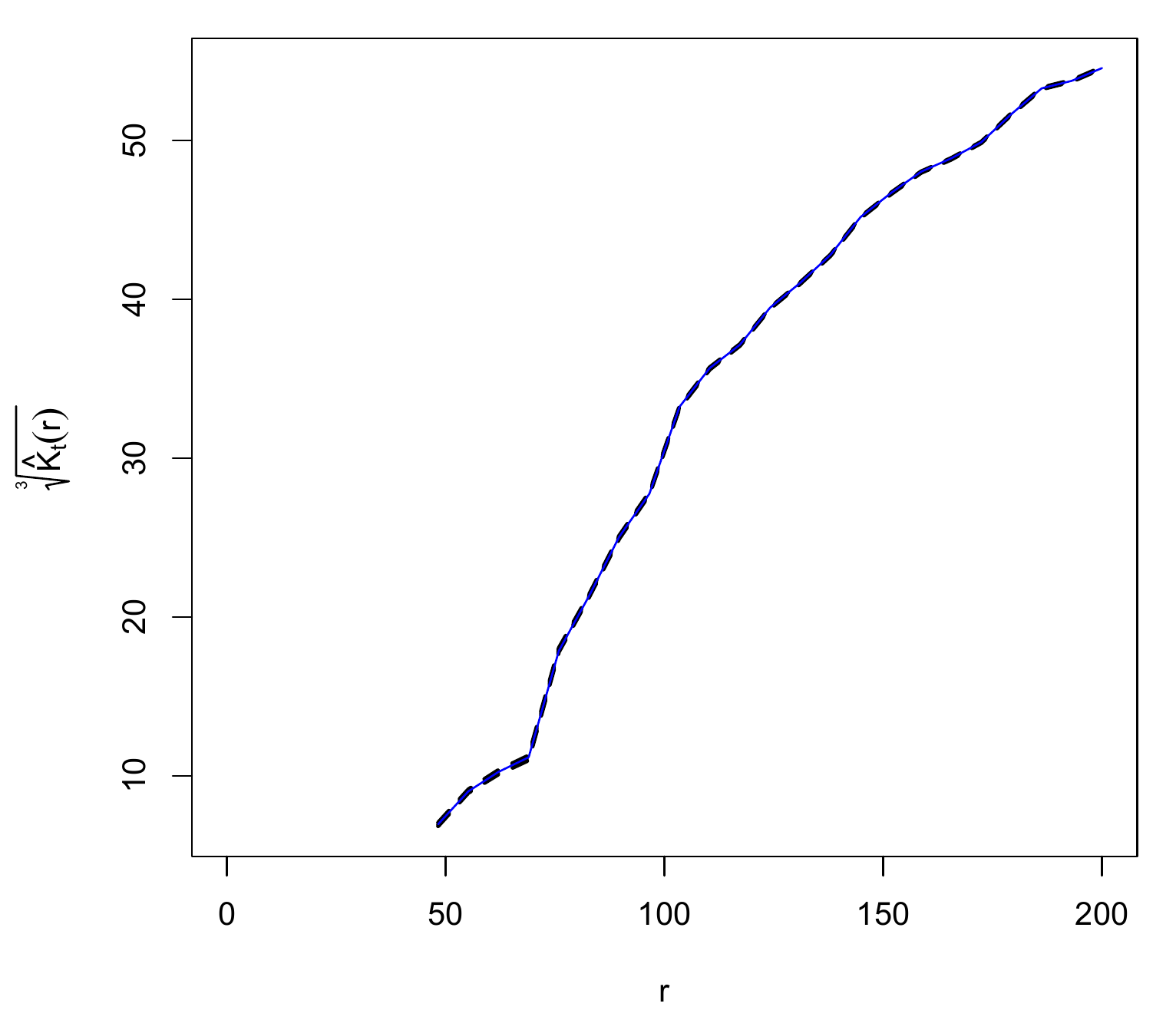}
   &
\includegraphics*[width=0.3\textwidth,height=0.3\textwidth]{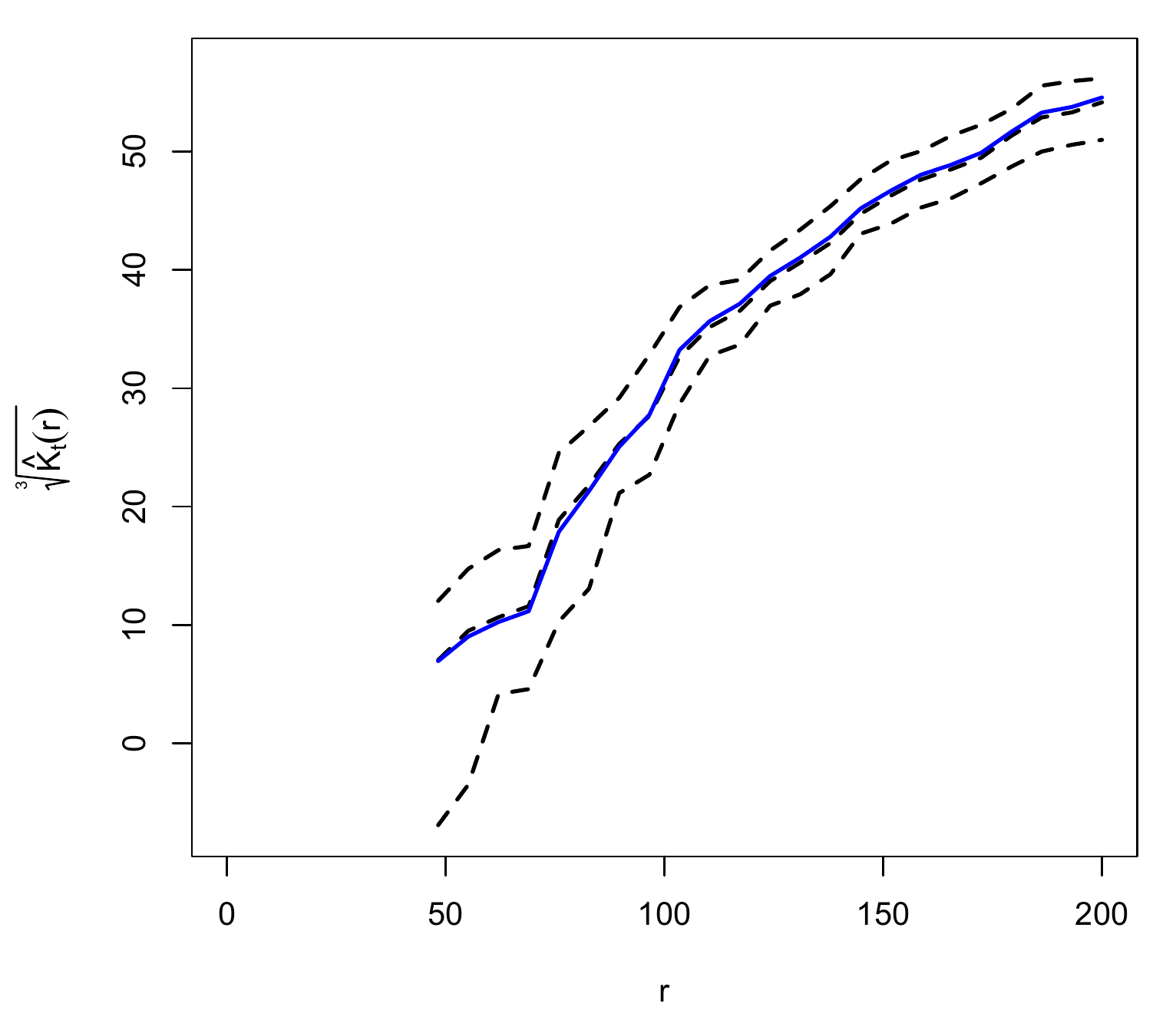}
\end{tabular}
\caption{Spatial point pattern of the centres of 47 provinces on the Spanish mainland (top left panel). 
 %Estimated functional mark  
 %$\sqrt[3]{\hat K_t(r)}$ 
 $(\hat K_t(r))^{1/3}$ 
 for the demographic evolution in 47 provinces of Spain (solid line), average and simulated pointwise  95\%-envelopes under the homogeneous Poisson process for $(\hat K_t(r))^{1/3}$ (dashed lines) (top right panel). Bottom left: 
 %Estimated functional mark $K$-function $\hat K_t(r)$ for the demographic evolution in 47 Spain provinces (solid line), 
 as top right panel but average and simulated 95\%-envelopes from 39 random relabellings of the  demographic evolution data (dashed lines). Bottom right: as left but for the sex ratio data.
 In the bottom panels the curves are shown only for $r \ge 48.27$ km since for the smaller distances the estimated functional mark $K$-function vanishes.}\label{f:3}
\end{center}
\end{figure}

The analysis is illustrated in Figure~\ref{f:3}. The top left panel shows the spatial point pattern of the centres of 47 Spanish provinces.
%{\color{red}(What does each point actually represent? The centre? The largest town?)}. 
%(top left). 
The other three panels show the resulting functional marked $K$-functions for the Spanish provinces functional marked point pattern (see Figure~\ref{Illu}). 
The transformed $\hat K_t(r)$  for the data together with 
simulated pointwise 95\%-envelopes 
%simulated pointwise 0.05 significance envelopes 
generated from 39 simulations of a homogeneous Poisson process, obtained by keeping the functional mark fixed, is shown in the top right panel; the obtained intensity estimate was quite flat so we proceeded assuming homogeneity. Such envelopes are obtained for each value of $r$ by calculating the smallest and largest simulated values of $(\hat K_t(r))^{1/3}$; see \citep{Diggle2013}. This suggests that the functional marked Poisson process model does not fit the functional marked data set at the top left panel of Figure~\ref{Illu} well; some regular model intuitively makes most sense. 
%One can probably fit functional marked Strauss process to these data sets. 
The bottom panels show the transformed version of $\hat K_t(r)$ for the data and the pointwise 0.05 
level 
%significance 
envelopes based on 39 simulations for 
%random labelling, 
demographic evolution on logarithmic scale (left) and sex ratio (right). For $r<48.27 km$, $\hat K_t(r)=0$ and is thus not depicted in the bottom panels.
These functions suggest that there is no spatial dependence between the functional marks, which points to 
%we can conclude 
that the way the population size and the sex ratio have evolved from 1998 to now in different provinces are spatially independent. 

\section{Discussion}

In principle, the current definition of FMPPs may also accommodate situations where we want to consider locations $X_i\in\mathcal{S}$ and functional marks $F_i(t)\in\mathcal{S}$, $t\in\T\subset[0,\infty)$, which live on some (Polish) space $\mathcal{S}$ other than some Euclidean space; e.g., $\mathcal{S}$ could be a linear network \citep{Baddeley:etal:16,dejby2017capturing} or a sphere \citep{moller2016functional}. 
%For instance, one could be interested in the case where $\mathcal{S}$ is given by a linear network \citep{haenggi:10} and each $F_i(t)$, $t\in\T\subset[0,\infty)$, 
For instance, in the linear network case, each functional mark would describe the movement along $\mathcal{S}$ of the $i$th point/event/individual, whereby we would have a setup for modelling e.g.\ cars driving on a road network during a given time period.

%{\color{red}Is the reference \citep{haenggi:10} the correct one here?}

One could also extend the current setting to having $\T$ be an arbitrary (connected) subset of $\R^d$, for some arbitrary $d\geq1$, so that when $d\geq2$ the variable $t$ in each $F_i(t)$ represents a "spatial" location and $F_i:\T\to\R^k$ is a $k$-variate random field/process. Moreover, this would allow us to let $\T$ be any suitable interval in $\R$, not necessarily a subset of $[0,\infty)$; e.g.\ $\T=\R$.

We have proposed a  general framework to analyse dependent functional data, with an emphasis on the mathematical and statistical aspects of this framework. A wealth of particular cases and models can be treated using our approach, and thus a plethora of real problems can be analysed using this new context. Although only one specific data analytic example has been illustrated here, we believe that we have clearly indicated that many different types of data can be analysed using our framework. 

\section*{Acknowledgements}
The authors are truly grateful to N.M.M.\ van Lieshout (CWI, The Netherlands) for feedback and ideas. The authors are also grateful to Aila S\"arkk\"a (Chalmers University of Technology, Sweden) and Eric Renshaw (University of Strathclyde, U.K.) for ideas, discussions and for introducing us to many of the concepts underlying FMPPs. 
We further thank Mark Hebblewhite (University of Montana, USA) for providing us with the Ya Ha Tinda Elks movements data set.
%This research was supported by the Netherlands Organisation for Scientific Research NWO (613.000.809) and the Spanish Ministry of Education and Science (NTN2010-14961). 

\putbib
\end{bibunit}

% \bibliographystyle{plainnat}
% \bibliography{references}

\providecommand{\noopsort}[1]{}
\begin{thebibliography}{69}
\providecommand{\natexlab}[1]{#1}
\providecommand{\url}[1]{\texttt{#1}}
\expandafter\ifx\csname urlstyle\endcsname\relax
  \providecommand{\doi}[1]{doi: #1}\else
  \providecommand{\doi}{doi: \begingroup \urlstyle{rm}\Url}\fi

\bibitem[Anselin(1995)]{anselin:95}
L.~Anselin.
\newblock Local indicators of spatial association-lisa.
\newblock \emph{Geographical Analysis}, 27:\penalty0 93--115, 1995.

\bibitem[Baddeley(1998)]{baddeley:98}
A.~Baddeley.
\newblock Spatial sampling and censoring.
\newblock In O.~Barndorff-Nielsen, W.~Kendall, and M.~van Lieshout, editors,
  \emph{Stochastic Geometry: Likelihood and Computation}, pages 37--78. Chapman
  and Hall, 1998.

\bibitem[Baddeley et~al.(2000)Baddeley, M{\o}ller, and
  Waagepetersen]{BaddeleyEtAl}
A.~Baddeley, J.~M{\o}ller, and R.~Waagepetersen.
\newblock Non- and semi-parametric estimation of interaction in inhomogeneous
  point patterns.
\newblock \emph{Statistica Neerlandica}, 54:\penalty0 329--350, 2000.

\bibitem[Baddeley et~al.(2016)Baddeley, Rubak, and Turner]{Baddeley:etal:16}
A.~Baddeley, E.~Rubak, and R.~Turner.
\newblock \emph{Spatial Point Patterns: Methodology and Applications with R}.
\newblock CRC, 2016.

\bibitem[Billingsley(1999)]{Billingsley}
P.~Billingsley.
\newblock \emph{Convergence of Probability Measures}.
\newblock Wiley, 2nd edition, 1999.

\bibitem[Chil\`{e}s and Delfiner(2012)]{chiles:delfiner:12}
J.~P. Chil\`{e}s and P.~Delfiner.
\newblock \emph{Geostatistics: Modeling Spatial Uncertainty}.
\newblock Wiley, 2rd edition, 2012.

\bibitem[Chiu et~al.(2013)Chiu, Stoyan, Kendall, and Mecke]{SKM}
S.~N. Chiu, D.~Stoyan, W.~S. Kendall, and J.~Mecke.
\newblock \emph{Stochastic Geometry and its Applications}.
\newblock John Wiley \& Sons, 2013.

\bibitem[Coeurjolly et~al.(2017)Coeurjolly, M{\o}ller, and
  Waagepetersen]{PalmTutorial}
J.~F. Coeurjolly, J.~M{\o}ller, and R.~Waagepetersen.
\newblock A tutorial on palm distributions for spatial point processes.
\newblock \emph{International Statistical Review}, 85:\penalty0 404--420, 2017.
\newblock \doi{10.1111/insr.12205}.

\bibitem[Collins and Cressie(2001)]{CollinsCressie}
L.~Collins and N.~Cressie.
\newblock Analysis of spatial point patterns using bundles of product lisa
  function.
\newblock \emph{Journal of Agricultural, Biological, and Environmental
  Statistics}, 6:\penalty0 118--135, 2001.

\bibitem[Comas(2009)]{Comas}
C.~Comas.
\newblock Modelling forest regeneration strategies through the development of a
  spatio-temporal growth interaction model.
\newblock \emph{Stochastic Environmental Research and Risk Assessment},
  23:\penalty0 1089--1102, 2009.

\bibitem[Comas et~al.(2011)Comas, Delicado, and Mateu]{MateuFMPP}
C.~Comas, P.~Delicado, and J.~Mateu.
\newblock A second order approach to analyse spatial point patterns with
  functional marks.
\newblock \emph{Test}, 20:\penalty0 503--523, 2011.

\bibitem[Cressie and Kornak(2003)]{cressie:kornak:03}
N.~Cressie and J.~Kornak.
\newblock Spatial statistics in the presence of location error with an
  application to remote sensing of the environment.
\newblock \emph{Statistical Science}, 18:\penalty0 436--456, 2003.

\bibitem[Cronie(2012)]{Cronie}
O.~Cronie.
\newblock Likelihood inference for a functional marked point process with
  cox-ingersoll-ross process marks.
\newblock \emph{arXiv}, 2012.

\bibitem[Cronie and {S\"arkk\"a}(2011)]{CronieSarkka}
O.~Cronie and A.~{S\"arkk\"a}.
\newblock Some edge correction methods for marked spatio-temporal point process
  models.
\newblock \emph{Computational Statistics \& Data Analysis}, 55:\penalty0
  2209--2220, 2011.

\bibitem[Cronie and van Lieshout(2015)]{CronieLieshoutSTPP}
O.~Cronie and M.~van Lieshout.
\newblock {A $J$-function for inhomogeneous spatio-temporal point processes}.
\newblock \emph{Scandinavian Journal of Statistics}, 42\penalty0 (2):\penalty0
  562--579, 2015.

\bibitem[Cronie and van Lieshout(2016)]{CronieLieshoutMPP}
O.~Cronie and M.~van Lieshout.
\newblock Summary statistics for inhomogeneous marked point processes.
\newblock \emph{Annals of the Institute of Statistical Mathematics},
  68:\penalty0 905--928, 2016.

\bibitem[Cronie and van Lieshout(2018)]{CvL}
O.~Cronie and M.~van Lieshout.
\newblock A non-model-based approach to bandwidth selection for kernel
  estimators of spatial intensity functions.
\newblock \emph{Biometrika}, 105\penalty0 (2):\penalty0 455--462, 2018.

\bibitem[Cronie et~al.(2013)Cronie, Nystr\"om, and Yu]{CronieForest}
O.~Cronie, K.~Nystr\"om, and J.~Yu.
\newblock Spatiotemporal modeling of swedish scots pine stands.
\newblock \emph{Forest Science}, 59:\penalty0 505--516, 2013.

\bibitem[Daley and Vere-Jones(2003)]{DVJ1}
D.~Daley and D.~Vere-Jones.
\newblock \emph{An Introduction to the Theory of Point Processes: Volume I:
  Elementary Theory and Methods}.
\newblock Springer Series in Statistics, 2nd edition, 2003.

\bibitem[Daley and Vere-Jones(2008)]{DVJ2}
D.~Daley and D.~Vere-Jones.
\newblock \emph{An Introduction to the Theory of Point Processes: General
  Theory and Structure}.
\newblock Springer, 2nd edition, 2008.

\bibitem[Dejby(2017)]{dejby2017capturing}
J.~Dejby.
\newblock Capturing continuous human movement on a linear network with mobile
  phone towers.
\newblock Master's thesis, Ume{\aa} University, Sweden, 2017.

\bibitem[Delicado et~al.(2010)Delicado, Giraldo, Comas, and
  Mateu]{DelicadoGiraldoComas}
P.~Delicado, R.~Giraldo, C.~Comas, and J.~Mateu.
\newblock Statistics for spatial functional data: some recent contributions.
\newblock \emph{Environmetrics}, 21:\penalty0 224--239, 2010.

\bibitem[Deza and Deza(2009)]{deza:deza:09}
M.~M. Deza and E.~Deza.
\newblock \emph{Encyclopedia of Distances}.
\newblock Springer-Verlag, Berlin Heidelberg, 2009.

\bibitem[Diggle(2013)]{Diggle2013}
P.~Diggle.
\newblock \emph{Statistical Analysis of Spatial and spatio-temporal Point
  Patterns}.
\newblock Chapman \& Hall/CRC, Boca Raton, 3nd edition, 2013.

\bibitem[Ethier and Kurtz(1986)]{EthierKurtz}
S.~Ethier and T.~Kurtz.
\newblock \emph{Markov Processes: Characterization and Convergence}.
\newblock Wiley-Interscience, 1986.

\bibitem[Ferraty and Vieu(2006)]{ferraty:vieu:06}
F.~Ferraty and P.~Vieu.
\newblock \emph{Nonparametric Functional Data Analysis: Theory and Practice}.
\newblock Springer Science \& Business Media, Berlin, Germany, 2006.

\bibitem[Gabriel(2014)]{Gabriel2014}
E.~Gabriel.
\newblock Estimating second-order characteristics of inhomogeneous
  spatio-temporal point processes.
\newblock \emph{Methodology and Computing in Applied Probability}, 16\penalty0
  (2):\penalty0 411--431, 2014.

\bibitem[Gelfand et~al.(2010)Gelfand, Diggle, Guttorp, and
  Fuentes]{Gelfand:etal:10}
A.~E. Gelfand, P.~J. Diggle, P.~Guttorp, and M.~Fuentes.
\newblock \emph{Handbook of Spatial Statistics}.
\newblock CRC Press, Boca Raton, 2010.

\bibitem[Getis and Franklin(1987)]{Getis:Franklin:87}
A.~Getis and J.~Franklin.
\newblock Second-order neighborhood analysis of mapped point patterns.
\newblock \emph{Ecology}, 68:\penalty0 473--477, 1987.

\bibitem[Giraldo et~al.(2010)Giraldo, Delicado, and
  Mateu]{GiraldoDelicadoMateu2010}
R.~Giraldo, P.~Delicado, and J.~Mateu.
\newblock Continuous time-varying kriging for spatial prediction of functional
  data: An environmental application.
\newblock \emph{Journal of Agricultural, Biological, and Environmental
  Statistics}, 15:\penalty0 66--82, 2010.

\bibitem[Giraldo et~al.(2011)Giraldo, Delicado, and
  Mateu]{GiraldoDelicadoMateu2011}
R.~Giraldo, P.~Delicado, and J.~Mateu.
\newblock Ordinary kriging for function-valued spatial data.
\newblock \emph{Journal of Agricultural, Biological, and Environmental
  Statistics}, 15:\penalty0 66--82, 2011.

\bibitem[Grabarnik et~al.(2011)Grabarnik, Myllymaki, and Stoyan]{Grabarnik2011}
P.~Grabarnik, M.~Myllymaki, and D.~Stoyan.
\newblock Correct testing of mark independence for marked point patterns.
\newblock \emph{Ecological Modelling}, 222:\penalty0 3888--3894, 2011.

\bibitem[Hebblewhite and Merrill(2008)]{Hebblewhite:Merrill:08}
M.~Hebblewhite and E.~H. Merrill.
\newblock Modelling wildlife-human relationships for social species with
  mixed-effects resource selection models.
\newblock \emph{Journal of Applied Ecology}, 45:\penalty0 834--844, 2008.

\bibitem[Heinrich(2013)]{Heinrich2013}
L.~Heinrich.
\newblock \emph{Asymptotic Methods in Statistics of Random Point Processes},
  pages 115--150.
\newblock Springer, Berlin, Heidelberg, 2013.

\bibitem[Hijmans(2019)]{hijmans:19}
R.~J. Hijmans.
\newblock raster: Geographic data analysis and modeling.
\newblock \emph{R package version 2.8-19}, 2019.
\newblock \doi{https://cran.r-project.org/web/packages/raster/raster.pdf}.

\bibitem[Ho and Stoyan(2008)]{HoStoyan}
L.~Ho and D.~Stoyan.
\newblock Modelling marked point patterns by intensity-marked cox processes.
\newblock \emph{Statistics \& Probability Letters}, 78:\penalty0 1194--1199,
  2008.

\bibitem[Iftimi et~al.(2018)Iftimi, Cronie, and Montes]{Iftimi}
A.~Iftimi, O.~Cronie, and F.~Montes.
\newblock Second-order analysis of marked inhomogeneous spatio-temporal point
  processes: Applications to earthquake data.
\newblock \emph{Scandinavian Journal of Statistics}, 46:\penalty0 661--685,
  2018.

\bibitem[Illian et~al.(2008)Illian, Penttinen, Stoyan, and Stoyan]{Illian}
J.~Illian, A.~Penttinen, H.~Stoyan, and D.~Stoyan.
\newblock \emph{Statistical Analysis and Modelling of Spatial Point Patterns}.
\newblock Wiley, 2008.

\bibitem[Jacod and Shiryaev(1987)]{JacodShiryaev}
J.~Jacod and A.~Shiryaev.
\newblock \emph{Limit Theorems for Stochastic Processes}.
\newblock Springer, 1987.

\bibitem[Kallenberg(2006)]{kallenberg2006foundations}
O.~Kallenberg.
\newblock \emph{Foundations of Modern Probability}.
\newblock Springer Science \& Business Media, 2006.

\bibitem[{\noopsort{Lieshout}}van~Lieshout(2000)]{VanLieshout}
M.~{\noopsort{Lieshout}}van~Lieshout.
\newblock \emph{Markov Point Processes and Their Applications}.
\newblock Imperial College Press, London, 2000.

\bibitem[{\noopsort{Lieshout}}van~Lieshout(2006)]{MCmarked}
M.~{\noopsort{Lieshout}}van~Lieshout.
\newblock {A $J$-function for marked point patterns}.
\newblock \emph{Annals of the Institute of Statistical Mathematics},
  58:\penalty0 235--259, 2006.

\bibitem[{\noopsort{Lieshout}}van~Lieshout(2011)]{VanLieshoutJfunction}
M.~{\noopsort{Lieshout}}van~Lieshout.
\newblock {A $J$-function for inhomogeneous point processes}.
\newblock \emph{Statistica Neerlandica}, 65:\penalty0 183--201, 2011.

\bibitem[Maniglia and Rhandi(2004)]{maniglia2004gaussian}
S.~Maniglia and A.~Rhandi.
\newblock Gaussian measures on separable hilbert spaces and applications.
\newblock \emph{Quaderni di Matematica}, 2004\penalty0 (1), 2004.

\bibitem[Mateu et~al.(2007)Mateu, Lorenzo, and Porcu]{MateuLorenzoPorcu}
J.~Mateu, G.~Lorenzo, and E.~Porcu.
\newblock Detecting features in spatial point processes with clutter via local
  indicators of spatial association.
\newblock \emph{Journal of Computational and Graphical Statistics},
  16:\penalty0 968--990, 2007.

\bibitem[M{\o}ller and Rubak(2016)]{moller2016functional}
J.~M{\o}ller and E.~Rubak.
\newblock Functional summary statistics for point processes on the sphere with
  an application to determinantal point processes.
\newblock \emph{Spatial Statistics}, 18:\penalty0 4--23, 2016.

\bibitem[M{\o}ller and Waagepetersen(2004)]{Moller}
J.~M{\o}ller and R.~Waagepetersen.
\newblock \emph{Statistical Inference and Simulation for Spatial Point
  Processes}.
\newblock Chapman \& Hall/CRC Press, 2004.

\bibitem[M{\o}ller et~al.(1998)M{\o}ller, Syversveen, and
  Waagepetersen]{MollerSyversveen}
J.~M{\o}ller, A.~Syversveen, and R.~Waagepetersen.
\newblock Log gaussian cox processes.
\newblock \emph{Scandinavian Journal of Statistics}, 25:\penalty0 451--482,
  1998.

\bibitem[M{\o}ller et~al.(2016)M{\o}ller, Ghorbani, and Rubak]{moeller:etal:16}
J.~M{\o}ller, M.~Ghorbani, and E.~Rubak.
\newblock Mechanistic spatio-temporal point process models for marked point
  processes, with a view to forest stand data.
\newblock \emph{Biometrics}, 72:\penalty0 687--696, 2016.

\bibitem[Montero et~al.(2015)Montero, Fern\'{a}ndez-Avil\'{e}s, and
  Mateu]{Montero:etal:15}
J.~M. Montero, G.~Fern\'{a}ndez-Avil\'{e}s, and J.~Mateu.
\newblock \emph{Spatial and Spatio-Temporal Geostatistical Modeling and
  Kriging}.
\newblock John Wiley and Sons, 2015.

\bibitem[Moradi et~al.(2018)Moradi, Cronie, Rubak, Lachieze-Rey, Mateu, and
  Baddeley]{Moradi}
M.~Moradi, O.~Cronie, E.~Rubak, R.~Lachieze-Rey, J.~Mateu, and A.~Baddeley.
\newblock Resample-smoothing of voronoi intensity estimators.
\newblock \emph{Statistics \& Computing}, 2018.
\newblock \doi{10.1007/s11222-018-09850-0}.

\bibitem[Myllym\"aki and Penttinen(2009)]{MyllymakiPenttinen}
M.~Myllym\"aki and A.~Penttinen.
\newblock Conditionally heteroscedastic intensity-dependent marking of log
  gaussian cox processes.
\newblock \emph{Statistica Neerlandica}, 63:\penalty0 450--473, 2009.

\bibitem[Penttinen and Stoyan(1989)]{penttinen:stoyan:89}
A.~Penttinen and D.~Stoyan.
\newblock Statistical analysis for a class of line segment processes.
\newblock \emph{Scandinavian Journal of Statistics}, 16:\penalty0 153--168,
  1989.

\bibitem[Platt et~al.(1988)Platt, Evans, and Rathbun]{platt:etal:88}
W.~J. Platt, G.~W. Evans, and S.~L. Rathbun.
\newblock The population dynamics of a long-lived conifer (pinus palustris).
\newblock \emph{The American Naturalist}, 131:\penalty0 491--525, 1988.

\bibitem[Rajput(1972)]{rajput1972gaussian}
B.~S. Rajput.
\newblock Gaussian measures on {$L_p$} spaces, $1\le p<\infty$.
\newblock \emph{Journal of Multivariate Analysis}, 2\penalty0 (4):\penalty0
  382--403, 1972.

\bibitem[Ramsay and Silverman(2005)]{RamsaySilverman2005}
J.~Ramsay and B.~Silverman.
\newblock \emph{Functional Data Analysis}.
\newblock Springer, 2nd edition, 2005.

\bibitem[Renshaw and Comas(2009)]{RenshawComas}
E.~Renshaw and C.~Comas.
\newblock Space-time generation of high intensity patterns using
  growth-interaction processes.
\newblock \emph{Statistics and Computing}, 19:\penalty0 423--437, 2009.

\bibitem[Renshaw and {S\"arkk\"a}(2001)]{RS1}
E.~Renshaw and A.~{S\"arkk\"a}.
\newblock Gibbs point processes for studying the development of
  spatial-temporal stochastic processes.
\newblock \emph{Computational Statistics \& Data Analysis}, 36:\penalty0
  85--105, 2001.

\bibitem[Renshaw et~al.(2009)Renshaw, Comas, and Mateu]{RenshawComasMateu}
E.~Renshaw, C.~Comas, and J.~Mateu.
\newblock Analysis of forest thinning strategies through the development of
  space-time growth-interaction simulation models.
\newblock \emph{Stochastic Environmental Research and Risk Assessment},
  23:\penalty0 275--288, 2009.

\bibitem[Ripley(1976)]{RipleyK}
B.~Ripley.
\newblock The second-order analysis of stationary point processes.
\newblock \emph{Journal of Applied Probability}, 13:\penalty0 255--266, 1976.

\bibitem[{S\"arkk\"a} and Renshaw(2006)]{RS2}
A.~{S\"arkk\"a} and E.~Renshaw.
\newblock The analysis of marked point patterns evolving through space and
  time.
\newblock \emph{Computational Statistics \& Data Analysis}, 51:\penalty0
  1698--1718, 2006.

\bibitem[Schlather et~al.(2004)Schlather, {Ribeiro Jr.}, and
  Diggle]{Schlather2004}
M.~Schlather, P.~{Ribeiro Jr.}, and P.~Diggle.
\newblock Detecting dependence between marks and locations of marked point
  processes.
\newblock \emph{Journal of the Royal Statistical Society B}, 66:\penalty0
  79--93, 2004.

\bibitem[Schneider and Weil(2008)]{SchneiderWeil}
R.~Schneider and W.~Weil.
\newblock \emph{Stochastic and Integral Geometry}.
\newblock Springer, 2008.

\bibitem[Sebastian et~al.(2006)Sebastian, D\'{i}az, D\'{i}az, Ayala, R., and
  Toomre]{Sebastian:etal:06}
R.~Sebastian, E.~D\'{i}az, M.~E. D\'{i}az, G.~Ayala, Z.~R., and D.~Toomre.
\newblock Studying endocytosis in space and time by means of temporal boolean
  models.
\newblock \emph{Pattern Recognition}, 39:\penalty0 2175--2185, 2006.

\bibitem[Silvestrov(2004)]{Silvestrov}
S.~Silvestrov.
\newblock \emph{Limit Theorems for Randomly Stopped Stochastic Processes}.
\newblock Springer, 2004.

\bibitem[Skorohod(1967)]{Skorohod}
A.~Skorohod.
\newblock On the densities of probability measures in functional spaces.
\newblock \emph{Proc.\ Fifth Berkeley Symp.\ on Math.\ Statist.\ and Prob.},
  2:\penalty0 163--182, 1967.

\bibitem[Stoyan and Stoyan(1994)]{StoyanStoyan}
D.~Stoyan and H.~Stoyan.
\newblock \emph{Fractals, Random Shapes and Point Fields}.
\newblock Wiley, 1994.

\bibitem[Stoyan and Stoyan(2000)]{StoyanStoyanRatio}
D.~Stoyan and H.~Stoyan.
\newblock Improving ratio estimators of second order point process
  characteristics.
\newblock \emph{Scandinavian Journal of Statistics}, 27:\penalty0 641--656,
  2000.

\bibitem[Zhao and Wang(2010)]{Zhao}
J.~Zhao and J.~Wang.
\newblock Asymptotic properties of an empirical {K}-function for inhomogeneous
  spatial point processes.
\newblock \emph{Statistics}, 44\penalty0 (3):\penalty0 261--267, 2010.

\end{thebibliography}


\providecommand{\noopsort}[1]{}
\begin{thebibliography}{28}
\providecommand{\natexlab}[1]{#1}
\providecommand{\url}[1]{\texttt{#1}}
\expandafter\ifx\csname urlstyle\endcsname\relax
  \providecommand{\doi}[1]{doi: #1}\else
  \providecommand{\doi}{doi: \begingroup \urlstyle{rm}\Url}\fi

\bibitem[Baum and Kalashnikov(2001)]{Baum}
D.~Baum and V.~Kalashnikov.
\newblock Stochastic models for communication networks with moving customers.
\newblock \emph{Information Processes}, 1:\penalty0 1--23, 2001.

\bibitem[Billingsley(1995)]{Billingsley1995PM}
P.~Billingsley.
\newblock \emph{Probability and Measures}.
\newblock Wiley, 3rd edition, 1995.

\bibitem[Billingsley(1999)]{Billingsley}
P.~Billingsley.
\newblock \emph{Convergence of Probability Measures}.
\newblock Wiley, 2nd edition, 1999.

\bibitem[Chiu et~al.(2013)Chiu, Stoyan, Kendall, and Mecke]{SKM}
S.~N. Chiu, D.~Stoyan, W.~S. Kendall, and J.~Mecke.
\newblock \emph{Stochastic Geometry and its Applications}.
\newblock John Wiley \& Sons, 2013.

\bibitem[Comas(2009)]{Comas}
C.~Comas.
\newblock Modelling forest regeneration strategies through the development of a
  spatio-temporal growth interaction model.
\newblock \emph{Stochastic Environmental Research and Risk Assessment},
  23:\penalty0 1089--1102, 2009.

\bibitem[Comas et~al.(2011)Comas, Delicado, and Mateu]{MateuFMPP}
C.~Comas, P.~Delicado, and J.~Mateu.
\newblock A second order approach to analyse spatial point patterns with
  functional marks.
\newblock \emph{Test}, 20:\penalty0 503--523, 2011.

\bibitem[Cr\'et\'e et~al.(2013)Cr\'et\'e, Pumo, Soubeyrand, Didelot, and
  V.]{AppleTrees}
R.~Cr\'et\'e, B.~Pumo, S.~Soubeyrand, F.~Didelot, and C.~V.
\newblock A continuous time-and-state epidemic model fitted to ordinal
  categorical data observed on a lattice at discrete times.
\newblock \emph{Journal of Agricultural, Biological, and Environmental
  Statistics}, 18:\penalty0 538--555, 2013.

\bibitem[Cronie(2012)]{Cronie}
O.~Cronie.
\newblock Likelihood inference for a functional marked point process with
  cox-ingersoll-ross process marks.
\newblock \emph{arXiv}, 2012.

\bibitem[Cronie and {S\"arkk\"a}(2011)]{CronieSarkka}
O.~Cronie and A.~{S\"arkk\"a}.
\newblock Some edge correction methods for marked spatio-temporal point process
  models.
\newblock \emph{Computational Statistics \& Data Analysis}, 55:\penalty0
  2209--2220, 2011.

\bibitem[Cronie et~al.(2013)Cronie, Nystr\"om, and Yu]{CronieForest}
O.~Cronie, K.~Nystr\"om, and J.~Yu.
\newblock Spatiotemporal modeling of swedish scots pine stands.
\newblock \emph{Forest Science}, 59:\penalty0 505--516, 2013.

\bibitem[Daley and Vere-Jones(2003)]{DVJ1}
D.~Daley and D.~Vere-Jones.
\newblock \emph{An Introduction to the Theory of Point Processes: Volume I:
  Elementary Theory and Methods}.
\newblock Springer Series in Statistics, 2nd edition, 2003.

\bibitem[Ethier and Kurtz(1986)]{EthierKurtz}
S.~Ethier and T.~Kurtz.
\newblock \emph{Markov Processes: Characterization and Convergence}.
\newblock Wiley-Interscience, 1986.

\bibitem[Jacod and Shiryaev(1987)]{JacodShiryaev}
J.~Jacod and A.~Shiryaev.
\newblock \emph{Limit Theorems for Stochastic Processes}.
\newblock Springer, 1987.

\bibitem[Kallenberg(2006)]{kallenberg2006foundations}
O.~Kallenberg.
\newblock \emph{Foundations of Modern Probability}.
\newblock Springer Science \& Business Media, 2006.

\bibitem[Klebaner(2005)]{Klebaner}
F.~Klebaner.
\newblock \emph{Introduction to Stochastic Calculus with Applications}.
\newblock Imperial College Press, 2nd edition, 2005.

\bibitem[{\noopsort{Lieshout}}van~Lieshout(2000)]{VanLieshout}
M.~{\noopsort{Lieshout}}van~Lieshout.
\newblock \emph{Markov Point Processes and Their Applications}.
\newblock Imperial College Press, London, 2000.

\bibitem[Maniglia and Rhandi(2004)]{maniglia2004gaussian}
S.~Maniglia and A.~Rhandi.
\newblock Gaussian measures on separable hilbert spaces and applications.
\newblock \emph{Quaderni di Matematica}, 2004\penalty0 (1), 2004.

\bibitem[M{\o}ller and Waagepetersen(2004)]{Moller}
J.~M{\o}ller and R.~Waagepetersen.
\newblock \emph{Statistical Inference and Simulation for Spatial Point
  Processes}.
\newblock Chapman \& Hall/CRC Press, 2004.

\bibitem[M{\o}ller et~al.(1998)M{\o}ller, Syversveen, and
  Waagepetersen]{MollerSyversveen}
J.~M{\o}ller, A.~Syversveen, and R.~Waagepetersen.
\newblock Log gaussian cox processes.
\newblock \emph{Scandinavian Journal of Statistics}, 25:\penalty0 451--482,
  1998.

\bibitem[M\"orters and Peres(2010)]{MortersPeresBMbook}
P.~M\"orters and Y.~Peres.
\newblock \emph{Brownian Motion}.
\newblock Cambridge University Press, 2010.

\bibitem[Rajput(1972)]{rajput1972gaussian}
B.~S. Rajput.
\newblock Gaussian measures on {$L_p$} spaces, $1\le p<\infty$.
\newblock \emph{Journal of Multivariate Analysis}, 2\penalty0 (4):\penalty0
  382--403, 1972.

\bibitem[Renshaw and Comas(2009)]{RenshawComas}
E.~Renshaw and C.~Comas.
\newblock Space-time generation of high intensity patterns using
  growth-interaction processes.
\newblock \emph{Statistics and Computing}, 19:\penalty0 423--437, 2009.

\bibitem[Renshaw and {S\"arkk\"a}(2001)]{RS1}
E.~Renshaw and A.~{S\"arkk\"a}.
\newblock Gibbs point processes for studying the development of
  spatial-temporal stochastic processes.
\newblock \emph{Computational Statistics \& Data Analysis}, 36:\penalty0
  85--105, 2001.

\bibitem[Renshaw et~al.(2009)Renshaw, Comas, and Mateu]{RenshawComasMateu}
E.~Renshaw, C.~Comas, and J.~Mateu.
\newblock Analysis of forest thinning strategies through the development of
  space-time growth-interaction simulation models.
\newblock \emph{Stochastic Environmental Research and Risk Assessment},
  23:\penalty0 275--288, 2009.

\bibitem[{S\"arkk\"a} and Renshaw(2006)]{RS2}
A.~{S\"arkk\"a} and E.~Renshaw.
\newblock The analysis of marked point patterns evolving through space and
  time.
\newblock \emph{Computational Statistics \& Data Analysis}, 51:\penalty0
  1698--1718, 2006.

\bibitem[Schladitz and Baddeley(2000)]{schladitz:Baddeley:00}
K.~Schladitz and A.~Baddeley.
\newblock A third order point process characteristic.
\newblock \emph{Scandinavian Journal of Statistics}, 27:\penalty0 657--671,
  2000.

\bibitem[Silvestrov(2004)]{Silvestrov}
S.~Silvestrov.
\newblock \emph{Limit Theorems for Randomly Stopped Stochastic Processes}.
\newblock Springer, 2004.

\bibitem[Skorohod(1967)]{Skorohod}
A.~Skorohod.
\newblock On the densities of probability measures in functional spaces.
\newblock \emph{Proc.\ Fifth Berkeley Symp.\ on Math.\ Statist.\ and Prob.},
  2:\penalty0 163--182, 1967.

\end{thebibliography}

\pagebreak
\begin{bibunit}

\appendix
\appendixpage

% \beginsupplement
% \section{Supplementary Material}

\section{The (stochastic) growth-interaction process}\label{SectionGI}
As mentioned in Section \ref{s:DeterministicFunctionalMarks}, one of the models which has given rise to a substantial part of the ideas underlying the construction of FMPPs is the \emph{growth-interaction process}. Originally defined by Renshaw and S{\"a}rkk{\"a} \citep{RS1,RS2}, it has been extensively studied in a series of papers \citep{Comas,MateuFMPP,Cronie,CronieSarkka,CronieForest,RenshawComas,RenshawComasMateu,RS1,RS2}, mainly within the forestry context. However, its representation as a functional marked point process has only been noted in \Citep{MateuFMPP,Cronie}. 

A growth-interaction process $\Psi$ is a spatio-temporal FMPP with $\Psi_{\X\times\A}=\{(X_i,T_i,L_i)\}_{i=1}^N$ and $k=1$, so that $\F=\U$. When the spatial domain is bounded, which is the case in all of the above references, the ground process $\Psi_G=\{X_i\}_{i=1}^N\subset\X$ is generated by a homogeneous Poisson process with intensity $\lambda>0$. 
% , where $\{X_i\}_{i=1}^N$
% , for which the ground process $\Psi_G$ is generated by a homogeneous Poisson process with intensity $\lambda>0$. 
Conditionally on $\Psi_G$, the auxiliary marks are given by $L_i=(T_i,D_i)\in\A_c=[0,\infty)^2$, $i=1,\ldots,N$, where the $T_i$'s are iid $Uni(\T)$-distributed {\em arrival times} and 
$D_i=T_i+\xi_i$,
%$D_i=(T_i+\xi_i)\wedge T^*$, 
where the $\xi_i$'s are iid $Exp(\mu)$-distributed, $\mu>0$, {\em death times}. 
%
%spatial birth-death process, which has Poisson arrivals $T_i$, with intensity $\alpha>0$, and uniformly distributed spatial locations $X_i$. Furthermore, the auxiliary marks are the associated holding times $L_i$, which are independently $Exp(\mu)$-distributed, $\mu>0$, and, 
Turning to $\Psi|\Psi_{\X\times\A}$, conditionally on $\Psi_{\X\times\A}$ the functional marks are governed by a system of ordinary differential equations,
\beann
\frac{dF_i(t)}{dt}=h(F_i(t);\theta) - %\sum_{(X_j,T_j,D_j,F_j(t))\in\Psi,\ j\neq i} 
\sum_{j=1}^{N}
\1\{j\neq i\}
\bar h((X_i,T_i,D_i,F_i(t)),(X_j,T_j,D_j,F_j(t));\theta),
\quad
t\in\supp(F_i)=[T_i,D_i),
\eeann
$i=1,\ldots,N$. 
%, where $t\in\supp(F_i)=[T_i,D_i)$. %, $D_i=(T_i+L_i)\wedge T^*$. 
Here $h(\cdot)$ represents the individual growth of the $i$th \emph{individual}, in absence of spatial interaction with other individuals, and $\bar h(\cdot)$ is the amount of spatial interaction to which individual $i$ is subjected by individual $j$ during the infinitesimal interval $[t,t+dt]$.

As can be found in the above-mentioned references, the usual application of this model is the modelling of the collective development of trees in a forest stand; $X_i$ is the location of the $i$th tree, $T_i$ is its birth time, $D_i$ is its death time, and $F_i(t)$ represents its radius (at breast height) at time $t$. 

As one may argue that this approach does not incorporate individual growth features in the radial growth sufficiently well, \citet{Cronie} suggested that a scaled white noise processes should be added to each functional mark equation, i.e., conditionally on $\Psi_{\X\times\A}$, we would instead consider functional marks
$$
dF_i^*(t) = dF_i(t) + \sigma(F_i(t);\theta)dW_i(t),
$$
where $W_1(t),\ldots,W_N(t)$, are independent standard Brownian motions and $\sigma(\cdot)$ is some suitable diffusion coefficient. 
Here the noise would represent measurement errors and give rise to individual growth deviations. 
The resulting stochastic differential equation marked point process, the \emph{stochastic growth-interaction process}, was then studied in the simplified case where the spatial interaction is negligible, i.e.\ $\bar h(\cdot)\equiv0$. 

A further extension of the model, to the multivariate setting, would be obtained by letting $L_i=(S_i,T_i,D_i)\in \A=\A_d\times\A_c=\{1,\ldots,k_d\}\times[0,\infty)^2$, where $S_i$ would represent the specie of the $i$th tree. The individual growth will here change to $$
h(F_i(t);\theta)=\sum_{l=1}^{k_d}\1\{S_i=l\}h_{l}(F_i(t);S_i,\theta)
$$ 
and the interaction $\bar h((X_i,T_i,D_i,F_i(t)),(X_j,T_j,D_j,F_j(t));\theta)$ will be given by 
\begin{align*}
\sum_{l=1}^{k_d}\sum_{m=1}^{k_d}\1\{S_i=l, S_j=m\} \bar h_{lm}((X_i,T_i,D_i,F_i(t)),(X_j,T_j,D_j,F_j(t));\theta),
\end{align*}
for species specific functions $h_{l}(\cdot)$ and $\bar h_{lm}(\cdot)$, $l,m=1,\ldots,k_d$. In other words, the growths and interactions depend explicitly on the species. 

\section{Examples of applications}\label{SectionApplications}
Besides the applications mentioned in the main text, we here give a list of further possible applications of FMPPs, providing a wide scope of the current framework.

\begin{enumerate}
\item 
%{\bf We need to go over this again!}
{\em Modelling individual/animal movements:} 
Spatial movement data sets include animal (e.g.\ elk) movements, car movements and eye movements, to name a few examples. 
Whether we are modelling the movements of a group of individuals or the movement of a specific individual (recall the lower row of Figure~\ref{Illu}), the $i$th path, $i=1,\ldots,N$, may be described by
\beann
F_i(t) = \1\{t\leq T_i\}X_i + \1\{t\in(T_i,D_i)\}F_i^*((t-T_i)\wedge0) + \1\{t\geq D_i\}F_i(D_i)
\in\X\subset\R^2, 
\quad t\in\T,
\eeann
where $X_i\in\X\subset\R^2$ is the starting location of the $i$th path/piece, $T_i\in\T$ is the associated starting time, $D_i\in\T$, $D_i\stackrel{a.s}{>}T_i$, is the associated end time and $F_i^*(t)=(F_{i1}^*(t), F_{i2}^*(t))\in\R^2$, $t\geq0$, $F_i^*(0)=X_i$, is some continuous spatial stochastic process describing the actual path; here $F_i=(F_{i1}, F_{i2})\in\F=\U^2$, where $F_{i1}=\{F_{i1}(t)\}_{t\in\T}$ and $F_{i2}=\{F_{i2}(t)\}_{t\in\T}$ control the $x$-axis and $y$-axis displacements, respectively. 
Note that the $i$th movement only consists of the spatial point $X_i$ for $t\leq T_i$ and it is absorbed in $F_i(D_i)\in\X$ once $t=D_i$. Here, $\Psi_G=\{(X_i,T_i)\}_{i=1}^N$ constitutes a spatio-temporal point process to which we assign auxiliary marks given by the end times $D_i$; if this point process is finite then we may instead let the auxiliary marks be given by $L_i=(T_i,D_i)\in\A=\T^2$ and the ground process by $\Psi_G=\{X_i\}_{i=1}^N$. What essentially sets the group movement modelling apart from the individual movement modelling is what we associate each of the above components with:
\begin{enumerate}
    \item {\em Movements of a group of individuals:} Here each index $i=1,\ldots,N$ indicates an individual, $X_i$ the location where it was first observed during the study period $\T$, $T_i\in\T$ the time point at which it first started moving during $\T$ and $D_i$ the time at which it stops moving, which happens at the location $F_i(D_i)$. 
    Note that since we assume $N$ to be random, we also make the assumption that we do not know a priori how many individuals we will observe during $\T$ -- we may always condition on $N=n\geq1$.
    An illustrative example is provided by the lower row of Figure~\ref{Illu}. 
    
    \item {\em Modelling the movement of only one individual, who stops at different locations:} Here $\{(X_i,T_i)\}_{i=1}^N$ describes the $N$ locations and times at which the individual stops during the time interval $\T$. The end times satisfy $F_i(D_i)=X_{i+1}$ (the individual moves between $X_i$ and $X_{i+1}$) and $T_i<D_i<T_{i+1}<D_{i+1}$ a.s.\ for any $i=1,\ldots,N-1$; note that the strict inequality $D_i<T_{i+1}$ ensures that there is actually a stop at location $X_{i+1}$ and $T_{i+1}-D_i$ is the amount of time spent at location $X_{i+1}$. 
    %An illustrative example is provided by one of the paths in the left-lower panel of Figure~\ref{Illu}. 
    
    Note that we may also accommodate analysing $n\geq1$ different individuals in the above fashion by considering a vector of $n$ different such FMPPs to obtain a multivariate FMPP. This may be superpositioned and treated as a multi-type FMPP, where we keep track of a specific individual's index by adding the component $\{1,\ldots,n\}$ to the auxiliary mark space.
    
\end{enumerate}
As monitoring (through e.g.\ GPS) happens discretely in practice, $F_i^*$ can be approximated in a number of ways, e.g.\ by means of line segments or basis expansions etc, and thus capture the main shape of the path/curve. 

If it is the case that the actual spatial movement path has not been recorded, or if the movements are essentially straight lines, we may replace the spatial functional movement mark above by the the total variation/arc length function of the $i$th movement, as it represents the distance travelled by individual $i$ up to time $t$, having started from the random location $X_i$. Note that the functional marks with which we mark $\Psi_{\X\times\A}=\{(X_i,T_i,D_i)\}_{i=1}^N$ here, i.e.\ the total variation functions, take values in $[0,\infty)$ as opposed to in $\X$. Here it may also be relevant to add the individual movement directions as auxiliary marks, since anisotropy may have to be accounted for/analysed.

\item {\em Spread of pollutant:} 
$X_i$ is the pollution location, $F_i(h)$ gives us the ground concentration of the contaminant at distance $h=\|X_i-x\|$, $x\in\X$, from $X_i$ and the auxiliary mark $L_i$ is the type of contaminant considered, provided that there are different types of contaminants present.

\item {\em Modelling tumours:} 
$\X$ represents a region in the human body,  
$X_i\in\X$ is the location of the centre of the $i$th tumour and $F_i(t)$ is its approximate volume/area at time $t$.

\item {\em Disease incidences in epidemics:} 
Each $F_i(t)$ is a stochastic process with piecewise constant sample paths (e.g.\ a Poisson process), which counts the number of incidences having occurred by time $t$ at the epidemic centre $X_i$. %Different types of epidemics may be considered jointly by including auxiliary marks as type classification.

\item {\em Population growth:} 
$X_i$ is the location of a village/town/city, $L_i$ the time point at which it was founded and $F_i(t)$ its total population at time $t$.

\item {\em Mobile communication:} 
Consider an FMPP $\Psi$ where each $X_i\in\X\subset\R^2$ represents the location of a cellphone caller who makes a call at time $T_i$, which lasts until $D_i=T_i+L_i$, where the auxiliary mark $L_i$ represents the duration of the call.  
Then the function $F_i(t)=\1_{[T_i,D_i)}(t)$ represents the phone call in question. The total load on a server/antenna located at $s\in\X$, which has spatial reach within the region $B\subset\X$, $s\in B$, is provided by $N_s(t)=\sum_{i=1}^{N}\1_B(X_i)F_i(t)$. Assuming that the server has capacity $c_s(t)$ at time $t$, it breaks down if $\sup_{t\in\T}c_s(t)-N_s(t)\leq0$. 
Note the connection with \citet{Baum}.

An extension here could be to let $F_i(t) = \Gamma_i\1_{[T_i,D_i)}(t)$ for some random quantity $\Gamma_i=\Gamma_i(X_i,T_i,D_i)$, which represents the specific load that call $i$ puts on the network.

\end{enumerate}

\section{Specific auxiliary and functional mark space choices}
\label{s:MarkSpaceChoices}
We here look closer at a few different choices for the auxiliary mark space $\A$ and the functional mark space $\F$, as well as the reference measures $\nu_{\A}$ and $\nu_{\F}$.

\subsection{Auxiliary mark spaces}
\label{s:AuxiliaryMarkSpace}

Recall that the auxiliary mark space is given by $\A\subset\R^{k_{\A}}$, $k_{\A}\geq1$. This implies that each auxiliary mark $L_i=(L_{1i},\ldots,L_{k_{\A}i})$ is given by a $k_{\A}$-dimensional random vector. 
We here provide a couple of illustrative examples:
%Depending on how $\A$ and the distributions of the $L_i$'s are specified, we are able to consider an array of different settings, where some examples include:
\begin{enumerate}
\item[i)] Type classifications/labels: $k_{\A}=1$ and $\A\subset\R$ is a discrete space, e.g.\ $\{1,\ldots,k_d\}$, $k_d\geq2$, whereby each random variable 
% the distributions of the random variables 
$L_i$ has a discrete distribution on $\A$. 
Recall from Section \ref{s:MarkSpaceChoices} that we refer to this as the \emph{multi-type/multivariate} setting, since here $\Psi_{\X\times\A}$ hereby becomes a multi-type/multivariate point process in $\R^d$.
% are discrete with support on the integers $\Z\cap\A$. 
%Since $\Psi_{\X\times\A}$ hereby becomes a multi-type/multivariate point process in $\R^d$, we call such FMPPs \emph{multi-type/multivariate} \citep{DVJ1,VanLieshout,Gelfand:etal:10}.

\item[ii)] Continuous auxiliary information: $k_{\A}\geq1$ and the distributions of the random vectors $L_i=(L_{1i},\ldots,L_{k_{\A}i})\in\A\subset\R^{k_{\A}}$ are continuous. This corresponds to e.g.\ some additional temporal information, such as a \emph{lifetime} which controls the support of the functional mark. Note that here $\Psi_{\X\times\A}$ becomes a marked point process in $\R^d$ with continuous real valued marks in $\A\subset\R^{k_{\A}}$. 

\item[iii)] A combination of i) and ii): $k_{\A}=k_{\A_d}+k_{\A_c}$, $k_{\A_d},k_{\A_c}\geq1$, so that 
$$
L_i=(L_{1i},\ldots,L_{k_{\A_d}i},
L_{(k_{\A_d}+1)i},\ldots,L_{k_{\A}i}
)\in\A=\A_d\times\A_c\subset\R^{k_{\A_d}}\times\R^{k_{\A_c}}=\R^{k_{\A}},
$$
where $L_{1i},\ldots,L_{k_{\A_d}i}$ are discrete random variables on the discrete space $\A_d$ and $L_{(k_{\A_d}+1)i},\ldots,L_{k_{\A}i}$ are continuous random variables on $\A_c$; the above marginal random variables may naturally be dependent. Here $\Psi_{\X\times\A}$ becomes a marked multivariate point process in $\R^d$ and exemplifying through trees, when $k_{\A_d}=k_{\A_c}=1$, we obtain that different types of trees may have different lifetimes. 
\end{enumerate}

Recall that the choice of $\A$ affects how we choose the reference measure $\nu_{\A}$ on $\A$; we require that $\nu_{\A}(\A)<\infty$. 
%If we e.g.\ choose $\A=\R^{k_{\A}}$, then we must require that $\nu_{\A}$ is a finite (probability) measure, and if $\A$ is non-discrete and bounded then we may use e.g.\ the Lebesgue measure on $\R^{k_{\A}}$ as reference measure.
%A rigorous mathematical description of how the choice of the auxiliary mark space $\A$ and the auxiliary mark reference measure $\nu_{\A}$ affect the distribution of the auxiliary marks, and thereby the moment characteristics of FMMPs, is given in Section~\ref{auxiliaryMarksDistribution}. 
To exemplify how to choose the auxiliary mark reference measure $\nu_{\A}$, taking the scenarios above into account, when $\A=\A_d\times\A_c\subset\R^{k_{\A_d}}\times\R^{k_{\A_c}}$ is as in iii), %in Section \ref{s:AuxiliaryMarkSpace} 
we will assume that 
% each component measure in $\nu$ governs the probabilistic structures of $\Psi$ on $\X$, $\A$ and $\F$, respectively, through its finite dimensional distribution. 
% Here 
it is given by the product measure 
$\nu_{\A}=\nu_{\A_d}\otimes\nu_{\A_c}$, 
%$\nu_{\A}=\nu_{\A_d}\otimes|\cdot|_{k_c}$, 
where:
\begin{itemize}
\item $\nu_{\A_d}(\cdot)=\sum_{i\in\A_d}\Delta_i\delta_i(\cdot)$, $\Delta_i\geq0$, is some measure on the discrete space $\A_d\subset\R^{k_{\A_d}}$ (e.g.\ some subset of $\Z^{k_{\A_d}}$) such as the counting measure (obtained by setting $\Delta_i\equiv1$) 
%$\Delta_{\A_d}(\cdot)=\sum_{i\in\Z^{k_d}\cap\A_d}\delta_i(\cdot)$ 
if $\A_d$ is bounded (e.g.\ if $k_{\A_d}=1$ and $\A_d=\{1,\ldots,k_d\}$ for some bounded integer $k_d \geq2$, where $\nu_{\A_d}(\A_d)=k_d$). 
If $\A_d$ is an unbounded set, e.g.\ $\A_d=\Z^{k_{\A_d}}$, we could instead choose some suitable discrete probability measure, i.e.\ $\sum_{i\in\A_d}\Delta_i=1$.

\item The measure $\nu_{\A_c}$ governing $\A_c\subset\R^{k_c}$ is given by the Lebesgue measure (or its normalised version, the uniform measure $\nu_{\A_c}(\cdot)=|\cdot|/|\A_c|$) when $\A_c$ is bounded and some suitable probability measure (i.e.\ $\nu_{\A_c}(\A_c)=1$) if $\A_c$ is an unbounded set such as $\R^{k_c}$ or $[0,\infty)^{k_c}$.
\end{itemize}

\subsubsection{ %Examples of 
Auxiliary mark distributions }\label{auxiliaryMarksDistribution}

Depending on how we define the distributions of the auxiliary marks, the auxiliary mark space $\A$ and the auxiliary mark reference measure $\nu_{\A}$, 
% $\A$, different reference measures are used and, consequently, 
the measures $P_{x_1,\ldots,x_n}^{\A}(\cdot)$ in \eqref{AuxiliaryMarkDistributions} and thereby the product densities and the correlation functionals can take quite different forms. Continuing the discussions above and in Section \ref{SectionRefernceMeasures}, we next look closer at a few particular cases.

\begin{enumerate}
\item {\em Multi-type/multivariate FMPPs:} 
Recall that when 
%When $\A=\A_d=\{1,\ldots,k_{\A}\}$, 
each auxiliary mark has a discrete distribution on $\A=\{1,\ldots,k_d\}$, $k_d\geq2$, so that we may represent $\Psi$ by $(\Psi_1,\ldots,\Psi_{k_d})$, where $\Psi_i=\{(x,f):(x,j,f)\in\Psi, j=i\}\subset\X\times\F$ 
%(or $\Psi_i=
%\sum_{(x,l,f)\in\Psi\cap\X\times\{i\}\times\F}\delta_{(x,f)}=
%\sum_{(x,f)\in\Psi_i}\delta_{(x,f)}$) 
is the projection of $\Psi$ based on the auxiliary mark set $\{i\}$, $i=1,\ldots,k_d$, we call 
%In other words, 
$\Psi$ a multivariate/multi-type FMPP. Its ground process $\Psi_G$ may be represented by $(\Psi_G^1,\ldots,\Psi_G^{k_{d}})$, where $\Psi_G^i$ is the ground process of $\Psi_i$, and $\Psi_i$ has intensity functional $\rho_i(x,f)=Q_{x}^{i,\F}(f)\rho_G^i(x)$, where $\rho_G^i(\cdot)$ is the intensity function of $\Psi_G^i$ and $Q_{x}^{i,\F}(f)$ is the conditional density  governing the distribution of a functional mark of $\Psi_i$ on $\F$, which we interpret conditionally on $\Psi_i$ having a point at location $x\in\X$. 

Turning to the discrete finite auxiliary mark reference measure $\nu_{\A}$, we obtain 
$$
P_{x_1,\ldots,x_n}^{\A}(D_1\times\cdots\times D_n)
=\sum_{l_1\in D_1\cap\A}\cdots\sum_{l_n\in D_n\cap\A} Q_{x_1,\ldots,x_n}^{\A}(l_1,\ldots,l_n)
\nu_{\A}(l_1)\cdots\nu_{\A}(l_n),
$$
for $D_1,\ldots,D_n\in\B(\R)$, where $Q_{x_1,\ldots,x_n}^{\A}(\cdot)$ is the corresponding $n$-dimensional probability mass function. 
Further, 
the 1-dimensional Campbell formula now reads
\[
\E\left[
\sum_{(x,l,f)\in\Psi}
h(x,l,f)
\right]
=
\sum_{l\in\A}
\nu_{\A}(l) 
\int_{\X\times\F}
h(x,l,f)
\rho(x,l,f)\de x\nu_{\F}(df)
\]
and comparing it to the Campbell formula for $\Psi_i$, we obtain 
\begin{align*}
&\int_{\X\times\F}
h(x,f)
\rho_i(x,f)\de x\nu_{\F}(df)
=
\E\left[
\sum_{(x,f)\in\Psi_i}
h(x,f)
\right]
% \\
% &=
% \sum_{l\in\A}
% \int_{\X\times\F}
% \1\{l=i\}
% h(x,f)
% \rho(x,l,f)\nu_{\A}(l) \de x\nu_{\F}(df)
=
\nu_{\A}(i)
\int_{\X\times\F}
h(x,f)
\rho(x,i,f) \de x\nu_{\F}(df)
\end{align*}
for any measurable $h:\X\times\F\to[0,\infty)$, so in particular,
$$
Q_{x}^{i,\F}(f)\rho_G^i(x)=
\rho_i(x,f)=\rho(x,i,f)\nu_{\A}(i)
=
Q_{(x,i)}^{\F}(f) \rho_G(x) Q_{x}^{\A}(i) \nu_{\A}(i)
, \qquad i\in\A=\{1,\ldots,k_d\}.
$$
Recalling where we expressed the auxiliary references measure as $\nu_{\A}(\cdot)=\sum_{i\in\A}\Delta_i\delta_i(\cdot)=\sum_{i=1}^{k_{d}}\Delta_i\delta_i(\cdot)$, $\Delta_i\geq0$, above, when $\A$ contains a finite set of labels we see that by setting all $\Delta_i=1$, i.e.\ letting $\nu_{\A}$ be given by the counting measure on $\A$, we obtain that 
$$
Q_{x}^{i,\F}(f)\rho_G^i(x)
=\rho_i(x,f)
=\rho(x,i,f)
=Q_{(x,i)}^{\F}(f) \rho_G(x) Q_{x}^{\A}(i), 
$$
which often is the most natural choice. 
Hence, if we ignore the functional marks, i.e.\ we consider $\Psi_{\X\times\A}$, we obtain that
$$
\rho_G^i(x)
=\rho_i(x)
=\rho_{\X\times\A}(x,i)
=\rho_G(x) Q_{x}^{\A}(i),
$$
which we recognise from the common multi-type point process setting. 
To exemplify, note that if each auxiliary mark has a (marginal) multinomial distribution with parameter $\pi_i$, $i\in\A=\{1,\ldots,k_d\}$, and $\nu_{\A}$ is the counting measure on $\A$, then 
\[
Q_{x}^{i,\F}(f)\rho_G^i(x)
=\rho_i(x,f)=\rho(x,i,f)
=Q_{(x,i)}^{\F}(f) \rho_G(x)\pi_i,
\]
%In other words, the intensity of $\Psi_i$ is proportional to the intensity of $\Psi$. 
so if we ignore the functional marks, we obtain that 
\[
\rho_G^i(x)
%=\rho_i(x,f)=\rho(x,i,f)
= \rho_G(x)\pi_i,
\]
which is the intensity often considered in the multi-type point process setting.

\item  When each $L_i=(L_{1i},\ldots,L_{k_{\A}i})\in\A=\A_c\subset\R^{k_{\A_c}}$, $k_{\A}\geq1$, is a continuous random variable and $\A$ is bounded, the natural candidate for $\nu_{\A}$ would be the Lebesgue measure on $(\A,\B(\A))$. Recalling \eqref{AuxMarkDensities}, each $Q_{x_1,\ldots,x_n}^{\A}(l_1,\ldots,l_n)$ may be interpreted as a (conditional) probability density function on $\A^n$ in the classical sense. 
%conditioned on $\Psi_G$ having points at locations $x_1,\ldots,x_n$. 
When $\A$ is not bounded, since we have required that $\nu_{\A}$ must be finite, it would be natural to choose $\nu_{\A}$ as some probability measure. E.g., when $\A=\R^{k_{\A}}$, recalling that we interpret $P_{x_1,\ldots,x_n}^{\A}(\cdot)$ in \eqref{AuxiliaryMarkDistributions} as the conditional probability $\P((L_1,\ldots,L_n)\in\cdot|(X_1,\ldots,X_n)=(x_1,\ldots,x_n))$,
by letting $Z_1,\ldots,Z_n$ be iid random variables with distribution $\nu_{\A}$, 
%{\color{red} and representing $l_i$ by $Z_i$, $i=1,\ldots,n$}, 
we would obtain
\begin{align*}
P_{x_1,\ldots,x_n}^{\A}(C_1\times\cdots\times C_n)
&=\int_{C_1\times\cdots\times C_n}
Q_{x_1,\ldots,x_n}^{\A}(l_1,\ldots,l_n)
\nu_{\A}(dl_1)\cdots\nu_{\A}(dl_n)
\\
&=
\E[\1\{Z_1\in C_1,\ldots, Z_n\in C_n\}Q_{x_1,\ldots,x_n}^{\A}(Z_1,\ldots,Z_n)]
\end{align*}
for any $C_1,\ldots,C_n\in\B(\R^{k_{\A}})$. If further $\nu_{\A}$ has a density $f_Z(\cdot)$ with respect to the Lebesgue measure on $\R^{k_{\A}}$, we obtain that the density of $P_{x_1,\ldots,x_n}^{\A}(\cdot)$ with respect to the Lebesgue measure is given by $Q_{x_1,\ldots,x_n}^{\A}(l_1,\ldots,l_n)\prod_{i=1}^n f_{Z_i}(l_i)$. 
Hence, there is always a natural way of specifying the density of $P_{x_1,\ldots,x_n}^{\A}(\cdot)$; it is a product of two components, where one controls the dependence structure and the other is a classical multivariate density corresponding to iid random variables. 

\item  In the last scenario,  
$k_{\A}=k_{\A_d}+k_{\A_c}$, $k_{\A_d},k_{\A_c}\geq1$, and 
$$
L_i=(L_{1i},\ldots,L_{k_{\A_d}i},
L_{(k_{\A_d}+1)i},\ldots,L_{k_{\A}i}
)\in\A=\A_d\times\A_c\subset\R^{k_{\A_d}}\times\R^{k_{\A_c}}=\R^{k_{\A}},
$$
where $(L_{1i},\ldots,L_{k_{\A_d}i})$ is a discrete random vector on the discrete space $\A_d$ and $(L_{(k_{\A_d}+1)i},\ldots,L_{k_{\A}i})$ is a continuous random vector on $\A_c\subset\R^{k_{\A_d}}$. 
Here we simply let the reference measure be given by $\nu_{\A}(\cdot)=[\nu_{\A_d}\otimes\nu_{\A_c}](\cdot)$, the product measure of the two reference measures defined on the two spaces $\A_d$ (discrete) and $\A_c$ (continuous).

To exemplify, consider the case where $k_{\A_d}=k_{\A_c}=1$, so that each auxiliary mark has the form $L_i=(L_{i1},L_{i2})\in\A_d\times\A_c\subset\R\times\R$. E.g., $\A_d=\{1,\ldots,k_d\}$ and $\A_c=\R$, where the discrete random variable $L_{i1}\in\{1,\ldots,k_d\}$ may indicate which type the $i$th point belongs to, whereas $L_{i2}$ may serve the purpose of, say, controlling the functional mark(s).  %{\color{red}\sf we should refer to the  place that first time we mentioned this sentence}
For Borel sets $D_i=D_{i1}\times D_{i2} \subset \A_d\times\A_c = \A$, $i=1,\ldots,n$, we have
\begin{align*}
&P_{x_1,\ldots,x_n}^{\A}(D_1\times\cdots\times D_n)
=\\
=&
\int_{D_1\times\cdots\times D_n}
Q_{x_1,\ldots,x_n}^{\A}((l_{11},l_{12}),\ldots,(l_{n1},l_{n2}))
\nu_{\A_d}(\{l_{11}\})\nu_{\A_c}(d l_{12})
\cdots
\nu_{\A_d}(\{l_{n1}\})\nu_{\A_c}(d l_{n2})
\\
=&\sum_{(l_{11},\ldots,l_{n1})\in D_{11}\times\cdots\times D_{n1}}\prod_{i=1}^n\nu_{\A_d}(\{l_{i1}\})
\times\\
&\times
\int_{D_{12}\times\cdots\times D_{n2}}
Q_{x_1,\ldots,x_n}^{\A}((l_{11},l_{12}),\ldots,(l_{n1},l_{n2}))
%Q_{x_1,\ldots,x_n}^{\A}(l_{11},\ldots,l_{n1}, l_{12},\ldots,l_{n2})
\nu_{\A_c}(d l_{12})\cdots\nu_{\A_c}(d l_{n2}),
\end{align*}
i.e., a conditional mixed distribution function of the auxiliary marks, given that $(L_{i1},L_{i2}) = (l_{i1},l_{i2})$, $i=1,\ldots,n$. 
Note first that if $\nu_{\A_d}$ is the counting measure on $\A_d$, then the product in the expression above vanishes. 
Moreover, 
in many settings it may be natural to let one of the following hold:
\begin{itemize}
\item If all the discrete random variables $L_{11},\ldots,L_{N1}$ are independent of all the continuous random variables $L_{12},\ldots,L_{N2}$, then 
\[
P_{x_1,\ldots,x_n}^{\A}(D_1\times\cdots\times D_n)
=
P_{x_1,\ldots,x_n}^{\A_d}(D_{11}\times\cdots\times D_{n1})
P_{x_1,\ldots,x_n}^{\A_c}(D_{12}\times\cdots\times D_{n2}),
\]
where the first term on the right hand side has the form described in item 1.\ above and the second the form described in item 2.\ above.

\item Let $L_{i1}$ and $L_{i2}$ only depend on each other as well as the associated location $X_i$, but be independent of $\Psi_{\X\times\A}\setminus\{(X_i,(L_{i1},L_{i2}))\}$. Then,
\[
P_{x_1,\ldots,x_n}^{\A}(D_1\times\cdots\times D_n)
=
\prod_{i=1}^n 
P_{x_i}^{\A}(D_{i1}\times D_{i2})
% =
% \prod_{i=1}^n
% \sum_{l_{i1}\in D_{i1}}\prod_{i=1}^n\nu_{\A_d}(\{l_{i1}\})
% \int_{D_{12}\times\cdots\times D_{n2}}
% Q_{x_1,\ldots,x_n}^{\A}((l_{11},l_{12}),\ldots,(l_{n1},l_{n2}))
% %Q_{x_1,\ldots,x_n}^{\A}(l_{11},\ldots,l_{n1}, l_{12},\ldots,l_{n2})
% \nu_{\A_c}(d l_{12})\cdots\nu_{\A_c}(d l_{n2})
.
\]
% where the joint density in the integral satisfies $Q_{x_1,\ldots,x_n}^{\A}((l_{11},l_{12}),\ldots,(l_{n1},l_{n2}))=\prod_{i=1}^n Q_{x_i}^{\A}(l_{i1},l_{i2})$.

\item Combining the former two independence assumptions we obtain that $L_{11},\ldots,L_{N1},L_{12},\ldots,L_{N2}$ are all independent of each other but still location-dependent. Hence,
\begin{align*}
P_{x_1,\ldots,x_n}^{\A}(D_1\times\cdots\times D_n)
&=
\prod_{i=1}^n
P_{x_i}^{\A_d}(D_{i1})
P_{x_i}^{\A_c}(D_{i2})
\\
&=
\prod_{i=1}^n
\sum_{l_{i1}\in D_{i1}}
Q_{x_i}^{\A_d}(l_{i1})\nu_{\A_d}(\{l_{i1}\})
\int_{D_{i2}}
Q_{x_i}^{\A_c}(l_{i2})
\nu_{\A_c}(d l_{i2})
.
\end{align*}
%To concretize, 
Note that 
if all $L_{i1}$ are conditionally independent Bernoulli distributed random variables with parameter $p(X_i)\in[0,1]$, then 
$Q_{x_i}^{\A_d}(l_{i1})=p(X_i)\1\{l_{i1}=1\} + (1-p(X_i))\1\{l_{i1}=0\}$. In a forestry context, where e.g.\ $L_{i1}=1$ would mean that tree $i$ is a spruce and $L_{i1}=0$ that it is a pine, we are here saying that a tree has a higher probability of being a pine in certain areas but a spruce in other areas. Moreover, if all $L_{i2}$ are independent and exponentially distributed with location-dependent parameter $\mu(X_i)>0$, then $Q_{x_i}^{\A_c}(l_{i2})
\nu_{\A_c}(d l_{i2}) = \mu(x_i)\e^{-\mu(x_i)l_{i2}}\de l_{i2}$, so if we choose the reference measure to be a unit rate exponential distribution, i.e.\ 
$\nu_{\A_c}(d l_{i2})=\e^{-l_{i2}}\de l_{i2}$, 
then 
$Q_{x_i}^{\A_c}(l_{i2})=\mu(x_i)\e^{-l_{i2} \left(\mu(x_i)+1\right)}$.

\end{itemize}

\end{enumerate}

\subsection{Functional mark spaces}
\label{s:FunctionSpaces}

As mentioned in Section \ref{SectionRefernceMeasures}, we here briefly provide an overview of the two most natural Polish functions spaces, which we may employ as functional mark space components $\U$. We further also look at different functional mark distribution  properties. 

Considering a stochastic process $Y$, i.e.\ a measurable mapping
\begin{align}
\label{e:StochasticProcess}
Y:(\T\times\Omega, \B(\T)\otimes\Sigma)\to (\R, \B(\R))
,
\end{align}
we say that $Y$ is a random element in $\U$, or that $Y$ has sample paths in $\U$, 
if for each fixed $\omega\in \Omega$, the function $Y(\cdot,\omega)=\{Y(t, \omega), t\in\T\}$ with parameter $t\in\T$, known as a {\em sample path/realisation}, belongs to $\U$. As such, any sample path is a  
%  $$
% Y_t(\omega):\,\, \T\ni t\rightarrow Y(t,\omega)\in \R,\quad \omega\in\omega
% $$
measurable mapping from $(\T,\B(\T))$ to $(\R,\B(R))$ and for each fixed $t\in\T$, the mapping $\Omega\ni\omega\rightarrow Y(t,\omega)\in\R$ is a well-defined random variable on $(\Omega,\Sigma)$. The induced probability measure $P_Y(E)=\P(\{\omega\in\Omega:Y(\cdot,\omega)\in E\})$, $E\in\B(\U)$, is called the distribution of $Y$ and it is governed by the finite dimensional distributions $\P(Y(t_1)\in B_1,\ldots,Y(t_n)\in B_n)$, $n\geq1$, $t_1,\ldots,t_n\in\T$, $B_1,\ldots,B_n\in\B(\R)$, by Kolmogorov's consistency theorem. 
Moreover, since $\U$ is assumed to be a Polish (topological) space, there exists a metric $d_{\U}(f,g)$, $f,g\in\U$, which turns $\U$ into a complete separable metric space.

\subsubsection{Skorohod and $L_p$ spaces}
%, which most naturally is given by $\X$, i.e.\ $d=k\geq2$ and $f(t)\in\X$ for any $t\in\T$. 
%Assuming that $T^*=\sup \T$ (with $T^*=\infty$ if $\T=[0,\infty)$), 
Consider first the case where $\U$ is given by %a Skorohod space  
% More specifically, 
% consider the function space
\beann
D_{\T}(\R) = \{f:\T\rightarrow\R | f  \text{ is c\`adl\`ag}\},
\eeann
which is the set of {\em c\`adl\`ag}, i.e.\ right continuous with left limits, functions $f:\T\rightarrow\R$ 
\citep{Billingsley,EthierKurtz,JacodShiryaev,Silvestrov}. 
Consider now the collection $\Lambda$ of all strictly increasing, surjective and Lipschitz continuous functions $\lambda:\T\rightarrow\T$, $\lambda(0)=0$, $\lim_{t\rightarrow\infty}\lambda(t)=T^*=\sup \T$ (with $T^*=\infty$ if $\T=[0,\infty)$), such that
$$
u(\lambda) = \sup_{s,t\in\T : t<s}\left|\log\frac{\lambda(s)-\lambda(t)}{s-t}\right| < \infty.
$$
Endowing $\U=D_{\T}(\R)$ with the metric
\beann
d_{\U}(f,g)=d_{D_{\T}(\R)}(f,g) = \inf_{\lambda\in\Lambda}
\left\{
u(\lambda) \vee
\int_{\T} \e^{-u} \sup_{t\in\T}\{
%d_{\R}(
|f(t\wedge u)
%,
-
g(\lambda(t)\wedge u)
|
%)
\wedge1
\}
\de u
\right\},
\eeann
%where $d_{\R}(x,y)=|x-y|$, 
which turns it into a complete and separable metric space \Citep{EthierKurtz}, the corresponding topology is called a {\em Skorohod topology} and $D_{\T}(\R)$ is called a {\em Skorohod space}. 
%The Borel sets generated by the corresponding topology will be denoted by $\B(\U)$ and it follows that $\B(\U^n)=\B(\U)^n$, $n\geq1$ \Citep{JacodShiryaev}. 
%
We note that functions in $D_{\T}(\R)$ include e.g.\ sample paths of Markov processes, L\'{e}vy processes and semi-martingales, as well as empirical distribution functions. We further note that the classical Wiener space, i.e., the space $C_{\T}(\R) = \{f:\T\rightarrow\R : f \text{ continuous}\}$ is a subspace of $D_{\T}(\R)$ and for these functions $d_{D_{\T}(\R)}$ reduces to the uniform metric $d_{\infty}(f,g)=\sup_{t\in\T}|f(t)-g(t)|$. 
In addition, the Borel $\sigma$-algebra $\B(C_{\T}(\R))$ generated by $d_{\infty}(\cdot,\cdot)$ on $C_{\T}(\R)$ satisfies $\B(C_{\T}(\R))=\{E\cap C_{\T}(\R):E\in\B(D_{\T}(\R))\}\subset\B(D_{\T}(\R))$ \Citep[Chapter VI]{JacodShiryaev}. 
Hence, we can accommodate e.g.\ diffusion processes or some other class of processes with continuous sample paths (note also that each space $C_{\T}^k(\R)$, $k\in\N$, of $k$ times continuously differentiable functions is a subspace of $C_{\T}(\R)$). 

Consider now the following definition, given in accordance with \Citep[1.6.1]{Silvestrov}.
\begin{definition}\label{DefCadlagProcess}
A stochastic process $Y(t)=(Y_1(t),\ldots,Y_k(t))$, $k\geq1$, $t\in\T$, is called a $k$-dimensional \emph{c\`adl\`ag stochastic process} if each of its sample paths $Y(\omega)=\{Y(t;\omega)\}_{t\in\T}$, $\omega\in\Omega$, is an element of $\F=\U^k$.
\end{definition}
In light of this definition, since $\U$ is given by the Skorohod space $D_{\T}(\R)$, the functional marks $F_i(t)=(F_{i1}(t),\ldots,F_{ik}(t))\in\R^k$, $t\in\T\subset[0,\infty)$, $i=1,\ldots,N$, will be a collection of (possibly dependent) $k$-dimensional c\`adl\`ag stochastic processes. 
For details on filtrations with respect to c\`adl\`ag stochastic processes, see \Citep[Chapter VI]{JacodShiryaev}.

%\subsubsection{$L_p$ spaces}
Next, consider the case where $\U$ is given by the class of measurable functions %$f:\T\rightarrow \R$ satisfying  
%Each function is assumed to belong to a 
\begin{align}
\label{Lp}
L_p=L_p(\T,\B(\T),|\cdot|)=\left\{f:\T\rightarrow \R \left| \|f\|_{p}=\left(\int_{\T} |f(t)|^p \de t\right)^{1/p} < \infty\right\}\right.
,\qquad
1\leq p<\infty.    
\end{align}
The metric on $\U=L_p$ is given by $d_{\U}(f,g)=d_{L_p}(f,g)=\|f-g\|_{p}$. 
Since $(\T,\B(\T),|\cdot|)$ is $\sigma$-finite and countably generated, it follows that $L_p$ is a complete and separable metric space whenever $1\leq p<\infty$ \citep[p.\ 243]{Billingsley1995PM}. Hence, given the $d_{L_p}$-induced topology, $L_p$ is an example of a Polish space. Moreover, we have that $L_p$ is a Banach space and in the particular case where $p=2$, which constitutes all square integrable functions, $\U$ is additionally a Hilbert space with inner product $\langle f,g \rangle=\int_{\T} f(t)g(t)\de t$.
%
% {\color{red}
% A natural question that arises here is when a functional mark is given by a $p$th order process, i.e.\ is a $k$-dimensional stochastic process $Y(t)=(Y_1(t),\ldots,Y_k(t))$, $k\geq1$, $t\in\T$, such that it has finite  $p$th order absolute moments. 
%
% Note that a stochastic process $Y(t)=(Y_1(t),\ldots,Y_k(t))$, $k\geq1$, $t\in\T$, is called a $k$-dimensional \emph{second order process} if $\mathbb{E}|Y_i(t)|^2=\int_{\Omega}|Y_i(t,\omega)|^2\P(d\omega)<\infty$, $i=1,\ldots,k$, i.e. if each of its sample paths $Y(\omega)=\{Y(t;\omega)\}_{t\in\T}$, $\omega\in\Omega$, is an element of $\F=\U^k$, where $\U=L_2$. As an example, $Y_i(t)$, $i=1,\ldots,k$ can be considered as Gaussian white noise with finite mean and variance. 
%
% If we consider for $\U$ a subclass of $L_p$, for which there is a uniformly dominating function $g\in L_p$, then our functional marks will be $p$th order processes.
% }
% Consider a stochastic process $Y$, i.e.\ a measurable mapping from $(\T\times\Omega, \B(\T)\otimes\Sigma)$ to $(\R, \B(\R))$. It follows that for each fixed $\omega\in \Omega$ the function $Y(\cdot,\omega)=\{Y(t, \omega), t\in\T\}$ with parameter $t\in\T$, known as a {\em sample path},
% %  $$
% % Y_t(\omega):\,\, \T\ni t\rightarrow Y(t,\omega)\in \R,\quad \omega\in\omega
% % $$
% is a measurable mapping from $(\T,\B(\T))$ to $(\R,\B(R))$ and 
%. For a measurable stochastic process Y, 
Recalling the stochastic process $Y$ in \eqref{e:StochasticProcess}, note that we here assume that 
$$
\int_{\T}|Y(t,\omega)|^p\de t<\infty \text{ for all }\omega\in\Omega,
$$
i.e., $Y$ is a (measurable) stochastic process with sample paths in $L_p$. 
There is further a connection between $L_p(\T,\B(\T),|\cdot|)$ and the setting where each $Y(t)$, $t\in\T$, belongs to the space $L_p(\Omega,\Sigma,\P)$ of random variables with finite $p$th moment, i.e., $(\E[|Y(t)|^p])^{1/p}<\infty$. If $Y$ has sample paths in $L_p(\T,\B(\T),|\cdot|)$, then $(\E[|Y(t)|^p])^{1/p}<\infty$, $t\in\T$, since
\[
\E[|Y(t)|^p]
=\int_{\Omega}|Y(t,\omega)|^p \P(d\omega)
\leq 
\int_{\Omega}\int_{\T}|Y(t,\omega)|^p \de t\P(d\omega) <\infty.
\]
In other words, by assuming that our functional mark space $\U$ is given by $L_p(\T,\B(\T),|\cdot|)$, we automatically have that each functional mark $F_i(\omega)=\{F_i(t,\omega)\}_{t\in\T}$, $\omega\in\Omega$, has finite $p$th moment for any $t\in\T$, i.e., $F_i(t)\in L_p(\Omega,\Sigma,\P)$ for any $t\in\T$.
Reversely, if $Y(t)\in L_p(\Omega,\Sigma,\P)$ for any $t\in\T$, i.e., $\E[|Y(t)|^p]=\int_{\Omega}|Y(t,\omega)|^p \P(d\omega)<\infty$, $t\in\T$, it follows that $Y$ a.s.\ has sample paths in $L_p(\T,\B(\T),|\cdot|)$ whenever $\T$ is bounded. When $\T$ is unbounded, by requiring that there is an integrable function $g\in L_p$ such that $|Y(t,\omega)|^p\leq g(t)$, $t\in\T$, for each $\omega\in\Omega$, we have that $\E[|Y(t)|^p]\leq g(t)$ and 
$
\int_{\T}|Y(t,\omega)|^p\de t
\leq
\int_{\T}g(t)\de t <\infty. 
$ 
In other words, $Y$ has sample paths in $L_p(\T,\B(\T),|\cdot|)$ and $Y(t)\in L_p(\Omega,\Sigma,\P)$ for any $t\in\T$, so a functional mark here belongs to both of theses $L_p$-space.

\subsubsection{Functional mark distributions and their finite-dimensional distributions}
% By assigning a specific structure to each $P_{(x_1,l_1),\ldots,(x_n,l_n)}^{\F}(\cdot)$, or equivalently each $Q_{(x_1,l_1),\ldots,(x_n,l_n)}^{\F}(\cdot)$, we can determine what kind of functional marks are obtained. 

We next look closer at different structures for the distributions $P_{(x_1,l_1),\ldots,(x_n,l_n)}^{\F}(\cdot)$ in \eqref{FunctionalMarkDistributions}, or equivalently the densities $Q_{(x_1,l_1),\ldots,(x_n,l_n)}^{\F}(\cdot)$ in \eqref{FunMarkDensities}. 
Recall, in particular, the random functional 
$
\Psi|\Psi_{\X\times\A}=\{F_1|\Psi_{\X\times\A},\ldots,F_N|\Psi_{\X\times\A}\}
=\{F_1(t)|\Psi_{\X\times\A},\ldots,F_N(t)|\Psi_{\X\times\A}\}_{t\in\T}\subset\F 
$ 
from Section \ref{SectionComponents},
%{\color{red}$\Psi_{\F}=\{F_i; (X_i,L_i,F_i)\in\Psi\}$} 
which we view as a stochastic process with dimension $N$ for which all the marginal distributions are the same. 
Note that $P_{(x_1,l_1),\ldots,(x_n,l_n)}^{\F}(\cdot)$ governs the distribution of $n$ components of $\Psi|\Psi_{\X\times\A}$.

%\subsubsection{Finite-dimensional distributions of the functional marks}

Being a distribution on the function space $(\F^n,\B(\F^n))$, each $P_{(x_1,l_1),\ldots,(x_n,l_n)}^{\F}(\cdot)$ is an abstract and non-tractable object, 
despite the fact that we may sometimes be able to explicitly define its density $Q_{(x_1,l_1),\ldots,(x_n,l_n)}^{\F}(\cdot)$ with respect to some reference measure $\nu_{\F}^n$ (recall expression \eqref{e:FunctionalReferenceMeasure}). Below, we provide different examples of how such functional mark distributions may be specified e.g.\ through the choice of functional reference measure $\nu_{\F}$. 
Since $\Psi|\Psi_{\X\times\A}$ may be treated as a continuous-time stochastic process, for all practical and mathematically explicit purposes, we turn to the \emph{finite-dimensional distributions} of the functional marks. 
For an informative discussion on finite-dimensional distributions for c\`adl\`ag processes, see \citet[Section 1.6.2]{Silvestrov}. 

Conditionally on $\Psi_{\X\times\A}$, 
assume that we have $\{(X_i,L_i)\}_{i\in I} = \{(x_i,l_i)\}_{i\in I}$, $I=\{1,\ldots,n\}\subset\{1,\ldots,N\}$, 
denote the cardinality of $I$ by $|I|=n$, 
and 
consider 
\bea
\label{ConditionalMarkProcess}
\Psi|\Psi_{\X\times\A}
\supset
F_I=\{F_{I}(t)\}_{t\in\T} = \{(F_{1}(t),\ldots,F_{n}(t))
|\{(X_{j},L_{j})=(x_j,l_j)\}_{j=1}^{n}\}_{t\in\T}\in\F^n=(\U^k)^n,
\eea
where we note that $F_{I}(t)\in(\R^k)^n$ for any $t\in\T$. 
It follows that $P_{(x_1,l_1),\ldots,(x_n,l_n)}^{\F}(\cdot)$, which is the distribution of $F_I$ on $(\F^n,\B(\F^n))=((\U^k)^n,\B((\U^k)^n))$, is uniquely determined by the finite-dimensional distributions of $F_I$ \Citep[Lemma 1.6.1.]{Silvestrov}:
\begin{align*}
&P_{F_I}
=
\{
P_{F_I(S_l)}(A_1\times\cdots\times A_l) 
: l\geq1, S_l=\{s_1,\ldots,s_l\}\subset\T, A_1,\ldots, A_l\in\B(\R^{k})^n\}
,
\\
&P_{F_I(S_l)}(\cdot)
=
\P((F_{I}(s_1),\ldots,F_{I}(s_l))\in \cdot) 
%\{
%\P((F_{I}(s_1),\ldots,F_{I}(s_l))\in A_1\times\cdots\times A_l) 
%: k\geq1, s_1,\ldots,s_k\in\T, A\in\B(\R^{n\times k})\}
.
\end{align*}
%Here we may also choose the sets $A$ as products $(-\infty,u_{ij1}]\times\cdots\times(-\infty,u_{ijk}]\subset\R$, $i=1,\ldots,n$, $j=1,\ldots,l$, of half-open intervals where each $(u_{1j},\ldots,u_{nj})$, $j=1,\ldots,k$, is a continuity point of the distribution function corresponding to $P_{F_I(S_k)}(\cdot)$. 
%Although we here have considered the STCFMPP case, the CFMPP case is analogous.  
Conditionally on $\Psi_{\X\times\A}$, 
it follows that $\{F_{i}\}_{i=1}^{N}$, i.e.\ $\Psi|\Psi_{\X\times\A}$, is completely determined by the collection $\{P_{F_I}\}_{I\in\mathcal{P}_N}$, where $\mathcal{P}_N$ denotes the power set of $\{1,\ldots,N\}$; recall that conditioning on $\Psi_{\X\times\A}$ implies conditioning on $N$. 
If, in addition, $P_{F_I(S_l)}$ is absolutely continuous with respect to the corresponding Lebesgue measure,
%$l$-fold product of $\ell_{k}^{n}=\ell_{1}^{n k}$, the density of $P_{F_I(S_l)}$ will be denoted by 
\bea
\label{FidiDensity}
P_{F_I(S_l)}(A_1\times\cdots\times A_l)
&=&
\int_{A_1}\cdots\int_{A_l}
Q_{(x_1,l_1),\ldots,(x_n,l_n)}^{s_1,\ldots,s_l}(u_1,\ldots,u_l) 
\de u_1\cdots \de u_l
\eea
for some probability density $Q_{(x_1,l_1),\ldots,(x_n,l_n)}^{s_1,\ldots,s_l}$ on $((\R^k)^n)^l$, where $s_j\in\T$, $j=1,\ldots,l$, correspond to the evaluation time points and
\beann
u_j
=
\begin{pmatrix}
u_{j1}\\
\vdots\\
u_{jn}
\end{pmatrix}
=
%f_{(x_1,l_1),\ldots,(x_n,l_n)}^{s_1,\ldots,s_l}
\begin{pmatrix}
u_{j11} & \cdots & u_{j1k} \\
\vdots & \ddots & \vdots \\
u_{jn1} & \cdots & u_{jnk}
\end{pmatrix}
\in\R^{n\times k}
,
\quad
j=1,\ldots,l
.
\eeann
Here row $i\in\{1,\ldots,n\}$ corresponds to the sampling at times $s_1,\ldots,s_l$ of an element $F_i|\Psi_{\X\times\A}=(F_{i1}|\Psi_{\X\times\A},\ldots,F_{ik}|\Psi_{\X\times\A})\in\U^k$ of $F_I$. 
This is a more natural and feasible way to specify a specific model structure for the functional marks, compared to specifying the functional densities directly. To exemplify, 
assume that $F_i(t)\in\R$, i.e.\ $k=1$, and that we are considering the joint  distribution of two functional marks $F_1$ and $F_2$ conditionally on $\Psi_{\X\times\A}$. Then this reduces to
$$
Q_{(x_1,l_1),(x_2,l_2)}^{s_1,\ldots,s_l}
(u_1,\ldots,u_l)
,\quad
u_j=
\begin{pmatrix}
u_{j11} \\
u_{j21}
\end{pmatrix}\equiv
\begin{pmatrix}
u_{j1} \\
u_{j2}
\end{pmatrix}
\in\R^2
,
\quad
j=1,\ldots,l.
$$ 

Considering an FMPP for which the marks have not been sampled in their entirety, but rather at the sample times $s_1,\ldots,s_l\in\T$, we see that the densities $Q_{(x_1,l_1),\ldots,(x_n,l_n)}^{s_1,\ldots,s_l}(\cdot)$, $n\geq1$, constitute the part the likelihood function that corresponds to the functional marks.

% {\color{red} As an example, we can consider the price of a liter gasoline at two different locations, such as Tehran and Stockholm, which has been recorded for successive years between 1960 and 2018. Please note that, this is only for our own intution.}

Recall the underlying filtered probability space $(\Omega,\Sigma,\Sigma_{\T},\P)$ mentioned in Section \ref{SectionComponents} and assume that $\Psi|\Psi_{\X\times\A}$ is adapted to it, i.e., $\Sigma_{\T}=\{\Sigma_{t}\}_{t\in\T}$ is an increasing family of $\sigma$-algebras such that $F_i(t)|\Psi_{\X\times\A}$ is $\Sigma_{t}$-measurable for any $t\in\T$ and any $i=1,\ldots,N$. 
Recalling $F_I\subset\Psi|\Psi_{\X\times\A}$ from expression \eqref{ConditionalMarkProcess}, 
%and denoting the cardinality of $I$ by $|I|$, 
one way of having a natural filtration/history in this context would be to consider $\Sigma_{t}^{F_I}=\sigma\{F_{I}(s)^{-1}(A) : s\in\T\cap[0,t], A\in\B(\R)^{|I|}\}$, i.e.\ the $\sigma$-algebra generated by $F_I$ over $[0,t]$, $t\in\T$, and 
to assume that the underlying filtered probability space satisfies $\Sigma_{t}^{F_I}\subset\Sigma_{t}$ for any element $I$ in the power set $\mathcal{P}_N$. %of $\{1,\ldots,N\}$.

% {\color{red}
% The functional conditional densities (\ref{FunMarkDensities}) here take the form
% \bea
% \label{MultivariateAuxiliaryDensity}
% Q_{(x_1,l_{12}),\ldots,(x_n,l_{n2})}^{\F}(\cdot; l_{11},\ldots,l_{n1}),
% \quad
% x_1,\ldots,x_n\in\X.
% \eea
% Hence, given the spatial locations $x_1,\ldots,x_n\in\X$ and continuous random variables $l_{12},\ldots,l_{n2}\in\A_c$, the functional mark distributions may still vary, depending on which class each point $i=1,\ldots,n$ is assigned. 
% Note that the $L_{i2}$'s may be treated as random parameter vectors which control e.g.\ the supports of the functional marks. 

% \begin{rem}
% Since each $L_{i2}$ is a random variable, which may play the role of a parameter in a stochastic process, the current setup for FMPPs connects naturally to a Bayesian stochastic process framework.
% \end{rem}
% }

\subsubsection{Random functional mark supports}
%So far we have not discussed specific choices for either of these components. 
We have previously mentioned that one of the main purposes of the auxiliary marks is to control the functional marks. One such setting is the case when the support of $F_i$ is such that $\supp(F_i)=S_i=S_i(X_i,L_i)\subset\T$, $i=1,\ldots,N$, i.e.\  conditionally on $\Psi_{\X\times\A}$, the support depends on $X_i$ and $L_i$. %where $(X_i,L_i)=(x_i,l_i)$, $i=1,\ldots,n$, 
Fixing $(X_i,L_i)=(x_i,l_i)$, $i=1,\ldots,n$, it then follows that $Q_{(x_1,l_1),\ldots,(x_n,l_n)}^{\F}(f_1,\ldots,f_n)=0$ if, for any $i=1,\ldots,n$, $f_i\in\F\setminus\{f\in\F : \supp(f)=S_i\}$. 
%Note that often a natural choice for the supports is $S_i=[T_i,(T_i+L_i)\wedge T^*)$, $L_i\geq0$, $i=1,\ldots,N$. 

\subsubsection{Deterministic functional marks}
As in the case of constructed marks (e.g.\ LISA functions) or in the case of the classical growth-interaction process, conditionally on 
%the locations (and the auxiliary marks)
$\Psi_{\X\times\A}$ we may want to consider deterministic functional marks. 

Given some deterministic function $f^*(x,l,t)\in\R^k$, $(x,l,t)\in\X\times\A\times\T$, such that, for any fixed $(x,l)\in\X\times\A$, the function $f_{(x,l)}^*=\{f_{(x,l)}^*(t)=f^*(x,l,t):t\in\T\}$ belongs to $\F=\U^k$, assume that we want to construct our functional marks in such a way that $F_i=f_{(x_i,l_i)}^*$ conditionally on $(X_i,L_i)=(x_i,l_i)$. 
To this end, for any $n\geq1$ and $E_1,\ldots,E_n\in \B(\F)$, let 
\begin{align*}
P_{(x_1,l_1),\ldots,(x_n,l_n)}^{\F}(E_1\times\cdots\times E_n) 
&= 
\prod_{i=1}^{n}P_{(x_i,l_i)}^{\F}(E_i)
=\prod_{i=1}^{n} \delta_{f_{(x_i,l_i)}^*}(E_i)
= \prod_{i=1}^{n} \1\{f_{(x_i,l_i)}^*\in E_i\}
%\int_{E_1}\cdots\int_{E_n}\prod_{i=1}^{n}P_{(x_i,l_i)}^{\F}(df_i)
%\\
%&= \int_{E_1}\delta_{f_{(x_1,l_1)}^*}(df_1)\cdots\int_{E_n}\delta_{f_{(x_n,l_n)}^*}(df_n)
%\\
%&=\1\{f_{(x_1,l_1)}^*\in E_1,\ldots, f_{(x_n,l_n)}^*\in E_n\}
,
\end{align*}
where we recall that $\delta_{f_{(x,l)}^*}(\cdot)$ denotes the point mass (Dirac measure) of the function $f_{(x,l)}^*$. 
Hence, $P_{(x_1,l_1),\ldots,(x_n,l_n)}^{\F}(E_1\times\cdots\times E_n) =1$ if for each $i=1,\ldots,n$ we have $F_i=f_{(x_i,l_i)}^*$.

\subsubsection{Wiener measure generated densities}
Assuming that the functional reference measure $\nu_{\F}(\cdot)$ in expression \eqref{e:FunctionalReferenceMeasure} is given by the Wiener measure $W_{\F}(\cdot)$ on $(\F,\B(\F))$, 
we may next ask ourselves the adequate question how one could obtain explicit forms for the densities $Q_{(x_1,l_1),\ldots,(x_n,l_n)}^{\F}(\cdot)$. To give an indication of what this really means, assume that conditionally on $\Psi_{\X\times\A}$,
%and the auxiliary marks, 
we want $(F_1(t),\ldots,F_n(t))$ to be given by, say, an $n$-dimensional diffusion process $(Y_1(t),\ldots,Y_n(t))$, $t\in\T$. 
Recalling Section \ref{SectionRefernceMeasures}, under certain conditions the use of the Cameron-Martin-Girsanov theorem (see e.g.\ 
\citet{Skorohod,rajput1972gaussian,JacodShiryaev,maniglia2004gaussian,Klebaner,kallenberg2006foundations,MortersPeresBMbook} and the references therein) gives rise to explicit expressions for $Q_{(x_1,l_1),\ldots,(x_n,l_n)}^{\F}(\cdot)$. 
Furthermore, changing the support of each $F_i$ to some interval $S_i\subset\T$ can be obtained by multiplying the density by the point mass $\delta_{\Gamma_i}(f)$, where 
$\Gamma_i$ is the collection of all functions with support given by $S_i$, $i=1,\ldots,n$, and/or by applying time-change/stopping results to $(Y_1(t),\ldots,Y_n(t))$ before applying the Cameron-Martin-Girsanov theorem.
We note that such a setup would be the underlying construction for the extensions discussed in Section \ref{SectionGI}. 
We stress that most of the ideas indicated may very well be applied to, say, 
% Poisson random measures (see e.g.\ \Citep{Klebaner,JacodShiryaev}) or some other 
L\'{e}vy process/semi-martingale generated random measures on $(\F,\B(\F))$ (see e.g.\ \citet{JacodShiryaev,Skorohod}). %. 
Note e.g.\ that in the Poisson process functional mark case one would be able to generate multivariate functional marks given by multivariate Poisson processes, a construction similar to the one in \citet{AppleTrees}. 
%This could likely be the required setup for Examples 4 and 5 in Section \ref{SectionApplications}. 

\subsubsection{Markovian functional marks}\label{SectionMarkovMarks}

In many cases it may be of interest to let the functional marks be given by Markov processes. This is e.g.\ the case when considering the stochastic growth-interaction process or, more generally, when each mark is given by some diffusion process.

%Recall that, conditionally on $\Psi$ and the auxiliary marks, 
%\bea
%\label{ConditionalMarkProcess}
%F_I=\{F_{I}(t)\}_{t\in\T} = \{(F_i(t))_{i\in I}\}_{t\in\T}
%\eea
%for any index set $I\in\mathcal{P}_N$. 

We say that $\Psi$ has Markovian functional marks if each component of $\Psi|\Psi_{\X\times\A}$ is a Markov process, which is to say that each $F_{I}$, $I\in\mathcal{P}_N$, constitutes a Markov process: for $s,t\in\T$, $s\leq t$,
\beann
\P\left(F_{I}(t)\in A|\Sigma_{s}\right)
=
\P\left(F_{I}(t)\in A|F_{I}(s)\right)
=
P_{t,s}^{F_{I}}(A;F_{I}(s))
,
\quad A\in\B(\R)^{|I|},
\eeann 
where the right hand side is the 
%Here we refer to 
%$
%P_{t,s}^{F_{I}}(A;F_{I}(s))
%=
%\P\left(F_{I}(t)\in A|F_{I}(s)\right)
%$ 
%as the 
$F_{I}$-transition probability. 
When there exist transition densities $p_{t,s}^{F_{I}}(u_t;u_s)$, $u_t,u_s\in\R^{n\times k}$, with respect to the corresponding Lebesgue measure, i.e.\ $P_{t,s}^{F_{I}}(A;u_s)=\int_Ap_{t,s}^{F_{I}}(u_t;u_s)\de u_t$, we find that the density in (\ref{FidiDensity}) reduces to 
\bea
\label{TransDens}
Q_{(x_1,l_1),\ldots,(x_n,l_n)}^{s_1}
(u_1) 
\prod_{i=2}^{l} p_{s_i,s_{i-1}}^{F_{I}}(u_i;u_{i-1})
,\quad
u_1,\ldots,u_l\in\R^{n\times k}. 
\eea

\section{Specific classes of FMPPs}\label{SectionClassesSTCFMPP}

Having defined a general structure for FMPPs, we here turn to different model constructions. 
%Since by the notion of a spatio-temporal point process is often meant a temporally grounded point process, we look closer at temporally grounded (ST)CFMPPs and as a result obtain a definition of functional marked conditional intensities \Citep{CoxIsham1980,DVJ1,VereJones}. 
%Furthermore, we look closer at finite (ST)CFMPPs, and then in particular Markov (ST)CFMPPs \Citep{DVJ1, VanLieshout}.
%However, we start by defining (spatio-temporal) c\`adl\`ag functional marked Poisson and Cox processes \Citep{DVJ1, VanLieshout, Moller, SKM}. 

\subsection{Functional marked Cox processes}
We here consider Cox processes (see e.g.\ \Citep[p.\ 154]{SKM}) in the current context of %c\`adl\`ag 
functional marking. 
These are common and interesting models for spatial clustering. 
%Recall from (\ref{GroundIntensityMeasure}) that $\mu_G$ is the ground intensity measure of a (ST)CFMPP. 

\begin{definition}
Given a locally finite random measure $\Lambda_G$ on $\X$, 
a (spatio-temporal) FMPP $\Psi$ is called a \emph{(spatio-temporal) %c\`adl\`ag 
functional marked %((ST)CFM) 
Cox process} (directed by $\Lambda_G$) if the ground process $\Psi_G$ constitutes a $\Lambda_G$-directed Cox process on $\X$. 
In other words, conditionally on $\Lambda_G$, $\Psi_G$ is a Poisson process with intensity measure $\mu_G=\Lambda_G$. 
\end{definition} 

%We note that, conditionally on $\Lambda_G$, $\Psi$ becomes a (ST)CFMPP ground Poisson process. Hence, a more suitable name for $\Psi$ would possibly be \emph{(ST)CFM ground Cox process}. 
Assume next that the %locally finite 
random measure $\Lambda_G(C)=\int_{C}\Lambda(x)\de x$, $C\in\B(\X)$, is generated by an a.s.\ non-negative random field $\Lambda=\{\Lambda(x)\}_{x\in\X}$, which consequently must be a.s.\ locally integrable. Note that in the spatio-temporal case it is natural to write $\Lambda(x,t)$ to emphasize that the random field has a time component.
%When $\G=\X$ (CFM Cox process), we may write $\Lambda=\{\Lambda(x)\}_{x\in\X}$ and when $\G=\X\times\T$ (STCFM Cox process) we may write $\Lambda=\{\Lambda(x,t)\}_{(x,t)\in\X\times\T}$. 
It now follows that 
% For a (ST)CFM Cox process, in light of Proposition \ref{PropositionProdDens}, 
the $n$th product density is given by \Citep[Chapter 6.2.]{DVJ1}
$$
\rho^{(n)}((x_1,l_1,f_1),\ldots,(x_n,l_n,f_n))
=
Q_{(x_1,l_1),\ldots,(x_n,l_n)}^{\F}(f_1,\ldots,f_n)
Q_{x_1,\ldots,x_n}^{\A}(l_1,\ldots,l_n)
\E\left[\prod_{i=1}^{n}\Lambda_{G}(x_i)\right].
$$
% whereby $g_G(\cdot)\equiv1$ and its pair correlation functional becomes
% \begin{align*}
% g_{\Psi}((g_1,l_1,f_1),(g_2,l_2,f_2)) 
% = \frac{f_{(g_1,l_1),(g_2,l_2)}^{\F}(f_1,f_2)f_{g_1,g_2}^{\A}(l_1,l_2)}
% {f_{(g_1,l_1)}^{\F}(f_1) f_{g_1}^{\A}(l_1)f_{(g_2,l_2)}^{\F}(f_2) f_{g_2}^{\A}(l_2)}
% .
% \end{align*}

When $\Psi$ is a spatio-temporal functional marked Cox process with spatio-temporal geostatistical marking (recall Definition \ref{DefGeostatMarking}), i.e.\ $F_i(t)=Z_{X_i}(t)$ for some spatio-temporal random field $Z=\{Z_{x}(t)\}_{(x,t)\in\X\times\T}$, we may connect random fields and point processes simultaneously in two different ways; the driving random field $\Lambda$ {\em from underneath} 
and a random field $Z$ {\em from above}. 
This structure is simplified when we consider intensity dependent marks (Section \ref{SectionIntensityDependentMarks}). In the current context this translates into the following definition.
\begin{definition}
A spatio-temporal functional marked Cox process $\Psi$ with random intensity field $\Lambda=\{\Lambda(x,t)\}_{(x,t)\in\X\times\T}$ is said to have \emph{intensity-dependent marks} if, conditionally on $\Psi_G$ and the random field $\Lambda$, the functional marks are given by $F_i(t)=\Lambda(X_i,t)$, $t\in\T$, $i=1,\ldots,N$.
\end{definition}

\subsection{Functional marked Gibbs processes}

We next consider another important class of point processes, in the context of functional marking, namely so-called {\em functional marked Gibbs processes}. These are simply marked Gibbs processes \citep{SKM,Moller,VanLieshout} for which the mark space is given by $\A\times\F$. 

There are various ways to define (marked) Gibbs processes \citep[Section 6]{Moller} and we here consider the statistically most convenient approach, which is through {\em Papangelou conditional intensities}. They are defined through the Georgii-Nguyen-Zessin formula \citep{SKM,Moller,VanLieshout}, which states that for any measurable mapping $h:\X\times\A\times\F\times  N_{lf}\rightarrow[0,\infty)$, 
\begin{align} 
\label{NguyenZessin}
\E\left[\sum_{(x,l,f)\in\Psi} h(x,l,f,\Psi\setminus\{(x,l,f)\})\right]
&=
\int_{N_{lf}}
\int_{\X\times\A\times\F}
h(x,l,f,\psi)\Lambda(d(x,l,f);\psi)
P(d\psi)
\nonumber
\\
&=
\int_{N_{lf}}
\int_{\X\times\A\times\F}
h(x,l,f,\psi)\lambda(x,l,f;\psi)
\de x\nu_{\A}(dl)\nu_{\F}(df)
P(d\psi)
\nonumber
\\
&=
\int_{\X\times\A\times\F}
\E\left[h(x,l,f,\Psi)\lambda(x,l,f;\Psi)\right]
\de x\nu_{\A}(dl)\nu_{\F}(df)
.
\end{align}
The kernel 
$$
\Lambda(C\times D\times E;\psi)=\int_{C\times D\times E}\lambda(x,l,f;\psi)
\de x\nu_{\A}(dl)\nu_{\F}(df), \quad C\times D\times E\in\B(\X\times D\times E), \psi\in N_{lf},
$$
is called the Papangelou kernel and its Radon-Nikodym derivative $\lambda$ (for fixed $\psi\in N_{lf}$) is called the {\em Papangelou conditional intensity} of $\Psi$. Heuristically, we have the following interpretation in terms of conditional infinitesimal probabilities \citep[Section 1.8.2]{VanLieshout}:
\[
\Lambda(d(x,l,f);\psi)
=
\lambda(x,l,f;\psi)\de x\nu_{\A}(dl)\nu_{\F}(df) 
= \P(\Psi(d(x,l,f))=1|\Psi\cap(d(x,l,f))^c=\psi\cap(d(x,l,f))^c)
,
\]
where $^c$ denotes complement and $d(x,l,f)$ is an infinitesimal neighbourhood of $(x,l,f)\in\X\times\A\times\F$, with measure $\de x\nu_{\M}(d(l,f))=\de x\nu_{\A}(dl)\nu_{\F}(df)$. 
It should further be mentioned that $\rho(x,l,f)=\E[\lambda((x,l,f);\Psi)]$ and, indeed, for a Poisson process the Papangelou conditional intensity is given by the intensity function. %$\rho(x,l,f)=\lambda((x,l,f);\Psi)$). 

\begin{definition}
The {\em $\nu_{\M}$-averaged Papangelou conditional intensity} with respect to $D\times E\in\B(\A\times\F)$ is defined as
$$
\lambda_{D\times E}(x;\Psi)
=
\frac{\Lambda(dx\times D\times E;\Psi)/\de x}{\nu_{\M}(D\times E)}
=
\frac{\int_{D\times E}
\lambda(x,l,f;\Psi)\nu_{\M}(d(l,f))}
{\nu_{\M}(D\times E)}
=
\frac{\int_{D\times E}
\lambda(x,l,f;\Psi)\nu_{\A}(dl)\nu_{\F}(df)}
{\nu_{\A}(D)\nu_{\F}(E)}
% =
% \frac{\int_{D\times E}
% \lambda(x,l,f;\Psi)
% P^{\A}(dl)P^{\F}(df)}
% {P^{\A}(D)P^{\F}(E)}
.
$$ 
\end{definition}

Combining \eqref{NguyenZessin} with \eqref{reducedCMMarked}, we obtain 
$$
\E^{!(x,l,f)}\left[h(x,l,f,\Psi)\right]
=\E\left[h(x,l,f,\Psi)\lambda((x,l,f);\Psi)\right]/\rho(x,l,f),
$$
whereby
\begin{align*}
\P_{D\times E}^{!x}(\Psi\in R) 
&= P_{D\times E}^{!x}(R) 
= 
\frac{\int_{D\times E}
\E\left[\1\{\Psi\in R\}\lambda(x,l,f;\Psi)\right]
/\rho(x,l,f)
%\E^{!(x,l,f)}[\1\{\Psi\in R\}] 
\nu_{\A}(dl)\nu_{\F}(df)}{\nu_{\A}(D)\nu_{\F}(E)}
\\
&=
\frac{
\E\left[\1\{\Psi\in R\}\int_{D\times E}\lambda(x,l,f;\Psi)/Q_x^{\M}(l,f)
\nu_{\A}(dl)\nu_{\F}(df)
\right]
}{\rho_G(x)\nu_{\A}(D)\nu_{\F}(E)}
\\
&=
\frac{
\int_R
\int_{D\times E}\lambda(x,l,f;\psi)/Q_x^{\M}(l,f)
\nu_{\A}(dl)\nu_{\F}(df)
P(d\psi)
}{\rho_G(x)\nu_{\A}(D)\nu_{\F}(E)}
% \\
% &=
% \frac{
% \E\left[\1\{\Psi\in R\}\lambda_{D\times E}(x;\Psi)/Q_x^{\M}(l,f)
% \right]
% }{\rho_G(x)\nu_{\A}(D)\nu_{\F}(E)}
,
\quad R\in \NN_{lf}.
\end{align*}
Moreover, when $\Psi$ has a common mark distribution $P^{\M}=P^{\A}\otimes P^{\F}$ which coincides with the mark reference measure $\nu_{\M}=\nu_{\A}\otimes\nu_{\F}$ (so that $Q_x^{\M}(l,f)\equiv1$), it follows that 
$$
\lambda_{D\times E}(x;\Psi)
=
\frac{1}{P^{\A}(D)P^{\F}(E)}
\int_{D\times E}
\lambda(x,l,f;\Psi)
P^{\A}(dl)P^{\F}(df),
$$
which is interpreted as the density of the conditional probability that $\Psi$ has a point with mark belonging to $D\times E$ in an infinitesimal neighbourhood $dx$ of $x\in\X$, given $\Psi\cap(\X\setminus dx)\times\A\times\F$. 
In addition, 
\begin{align*}
\P_{D\times E}^{!x}(\Psi\in R) 
&= 
\left.
\E\left[\1\{\Psi\in R\}
\int_{D\times E}
\lambda(x,l,f;\Psi)
P^{\A}(dl)P^{\F}(df)
\right]
\right/(\rho_G(x)P^{\A}(D)P^{\F}(E))
\\
&= 
\left.
\E\left[\1\{\Psi\in R\}
\lambda_{D\times E}(x;\Psi)
\right]
\right/\rho_G(x)
= 
\left.
\int_R
\lambda_{D\times E}(x;\psi)
P(d\psi)
\right/\rho_G(x)
.
\end{align*}
Turning to the summary statistic in \eqref{eq:K_measure}, we here obtain
\begin{align*}
&
%\prod_{i=1}^{n-1}\nu_{\M}(D_i\times E_i)
\mathcal K_t^{(D\times E)\bigtimes_{i=1}^{n-1}(D_i\times E_i)}(C_1\times\cdots\times C_{n-1})
=
% \\
% =&
% \E_{D\times E}^{!z}\Bigg[
% \mathop{\sum\nolimits\sp{\ne}}_{(x_1,l_1,f_1),\ldots,(x_{n-1},l_{n-1},f_{n-1})\in \Psi}
% t((L(z),F(z)),(l_1,f_1),\ldots,(l_{n-1},f_{n-1}))
% \times
% \\
% &\times
% \prod_{i=1}^{n-1}
% \frac{\1 \{x_i-z\in C_i\}\1\{(l_i,f_i)\in D_i\times E_i\}}{\rho(x_i,l_i,f_i)}
% \Bigg]
\\
=&
\frac{1}{\prod_{i=1}^{n-1}\nu_{\M}(D_i\times E_i)}
\E\Bigg[
\frac{
\lambda_{D\times E}(z;\Psi)
}{\rho_G(z)}
\mathop{\sum\nolimits\sp{\ne}}_{(x_1,l_1,f_1),\ldots,(x_{n-1},l_{n-1},f_{n-1})\in \Psi}
t((L(z),F(z)),(l_1,f_1),\ldots,(l_{n-1},f_{n-1}))
\times
\\
&\times
\prod_{i=1}^{n-1}
\frac{\1 \{x_i-z\in C_i\}\1\{(l_i,f_i)\in D_i\times E_i\}}{\rho_G(x_i)}
\Bigg]
\end{align*}
for almost every $z\in\X=\R^d$ by Lemma \ref{LemmaWeighted}.

When $\Psi=\{(X_i,L_i,F_i)\}_{i=1}^N$ is finite, i.e.\ $N<\infty$ a.s. (which e.g.\ is the case when $\X$ is bounded), with density $p(\cdot)$ on $N_{lf}$ with respect to the distribution on $(N_{lf},\NN_{lf})$ of a Poisson process with (non-atomic) finite intensity measure, then \citep[Theorem 1.6]{VanLieshout}
\[
\lambda(x,l,f;\psi)
=
\frac{p(\psi\cup\{(x,l,f)\})}{p(\psi)}
,\qquad \psi\in N_{lf}, 
\quad
(x,l,f)\notin\psi.
\]

\section{Proofs}
\label{s:Proofs}

\begin{proof}[Proof of Lemma \ref{LemmaWeighted}]
%Using that $\Psi$ is $k$th order marked intensity reweighted stationary, 
Applying the Campbell formula, we obtain that
\begin{align*}
&\mathcal K_t^{(D\times E)\bigtimes_{i=1}^{n-1}(D_i\times E_i)}(C_1\times\cdots\times C_{n-1})
=\\
=&
\frac{1}{|W|\nu_{\M}(D\times E)\prod_{i=1}^{n-1}\nu_{\M}(D_i\times E_i)}
\int_{W\times D\times E}
\de x \nu_{\A}(dl)\nu_{\F}(df)
\times
\\
&\times
\int_{(x+C_1)\times D_1\times E_1}\cdots\int_{(x+C_{n-1})\times D_{n-1}\times E_{n-1}}
t((l,f),(l_1,f_1),\ldots,(l_{n-1},f_{n-1}))
\times
\nonumber
\\
&\times
g_{\Psi}^{(n)}((x,l,f),(x_1,l_1,f_1),\ldots,(x_{n-1},l_{n-1},f_{n-1}))
\prod_{i=1}^{n-1}
\de x_i\nu_{\A}(dl_i)\nu_{\F}(df_i)
.
\end{align*}
At the same time, using the Campbell-Mecke formula
\begin{align*}
&\mathcal K_t^{(D\times E)\bigtimes_{i=1}^{n-1}(D_i\times E_i)}(C_1\times\cdots\times C_{n-1})
=
\frac{1}{|W|\nu_{\M}(D\times E)\prod_{i=1}^{n-1}\nu_{\M}(D_i\times E_i)}
\times
\\
&\times
\int_{W\times D\times E}
\E^{!(x,l,f)}\Bigg[
\mathop{\sum\nolimits\sp{\ne}}_{(x_1,l_1,f_1),\ldots,(x_{n-1},l_{n-1},f_{n-1})\in \Psi}
t((l,f),(l_1,f_1),\ldots,(l_{n-1},f_{n-1}))
\times
\nonumber
\\
&\times
\prod_{i=1}^{n-1}
\frac{\1 \{x_i-x\in C_i\}\1\{(l_i,f_i)\in D_i\times E_i\}}{\rho(x_i,l_i,f_i)}
\Bigg]
\de x
\nu_{\A}(dl)\nu_{\F}(df).
\end{align*}
 Hence, since we may choose $W$ to be any bounded Borel set in $\R^d$, 
\begin{align*}
&\E^{!(x,l,f)}\Bigg[
\mathop{\sum\nolimits\sp{\ne}}_{(x_1,l_1,f_1),\ldots,(x_{n-1},l_{n-1},f_{n-1})\in \Psi}
t((l,f),(l_1,f_1),\ldots,(l_{n-1},f_{n-1}))
\times
\nonumber
\\
&\times
\prod_{i=1}^{n-1}
\frac{\1 \{x_i-x\in C_i\}\1\{(l_i,f_i)\in D_i\times E_i\}}{\rho(x_i,l_i,f_i)}
\Bigg]
\\
\stackrel{a.e.}{=}
&
\int_{(x+C_1)\times D_1\times E_1}\cdots\int_{(x+C_{n-1})\times D_{n-1}\times E_{n-1}}
t((l,f),(l_1,f_1),\ldots,(l_{n-1},f_{n-1}))
\times
\nonumber
\\
&\times
g_{\Psi}^{(n)}((x,l,f),(x_1,l_1,f_1),\ldots,(x_{n-1},l_{n-1},f_{n-1}))
\prod_{i=1}^{n-1}
\de x_i\nu_{\A}(dl_i)\nu_{\F}(df_i)
\\
=&
\int_{(x+C_1)\times D_1\times E_1}\cdots\int_{(x+C_{n-1})\times D_{n-1}\times E_{n-1}}
t((l,f),(l_1,f_1),\ldots,(l_{n-1},f_{n-1}))
\times
\nonumber
\\
&\times
g_{\Psi}^{(n)}((0,l,f),(x_1-x,l_1,f_1),\ldots,(x_{n-1}-x,l_{n-1},f_{n-1}))
\prod_{i=1}^{n-1}
\de x_i\nu_{\A}(dl_i)\nu_{\F}(df_i)
\\
=&
\int_{C_1\times D_1\times E_1}\cdots\int_{C_{n-1}\times D_{n-1}\times E_{n-1}}
t((l,f),(l_1,f_1),\ldots,(l_{n-1},f_{n-1}))
\times
\nonumber
\\
&\times
g_{\Psi}^{(n)}((0,l,f),(x_1,l_1,f_1),\ldots,(x_{n-1},l_{n-1},f_{n-1}))
\prod_{i=1}^{n-1}
\de x_i\nu_{\A}(dl_i)\nu_{\F}(df_i)
\\
=&
\int_{C_1\times\cdots\times C_{n-1}}
\kappa_t^{(D\times E)\bigtimes_{i=1}^{n-1}(D_i\times E_i)}
(0,x_1,\ldots,x_{n-1})
\mathcal K_G(dx_1\times\cdots\times dx_{n-1})
,
\end{align*}
where the second equality follows by the imposed $k$th order marked intensity reweighted stationarity of $\Psi$. 
In other words, the reduced Palm expectation above is a.e.\ constant as a function of $x\in\R^d$, and we have
\begin{align*}
&
\nu_{\M}(D\times E)\prod_{i=1}^{n-1}\nu_{\M}(D_i\times E_i)
\mathcal K_t^{(D\times E)\bigtimes_{i=1}^{n-1}(D_i\times E_i)}(C_1\times\cdots\times C_{n-1})
=
\\
=&
\int_{D\times E}
\E^{!(z,l,f)}\Bigg[
\mathop{\sum\nolimits\sp{\ne}}_{(x_1,l_1,f_1),\ldots,(x_{n-1},l_{n-1},f_{n-1})\in \Psi}
t((l,f),(l_1,f_1),\ldots,(l_{n-1},f_{n-1}))
\times
\\
&\times
\prod_{i=1}^{n-1}
\frac{\1 \{x_i-z\in C_i\}\1\{(l_i,f_i)\in D_i\times E_i\}}{\rho(x_i,l_i,f_i)}
\Bigg]
\nu_{\A}(dl)\nu_{\F}(df)
\\
=&
\int_{C_1\times\cdots\times C_{n-1}}
\kappa_t^{(D\times E)\bigtimes_{i=1}^{n-1}(D_i\times E_i)}
(0,x_1,\ldots,x_{n-1})
\mathcal K_G(dx_1\times\cdots\times dx_{n-1})
,
\end{align*}
where $z$ is an arbitrary location in $\R^d$ and the last integrand is defined in \eqref{e:MarkDependenceFunction}.

\end{proof}

\begin{proof}[Proof of Theorem \ref{LemmaUnbiased}]

By the Campbell formula,
\begin{align*}
&
\E\Bigg[
\sum_{(x,l,f)\in\Psi\cap W\times D\times E}
\mathop{\sum\nolimits\sp{\ne}}_{(x_1,l_1,f_1),\ldots,(x_{n-1},l_{n-1},f_{n-1})\in \Psi\setminus\{(x,l,f)\}}
t((l,f),(l_1,f_1),\ldots,(l_{n-1},f_{n-1}))
\times
\\
&
\times
\frac{1}{\rho(x,l,f)}
\prod_{i=1}^{n-1}
\frac{\1 \{x_i\in (W\cap(x + C_i))\}\1\{(l_i,f_i)\in D_i\times E_i\}}{\rho(x_i,l_i,f_i)}
w(x,x_1,\ldots,x_{n-1})
\Bigg]
\\
&=
\int_{D\times E}
\int_{D_1\times E_1\times\cdots\times D_{n-1}\times E_{n-1}}
t((l,f),(l_1,f_1),\ldots,(l_{n-1},f_{n-1}))
\times
\\
&\times
\Bigg(
\int_{W}
\int_{(W\cap(x + C_1))\times\cdots\times (W\cap(x + C_{n-1}))}
w(x,x_1,\ldots,x_{n-1})
\times
\\
&\times
g_{\Psi}^{(n)}((x,l,f),(x_1,l_1,f_1),\ldots,(x_{n-1},l_{n-1},f_{n-1}))
\prod_{i=1}^{n-1}
\de x_{i}
\de x
\Bigg)
\prod_{i=1}^{n-1}
\nu_{\M}(d(l_i,f_i))
\nu_{\M}(d(l,f))
\end{align*}
and by the imposed $k$-MIRS and Fubini's theorem the inner expression satisfies
\begin{align*}
&\int_{W}
\int_{(W\cap(x + C_1))\times\cdots\times (W\cap(x + C_{n-1}))}
w(x,x_1,\ldots,x_{n-1})
\times
\\
&\times
g_{\Psi}^{(n)}((x,l,f),(x_1,l_1,f_1),\ldots,(x_{n-1},l_{n-1},f_{n-1}))
\prod_{i=1}^{n-1}
\de x_{i}
\de x
\\
&\stackrel{k-{\text{MIRS}}}{=}
\int_{W}
%\int_{(x + C_1)\times\cdots\times(x + C_{n-1})}
\int_{\R^d}\cdots\int_{\R^d}
\prod_{i=1}^{n-1}\1\{x_i\in(W\cap(x + C_i))\}
w(x,x_1,\ldots,x_{n-1})
\times\\
&\times
g_{\Psi}^{(n)}((0,l,f),(x_1-x,l_1,f_1),\ldots,(x_{n-1}-x,l_{n-1},f_{n-1}))
\prod_{i=1}^{n-1}
\de x_{i}
\de x
\\
&\stackrel{u_i=x_i-x}{=}
% \int_{(x + C_1)\times\cdots\times(x + C_{n-1})}
\int_{\R^d}\cdots\int_{\R^d}
\int_{W}
\prod_{i=1}^{n-1}\1\{u_i+x\in(W\cap(x + C_i))\}
w(x,u_1+x,\ldots,u_{n-1}+x)
\de x
\times\\
&\times
g_{\Psi}^{(n)}((0,l,f),(u_1,l_1,f_1),\ldots,(u_{n-1},l_{n-1},f_{n-1}))
\prod_{i=1}^{n-1}
\de u_{i}
\\
&=
\int_{\R^d}\cdots\int_{\R^d} 
\prod_{i=1}^{n-1}\1\{u_i\in C_i\}
\int_{W}
\prod_{i=1}^{n-1}\1\{(u_i+x)\in W\}
w(x,u_1+x,\ldots,u_{n-1}+x)
\de x
\times\\
&\times
g_{\Psi}^{(n)}((0,l,f),(u_1,l_1,f_1),\ldots,(u_{n-1},l_{n-1},f_{n-1}))
\prod_{i=1}^{n-1}
\de u_{i}
\\
&=
\int_{C_1\times\cdots\times C_{n-1}}
g_{\Psi}^{(n)}((0,l,f),(u_1,l_1,f_1),\ldots,(u_{n-1},l_{n-1},f_{n-1}))
\prod_{i=1}^{n-1}
\de u_{i}
\end{align*}
since
\[
\int_{W}
\prod_{i=1}^{n-1}\1\{(u_i+x)\in W\}
w(x,u_1+x,\ldots,u_{n-1}+x)
\de x
=
1
\]
for almost any $u_i\in C_i$, $i=1,\ldots,n-1$. Hence, by Fubini's theorem and Lemma \ref{LemmaWeighted} the initial expectation is given by
\begin{align*}
&
\int_{D\times E}
\int_{D_1\times E_1\times\cdots\times D_1\times E_1}
t((l,f),(l_1,f_1),\ldots,(l_{n-1},f_{n-1}))
\times\\
&\times
\int_{C_1\times\cdots\times C_{n-1}}
g_{\Psi}^{(n)}((0,l,f),(u_1,l_1,f_1),\ldots,(u_{n-1},l_{n-1},f_{n-1}))
\prod_{i=1}^{n-1}
\de u_{i}
\prod_{i=1}^{n-1}
\nu_{\M}(d(l_i,f_i))
\nu_{\M}(d(l,f))
\\
&=
\int_{C_1\times\cdots\times  C_{n-1}}
\kappa_t^{(D\times E)\bigtimes_{i=1}^{n-1}(D_i\times E_i)}
(0,u_1,\ldots,u_{n-1})
\mathcal K_G(du_1\times\cdots\times du_{n-1})
\\
&=
\nu_{\M}(D\times E)
\prod_{i=1}^{n-1}\nu_{\M}(D_i\times E_i)
\mathcal K_t^{(D\times E)\bigtimes_{i=1}^{n-1}(D_i\times E_i)}(C_1\times\cdots\times C_{n-1})
.
\end{align*}

\end{proof}

\begin{proof}[Proof of Corollary \ref{CorUnbiased}]

Since $x_i\in C_i$ we have that $\{(x_i+x)\in W\}=\{x\in (W-x_i)\}\supset\{x\in\bigcap_{u\in C_i}(W-u)\}=\{x\in W\ominus C_i\}$ by the definition of Minkowski subtraction, so $\{x\in \bigcap_{i=1}^{n-1}W\ominus C_i\}\subset \{x\in\bigcap_{i=1}^{n-1}(W-x_i)\}$ and $\1_{\bigcap_{i=1}^{n-1}W\ominus C_i}(x) \leq \1_{\bigcap_{i=1}^{n-1}(W-x_i)}(x)$, $x\in W$. Hence, 
\begin{align*}
&\int_{W}
\prod_{i=1}^{n-1}\1\{(x_i+x)\in W\}
w_{\ominus}(x,x_1+x,\ldots,x_{n-1}+x)
\de x
=\\
=&
\frac{
\int_{W}
\1\left\{x\in\bigcap_{i=1}^{n-1}(W-x_i)\right\}
\1\left\{x\in \bigcap_{i=1}^{n-1}W\ominus C_i\right\}
\de x
}{\left|\bigcap_{i=1}^{n-1}W\ominus C_i\right|}
=
\frac{
\int_{W}
\1\left\{x\in \bigcap_{i=1}^{n-1}W\ominus C_i\right\}
\de x
}{\left|\bigcap_{i=1}^{n-1}W\ominus C_i\right|}
=1.
\end{align*}
Furthermore, 
\begin{align*}
&\int_W
w_{\cap}(x,x+x_1,\ldots,x+x_{n-1})
%{\bf 1}\big[x\in \bigcap_{i=1}^{n-1} W_{-x_i}\big]
\1\left\{
x\in \bigcap_{i=1}^{n-1} (W-x_i)
\right\}
\de x 
= 
\int_W
\frac{
\1\{
x\in \bigcap_{i=1}^{n-1} (W-x_i)
\}
%\1 \big[x\in \bigcap_{i=1}^{n-1} W_{-x_i}\big]
}{|\bigcap_{i=1}^{n-1}(W+(x+x_i))\cap (W+x)|}
\de x
=
\\
&
=
\frac{\int\1\{x\in \bigcap_{i=1}^{n-1} (W-x_i)
\cap W\}\de x}
{|\bigcap_{i=1}^{n-1}(W+x_i)\cap W|}
=
\frac{|\bigcap_{i=1}^{n-1} (W-x_i)
\cap W|}
{|\bigcap_{i=1}^{n-1}(W+x_i)\cap W|}
=1
\end{align*}
since 
$x\mapsto|\bigcap_{i=1}^{n-1}
(W+(x+u_i))\cap (W+x)|=|\bigcap_{i=1}^{n-1}
(W+u_i)\cap W|$, $x\in W$, and 
$|W\cap (W-u)|=|W\cap (W+u)|$ for any $u$ \citep[Section 4.3.2]{Moller}.

Turning to the isotropic correction, we give  the details for $d=2$ here and we refer the reader to \citet{MollerSyversveen,schladitz:Baddeley:00} for $d=3$;
\begin{align*}
 \int_W
w_{\partial}(x,x+x_1){\bf 1}\{x\in  W-x_1\}
\de x 
&= 
\int_{W\cap W-{x_1}}
\frac{\ell\big(\partial b(x, \|x_1\|\big)}{\ell\big(\partial b(x,\|x_1\|)\cap W\big)}
\de x\\
&=
\int_{W\cap W-{x_1}}
\frac{ 2\pi\|x_1\|}{\ell\big(\partial b(x,\|x_1\|)\cap W\big)}
\de x 
=1,
\end{align*}
where the last equality is obtained by using polar coordinates.

\end{proof}

\putbib
\end{bibunit}

\end{document}